\documentclass[12pt,reqno,a4paper]{article}

\usepackage{fullpage}
\usepackage{setspace}
\usepackage{amssymb}
\usepackage{amsmath}
\usepackage{amsthm}
\usepackage{subcaption}
\usepackage[utf8]{inputenc}
\usepackage[T1]{fontenc}
\usepackage{ae}
\usepackage{aecompl}
\usepackage{verbatim}
\usepackage[usenames,dvipsnames]{color}
\usepackage{graphicx}
\usepackage{tikz}
\usetikzlibrary{arrows}
\usetikzlibrary{calc}
\usetikzlibrary{external}

\usepackage{kbordermatrix}
\usepackage{array}
\usepackage{bm}
\usepackage{stmaryrd}
\usepackage{enumitem}
\usepackage{mathtools}

\renewcommand{\kbldelim}{(}
\renewcommand{\kbrdelim}{)}

\newtheoremstyle{plain}
  {\medskipamount}
  {\smallskipamount}
  {\slshape}
  {0pt}
  {\bfseries}
  {.}
  { }
  {\thmname{#1}\thmnumber{ #2}{\normalfont\thmnote{ (#3)}}}

\newtheorem{theorem}             {Theorem}[section]
\newtheorem{lemma}     [theorem] {Lemma}

\newtheorem{definition}[theorem] {Definition}   
\newtheorem{proposition}[theorem]{Proposition}   

\newtheorem{remark}[theorem]	 {Remark}

\newtheorem{problem}   [theorem] {Problem}   
\newtheorem{claim}[theorem]      {Claim}

\newtheorem*{theorem*}{Theorem}
\newtheorem*{lemma*}{Lemma}

\def\o{\mathop{\text{\rm O}}\nolimits}

\def\sup{\mathop{\text{\rm sup}}\nolimits}
\def\dist{\mathop{\text{\rm dist}}\nolimits}

\def\im{\mathop{\text{\rm im}}\nolimits}

\let\:=\colon

\def\bbn{{\mathbb N}}

\def\calf{{\mathcal F}}

\def\calt{{\mathcal T}}

\let\oldepsilon\epsilon
\let\epsilon\varepsilon
\def\sup{\mathop{\text{\rm sup}}\nolimits}
\def\dist{\mathop{\text{\rm dist}}\nolimits}

\def\im{\mathop{\text{\rm im}}\nolimits}

\usepackage[%
  pdftex,%
  plainpages=false,%
  pdfpagelabels,%
  colorlinks=false,%
  bookmarksopen,%
  hidelinks%
]{hyperref}
\usepackage[all]{hypcap}

\newcommand{\tsete}{T_5^7}
\newcommand{\toito}{T_5^8}
\newcommand{\tnove}{T_5^9}
\newcommand{\tdez}{T_5^{10}}
\newcommand{\tonze}{T_5^{11}}
\newcommand{\tdoze}{T_5^{12}}

\newcommand{\trianglen}[1][n]{\vec C^3_{#1}}
\newcommand{\set}[2][]{#1\{#2#1\}}

\newcommand{\ceil}[2][]{#1\lceil#2#1\rceil}
\newcommand{\floor}[2][]{#1\lfloor#2#1\rfloor}

\DeclareMathOperator{\Tr}{Tr}
\DeclareMathOperator{\id}{id}
\DeclareMathOperator{\Hom}{Hom}
\DeclareMathOperator{\Aut}{Aut}
\newcommand{\comp}{\mathbin{\circ}}
\newcommand{\prob}[1]{\mathbb{P}\left[#1\right]}
\newcommand{\expect}[1]{\mathbb{E}\left[#1\right]}
\newcommand{\Homp}[1]{\Hom^+(\mathcal{A}^{#1},\mathbb{R})}
\newcommand{\Csem}[1]{\mathcal{C}_{\text{\rm sem}}(\mathcal{F}^{#1})}
\newcommand{\cC}[1]{\mathcal{C}(\mathcal{F}^{#1})}
\newcommand{\phiqr}{\phi_{\operatorname{qr}}}
\newcommand{\phiR}{\phi_{\operatorname{R}}}
\newcommand{\phiC}{\phi_{\vec C_3}}
\newcommand{\as}{\;\;\text{a.s.}}
\newcommand{\tind}{\mathop{t_{\operatorname{ind}}}}

\title{On the maximum density of fixed strongly connected subtournaments}

\begin{document}

\author{%
  Leonardo N.~Coregliano%
  \thanks{Institute of Mathematics and Statistics, Universidade de S\~ao Paulo,
    supported by Funda\c c\~ao de Amparo \`a Pesquisa do Estado de S\~ao Paulo
    (FAPESP) under grant no.~2013/23720-9}%
  \and%
  Roberto F.~Parente%
  \thanks{Institute of Mathematics and Statistics, Universidade de S\~ao Paulo,
	Supported by CNPq (140987/2012-6).}
  \and%
  Cristiane M.~Sato%
  \thanks{Center of Mathematics, Computing and Cognition, Universidade Federal do
    ABC. Partially supported by FAPESP (Proc.2103/03447-6).}%
}

\date{\today}


\pagestyle{plain}

\thispagestyle{empty}

\onehalfspacing

\maketitle

\begin{abstract}
  We study the density of fixed strongly connected subtournaments on~$5$ vertices in
  large tournaments. We determine the maximum density asymptotically for five
  tournaments as well as unique extremal sequences for each tournament. As a
  byproduct we also characterize tournaments that are recursive blow-ups of
  a~$3$-cycle as tournaments that avoid three specific tournaments of size~$5$.
\end{abstract}

\section{Introduction}

A \emph{locally transitive} tournament is a tournament~$T$ such that the
outneighbourhood~$N^+(v)=\{w\in V(T):vw\in A(T)\}$ and the
inneighbourhood~$N^-(v)=\{w\in V(T):wv\in A(T)\}$ of every vertex~$v\in V(T)$ are both
transitive. Alternatively a locally transitive tournament is a tournament that has no
occurrences of~$W_4$ nor of~$L_4$, where~$W_4$ and~$L_4$ are the tournaments of
size~$4$ with outdegree sequences~$(1,1,1,3)$ and~$(0,2,2,2)$ respectively. On the
other hand, a \emph{balanced} tournament is a tournament with an odd number of
vertices~$2n+1$ where each vertex has outdegree~$n$. With these definitions, there is
only one locally transitive balanced tournament~$R_{2n+1}$ of order~$2n+1$ up to
isomorphism called \emph{carousel tournament} (see Figure~\ref{fig:R}). This
tournament is defined
by~$V(R_{2n+1})=\mathbb{Z}_{2n+1} = \{0,1,\ldots,2n\}$ and
\begin{align*}
  A(R_{2n+1}) & = \{(v,(v+i)\bmod (2n+1)) : v\in V(R_{2n+1})\land i\in[n]\},
\end{align*}
where~$[n]=\{1,2,\ldots,n\}$.

\begin{figure}[ht]
  \begin{center}
    \tikzsetnextfilename{R}
\begin{tikzpicture}
\begingroup
\def\radius{2cm}
\def\shorten{0.2cm}
\def\step{5cm}
\foreach \n in {2,3,4}{
  \pgfmathsetmacro{\tn}{2*\n}
  \pgfmathsetmacro{\N}{\tn+1}
  \foreach \i in {0,...,\tn}{
    \pgfmathsetmacro{\angle}{\i*360/\N+90}
    \filldraw ($(\angle:\radius) + (\n*\step,0cm)$) circle (2pt);

    \foreach \j in {1,...,\n}{
      \pgfmathsetmacro{\angletwo}{\angle-\j*360/\N}
      \draw[shorten >=\shorten,shorten <=\shorten,arrows={-latex}]
      ($(\angle:\radius) + (\n*\step,0cm)$) --
      ($(\angletwo:\radius) + (\n*\step,0cm)$);
    }

    \node[below] at (\n*\step,-\radius) {$R_{\pgfmathprintnumber{\N}}$};
  }
}

\endgroup
\end{tikzpicture}
    \caption{Carousel tournaments~$R_{2n+1}$ for~$n=2,3,4$.}
    \label{fig:R}
  \end{center}
\end{figure}
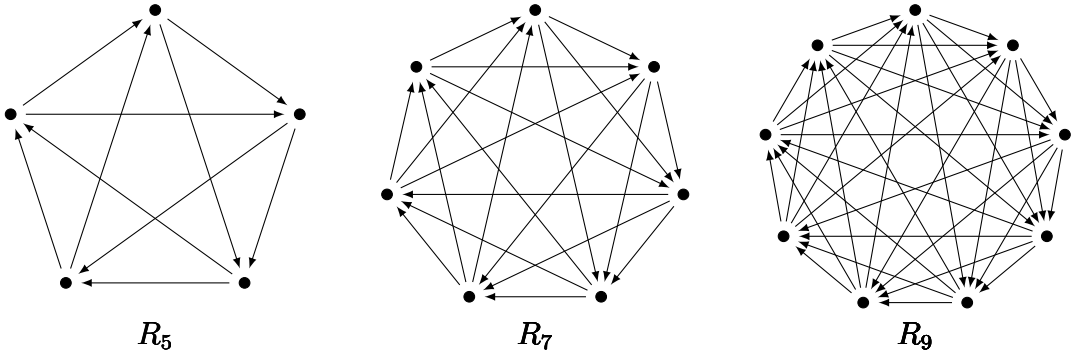

In~1964, Colombo~\cite{C:SuiCircuitiNeiGrafiCompleti} proved that the maximum number
of oriented~$4$-cycles in a tournament of odd size is attained by the carousel
tournament~$R_{2n+1}$. One year later, Beineke and~Harary~\cite{BeHa65} extended this
result by proving that the carousel tournament~$R_{2n+1}$ also maximizes the number
of strongly connected subtournaments of a fixed size in a tournament of odd
size\footnote{Both results also provided a maximizer for even size, but we refrain
  from defining it here.}.

The problem of maximizing subtournaments in a tournament of fixed size is in general
a hard problem and determining which tournaments are extremal is an even harder
one. However, in many cases, the easier problem of maximizing the asymptotic density
is completely solvable, that is, not only can we find the value of the maximum
asymptotic density, but we can also characterize all extremal families of
tournaments.

Such characterization of an asymptotically extremal family~$(T_n)_{n\in\mathbb{N}}$
generally comes in the flavour of saying that~$(T_n)_{n\in\mathbb{N}}$ converges to a
certain ``limit object'', that is, for every fixed tournament~$T$, we have
\begin{align*}
  \lim_{n\to\infty}p(T;T_n) & = \phi(T),
\end{align*}
where~$p(T;T_n)$ denotes the unlabelled density of~$T$ in~$T_n$ (which is a number
in~$[0,1]$)\footnote{Of course that if~$(T_n)_{n\in\mathbb{N}}$ maximizes the
  asymptotic density of~$T$, then we will know the value
  of~$\lim_{n\to\infty}p(T;T_n)$, but the strength of this characterization is that
  extremality for~$T$ forces the densities of other tournaments~$T'\neq T$ to
  converge to specific values.}.

There are basically two approaches to defining what is a ``limit object''. The first
approach is to define the limit object to be semantically close to the underlying
object, which has been carried out successfully for several combinatorial objects
such as graphs~\cite{LoSz06}, uniform hypergraphs~\cite{ES:MeasureTheoreticApproach},
digraphs~\cite[Section~9]{DJ:GraphLimitsAndExchangeableRandomGraphs}\footnote{As
  expected, the limit object for a tournament is just a special case of the limit
  object for a digraph.} and
permutations~\cite{HKMRS:LimitsOfPermutationSequences}. The second approach is to
define the limit object syntactically, that is, to study what sorts of properties
must a sequence~$(\phi(T))_T$ satisfy if it is obtained as~$\forall
T,\lim_{n\to\infty}p(T;T_n)=\phi(T)$ for a sequence of
objects~$(T_n)_{n\in\mathbb{N}}$. This latter approach is precisely the thrust of the
theory of flag algebras~\cite{Raz07} and in what follows we will mostly use this language.

In this paper, we study the problem of maximizing the asymptotic density of a
\emph{single} fixed strongly connected tournament~$T$ of size~$5$, that is, we are
interested in computing
\begin{align*}
  \lim_{n\to\infty} \max\{p(T;T_n) : \lvert V(T_n)\rvert = n\},
\end{align*}
for~$T\in\{T_5^7,T_5^8,T_5^9,T_5^{10},T_5^{11},T_5^{12}\}$ (see
Figure~\ref{fig:strongT5})\footnote{It is easy to see that the sequence of maximum
  values has decreasing tail, hence it is convergent.} and characterizing
sequences~$(T_n)_{n\in\mathbb{N}}$ such that~$\lim_{n\to\infty}p(T;T_n)$ is equal to
this value. Note that this differs from the result of Beineke and~Harary in that they
proved that for every~$k\in\mathbb{N}$, we have
\begin{align*}
  \max\left\{\sum_{T\in\mathcal{S}_k}p(T;T_{2n+1}) :
  \lvert V(T_{2n+1})\rvert = 2n+1\right\}
  & =
  \sum_{T\in\mathcal{S}_k}p(T;R_{2n+1}),
\end{align*}
where~$\mathcal{S}_k$ denotes the set of all strongly connected tournaments of
size~$k$.

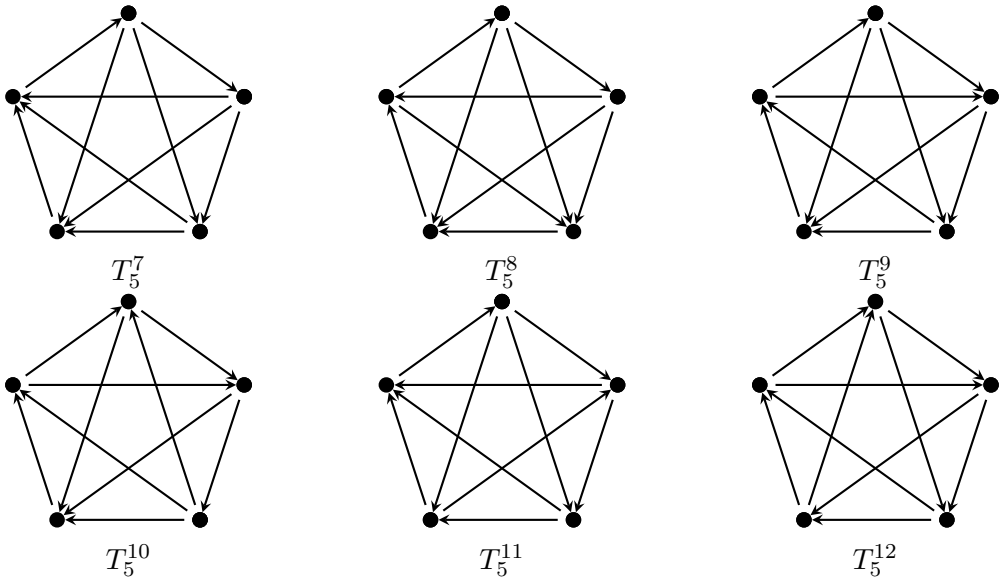
\begin{figure}[ht]
  \begin{center}
    \tikzstyle{every node}=[circle, draw, fill=black!100, inner sep=0, minimum width=5pt]
\begin{subfigure}[t]{0.3\textwidth}
  \begin{center}
    \tikzsetnextfilename{strongT7}
	\begin{tikzpicture}[thick,scale=0.8,shorten >=3pt,shorten <=3pt]
      \draw \foreach \x in {18,90,162,...,306}
            {
              (\x:2) node {}  edge[->,>=stealth] (\x-72:2)
            };

            \draw 
                {
                  (90:2) node {}  edge[->,>=stealth] (90-2*72:2)
                  (90:2) node {}  edge[->,>=stealth] (90-3*72:2)
                  (18:2) node {}  edge[->,>=stealth] (18-2*72:2)
                  (18:2) node {}  edge[->,>=stealth] (18-3*72:2)
                  (306:2) node {}  edge[->,>=stealth] (306-2*72:2)
                };
    \end{tikzpicture}
    \caption*{$\tsete$}
  \end{center}
\end{subfigure}
\begin{subfigure}[t]{0.3\textwidth}
  \begin{center}
    \tikzsetnextfilename{strongT8}
    \begin{tikzpicture}[thick,scale=0.8,shorten >=3pt,shorten <=3pt]
      \draw \foreach \x in {18,90,162,...,306}
            {
              (\x:2) node {}  edge[->,>=stealth] (\x-72:2)
            };

            \draw 
                {
                  (90:2) node {}  edge[->,>=stealth] (90-2*72:2)
                  (90:2) node {}  edge[->,>=stealth] (90-3*72:2)
                  (18:2) node {}  edge[->,>=stealth] (18-2*72:2)
                  (18:2) node {}  edge[->,>=stealth] (18-3*72:2)
                  (162:2) node {}  edge[->,>=stealth] (162+2*72:2)
                };
    \end{tikzpicture}
    \caption*{$\toito$}
  \end{center}
\end{subfigure}
\begin{subfigure}[t]{0.3\textwidth}
  \begin{center}
    \tikzsetnextfilename{strongT9}
	\begin{tikzpicture}[thick,scale=0.8,shorten >=3pt,shorten <=3pt]
	  \draw \foreach \x in {18,90,162,...,306}
	        {
		      (\x:2) node {}  edge[->,>=stealth] (\x-72:2)
	        };

	        \draw 
	            {
		          (90:2) node {}  edge[->,>=stealth] (90-2*72:2)
		          (90:2) node {}  edge[->,>=stealth] (90-3*72:2)
		          (18:2) node {}  edge[->,>=stealth] (18-2*72:2)
		          (306:2) node {}  edge[->,>=stealth] (306-2*72:2)
		          (162:2) node {}  edge[->,>=stealth] (162-2*72:2)
	            };
	\end{tikzpicture}
	\caption*{$\tnove$}
  \end{center}
\end{subfigure}
\begin{subfigure}[t]{0.3\textwidth}
  \begin{center}
    \tikzsetnextfilename{strongT10}
	\begin{tikzpicture}[thick,scale=0.8,shorten >=3pt,shorten <=3pt]
	  \draw \foreach \x in {18,90,162,...,306}
	        {
		      (\x:2) node {}  edge[->,>=stealth] (\x-72:2)
	        };

	        \draw 
	            {
		          (90:2) node {}  edge[->,>=stealth] (90-3*72:2)
		          (18:2) node {}  edge[->,>=stealth] (18-2*72:2)
		          (306:2) node {}  edge[->,>=stealth] (306-2*72:2)
		          (306:2) node {}  edge[->,>=stealth] (306-3*72:2)
		          (162:2) node {}  edge[->,>=stealth] (162-2*72:2)
	            };
	\end{tikzpicture}
	\caption*{$\tdez$}
  \end{center}
\end{subfigure}
\begin{subfigure}[t]{0.3\textwidth}
  \begin{center}
    \tikzsetnextfilename{strongT11}
	\begin{tikzpicture}[thick,scale=0.8,shorten >=3pt,shorten <=3pt]
	  \draw \foreach \x in {18,90,162,...,306}
	        {
		      (\x:2) node {}  edge[->,>=stealth] (\x-72:2)
	        };

	        \draw 
	            {
		          (90:2) node {}  edge[->,>=stealth] (90-2*72:2)
		          (90:2) node {}  edge[->,>=stealth] (90-3*72:2)
		          (18:2) node {}  edge[->,>=stealth] (18+2*72:2)
		          (306:2) node {}  edge[->,>=stealth] (306-2*72:2)
		          (234:2) node {}  edge[->,>=stealth] (234-3*72:2)
	            };
	\end{tikzpicture}
	\caption*{$\tonze$}
  \end{center}
\end{subfigure}
\begin{subfigure}[t]{0.3\textwidth}
  \begin{center}
    \tikzsetnextfilename{strongT12}
	\begin{tikzpicture}[thick,scale=0.8,shorten >=3pt,shorten <=3pt]
	  \draw \foreach \x in {18,90,162,...,306}
		    {
			  (\x:2) node {}  edge[->,>=stealth] (\x-72:2)
		    };

		    \draw \foreach \x in {18,90,162,...,306}
		          {
			        (\x:2) node {}  edge[->,>=stealth] (\x-144:2)
		          };
	\end{tikzpicture}
	\caption*{$\tdoze$}
  \end{center}
\end{subfigure}

    \caption{Strongly connected tournaments of size~$5$.}
    \label{fig:strongT5}
  \end{center}
\end{figure}

\medskip

For the asymptotic study of tournaments, we have a weaker notion of balanced
tournament. Namely a sequence of tournaments~$(T_n)_{n\in\mathbb{N}}$ of increasing
sizes is said to be \emph{asymptotically balanced} if all except for~$o(\lvert
T_n\rvert)$ vertices of~$T_n$ have outdegree~$(1/2+o(1))\lvert T_n\rvert$.

In the theory of flag algebras, the notion of a limit object is captured by a
positive homomorphism~$\phi\in\Hom^+(\mathcal{A}^0,\mathbb{R})$ (see
Section~\ref{subsec:flagbasic}) and we have the analogous notion of a \emph{balanced
  homomorphism}, which a homomorphism that is the limit of an asymptotically balanced
sequence of tournaments.

In this paper, we prove extremality theorems involving three important balanced
homomorphisms.

The first homomorphism is the \emph{quasi-random homomorphism}~$\phiqr$, which is the main
topic of study of the theory of tournament quasi-randomness. 

For every~$n\in\mathbb{N}$, let~$\bm{R_{n,1/2}}$ be the random tournament of size~$n$
where each arc orientation is present with probability~$1/2$ independently of all
other pairs of vertices. It is a straightforward exercise in binomial concentration
to prove that
\begin{align*}
  \forall T\text{ tournament},
  \lim_{n\to\infty}p(T;\bm{R_{n,1/2}}) & = \expect{p(T;\bm{R_{\lvert T\rvert,1/2}})}
  = \frac{\lvert T\rvert!}{\lvert\Aut(T)\rvert 2^{\binom{\lvert T\rvert}{2}}},
\end{align*}
that is, the sequence~$(\bm{R_{n,1/2}})_{n\in\mathbb{N}}$ of random tournaments is
convergent with probability~$1$. The quasi-random homomorphism~$\phiqr$ is then
defined as the almost sure limit of this sequence.

The theory of quasi-randomness started with the study of quasi-random graphs in the
seminal papers by Thomason~\cite{T:PseudoRandomGraphs} and Chung, Graham
and~Wilson~\cite{ChGrWi88} and now has branches in several other theories such as
uniform hypergraphs~\cite{CG:QuasiRandomHypergraphs,
  C:QuasiRandomHypergraphsRevisited, BR:NoteOnUpperDensityOfQuasiRandomHypergraphs},
graph orientations~\cite{G:QuasiRandomOrientedGraphs},
permutations~\cite{C:QuasirandomPermutations,
  KP:QuasirandomPermutationsAreCharacterized} and
tournaments~\cite{CG:QuasiRandomTournaments,KS:ANoteOnEvenCyclesAndQuasirandomTournaments,
  Na15}. We do not attempt to provide a detailed review of the theory of
quasi-randomness here (see~\cite{KS:PseudoRandomGraphs} for a survey), but its main
gist is that there are several a priori different properties of a
homomorphism~$\phi\in\Hom^+(\mathcal{A}^0,\mathbb{R})$, called quasi-random
properties, that force~$\phi=\phiqr$.

For instance, an example of a quasi-random property was proven by the first author
and Razborov~\cite{NaRaz15a}: for every~$k\geq 4$, if~$\phi$ minimizes the density of
the transitive tournament of size~$4$, then~$\phi=\phiqr$.

Regarding~$\phiqr$, we prove the following result.

\begin{theorem}\label{thm:quasirandom}
  We have
  \begin{align*}
    \lim_{n \to \infty}\max\{p(\toito;T_n) : \lvert T_n\rvert = n\} & = \frac{15}{128}.
  \end{align*}
  Furthermore, if~$(T_n)_{n\in\mathbb{N}}$ is a sequence of tournaments of increasing
  sizes, then~$\lim_{n\to\infty} p(\toito,T_n) = 15/128$ if and
  only if~$(T_n)_{n\in \bbn}$ is quasi-random, that is, if and only
  if~$(T_n)_{n\in\mathbb{N}}$ converges to~$\phiqr$.
\end{theorem}

The second homomorphism studied is the \emph{carousel homomorphism}~$\phiR$, which is
the limit of the sequence~$(R_{2n+1})_{n\in\mathbb{N}}$ of carousel
tournaments. Analogously to quasi-random properties, the quasi-carousel
properties~\cite{Na15} are properties of a
homomorphism~$\phi\in\Hom^+(\mathcal{A}^0,\mathbb{R})$ that force~$\phi=\phiR$.

In analogy with a locally transitive tournament, a
homomorphism~$\phi\in\Hom^+(\mathcal{A}^0,\mathbb{R})$ satisfying~$\phi(W_4+L_4)=0$
is called \emph{locally transitive}. Perhaps the most important quasi-carousel
property says that~$\phiR$ is the only homomorphism that is both balanced and locally
transitive.

In this paper, we prove the following result involving~$\phiR$.

\begin{theorem}\label{thm:carousel}
  We have
  \begin{align}
    \lim_{n\to\infty}\max\{p(\tsete;T_n) : \lvert T_n\rvert = n\} & = \frac{5}{16};
    \nonumber
    \\
    \lim_{n\to\infty}\max\{p(\tdoze;T_n) : \lvert T_n\rvert = n\} & = \frac{1}{16}.
    \label{eq:t12}
  \end{align}
  Furthermore, a sequence of tournaments~$(T_n)_{n\in\mathbb{N}}$ of increasing sizes
  is extremal for any of~$\tsete$ or~$\tdoze$ if and only if it is quasi-carousel,
  that is, if and only if~$(T_n)_{n\in\mathbb{N}}$ converges to~$\phiR$.
\end{theorem}

We remark that~\eqref{eq:t12} partially confirms a conjecture proposed by the first
author in~\cite{Na15}.

Finally, the last homomorphism studied is what we call here \emph{triangular
  homomorphism}~$\phiC$.

To define~$\phiC$, we must first define recursive blow-ups of the~$3$-cycle~$\vec
C_3$. For every~$n\geq 3$, let~$n_0\geq n_1\geq n_2$ be such that~$n_0+n_1+n_2 = n$
and~$n_i\in\set{\floor{n/3}, \ceil{n/3}}$ for all~$i\in\set{0,1,2}$. Define~$A_0 =
\set{1,\dotsc, n_0}$, $A_1 = \set{n_0+1,\dotsc, n_1}$ and~$A_2 = \set{n_1+1,\dotsc,
  n_2}$. Let~$\trianglen$ be the tournament on~$[n]$ such that~$vw\in A(\trianglen)$
for every~$v\in A_i$ and~$w\in A_{(i+1)\bmod 3}$, and~$\trianglen[n]\vert_{A_i}$ is
isomorphic to~$\trianglen[n_i]$ for every~$i=0,1,2$ (see Figure~\ref{fig:C3-rec}).

\begin{figure}[ht]
  \begin{center}
    \tikzsetnextfilename{C3recblowup}
\begin{tikzpicture}[scale=0.5]
  \begingroup
  \def\iterationsmo{4}
  \def\initialradius{5}
  \def\fracprop{0.35}
  \def\initialarrowscale{2}
  \def\initiallinewidth{5}

  \coordinate (P) at (0cm,0cm);

  \def\prevcenters{P}
  \foreach[%
    remember=\centers as \prevcenters,
    remember=\radius as \prevradius (initially \initialradius),
    remember=\arrowscale as \prevarrowscale (initially \initialarrowscale),
    remember=\linewidth as \prevlinewidth (initially \initiallinewidth),
    evaluate=\radius using \prevradius * \fracprop,
    evaluate=\arrowscale using \prevarrowscale * \fracprop,
    evaluate=\linewidth using \prevlinewidth * \fracprop,
    evaluate=\drawradius using 1/(1-\fracprop)*\radius%
  ] \i in {0,...,\iterationsmo}{
    \xdef\centers{\relax}
    \foreach \c in \prevcenters {
      \coordinate (\c1) at ($(\c) + (90:\prevradius)$);
      \coordinate (\c2) at ($(\c) + (210:\prevradius)$);
      \coordinate (\c3) at ($(\c) + (330:\prevradius)$);

      \coordinate (\c-12) at ($(\c1) + (240:\drawradius)$);
      \coordinate (\c-13) at ($(\c1) + (-60:\drawradius)$);
      \coordinate (\c-21) at ($(\c2) + (60:\drawradius)$);
      \coordinate (\c-23) at ($(\c2) + (0:\drawradius)$);
      \coordinate (\c-31) at ($(\c3) + (120:\drawradius)$);
      \coordinate (\c-32) at ($(\c3) + (180:\drawradius)$);

      \draw[line width=\prevlinewidth,arrows={->}]
      (\c-12) -- (\c-21);
      \draw[line width=\prevlinewidth,arrows={->}]
      (\c-23) -- (\c-32);
      \draw[line width=\prevlinewidth,arrows={->}]
      (\c-31) -- (\c-13);

      \draw (\c1) circle (\drawradius);
      \draw (\c2) circle (\drawradius);
      \draw (\c3) circle (\drawradius);

      \if\centers\relax
      \xdef\centers{\c1,\c2,\c3}
      \else
      \xdef\centers{\centers,\c1,\c2,\c3}
      \fi
    }
  }
  \endgroup
\end{tikzpicture}
    \caption{Typical structure of~$\trianglen$.}
    \label{fig:C3-rec}
  \end{center}
\end{figure}

We define then~$\phiC$ as the limit of the sequence~$(\trianglen)_{n\in\mathbb{N}}$
(see Proposition~\ref{prop:triangleconv} for the convergence of this sequence).

Our final extremality result concerns~$\phiC$.

\begin{theorem}\label{thm:triangle}
  We have
  \begin{align*}
    \lim_{n\to\infty}\max\{p(\tnove;T_n) : \lvert T_n\rvert = n\} & = \frac{3}{8};\\
    \lim_{n\to\infty}\max\{p(\tonze;T_n) : \lvert T_n\rvert = n\} & = \frac{1}{16}.
  \end{align*}
  Furthermore, a sequence of tournaments~$(T_n)_{n\in\mathbb{N}}$ of increasing sizes
  is extremal for any of~$\tnove$ or~$\tonze$ if and only if it is quasi-triangular,
  that is, if and only if~$(T_n)_{n\in\mathbb{N}}$ converges to~$\phiC$.
\end{theorem}

Motivated by the definition of~$\trianglen$, we say that a tournament is~\emph{$\vec
  C_3$-decomposable} (see Definition~\ref{def:C3dec}) if it has the same structure
of~$\trianglen$ but without requiring the parts to be as equal as possible. And, as
an auxiliary result, we prove (Theorem~\ref{thm:C3deccharac}) that a tournament
is~$\vec C_3$-decomposable if and only if it has no copies of~$T_5^8$, $T_5^{10}$ nor
of~$T_5^{12}$. We then extend this notion to homomorphisms by saying that a
homomorphism~$\phi\in\Hom^+(\mathcal{A}^0,\mathbb{R})$ is~$\vec C_3$-decomposable
if~$\phi(T_5^8 + T_5^{10} + T_5^{12})=0$.

As a byproduct of Theorem~\ref{thm:triangle}, we prove several quasi-triangular
properties (i.e., properties of a
homomorphism~$\phi\in\Hom^+(\mathcal{A}^0,\mathbb{R})$ that force~$\phi=\phiC$). One
of them, namely~\ref{it:QTbalC3dec}, says that~$\phiC$ is the only homomorphism that
is both balanced and~$\vec C_3$-decomposable.

Let us finally remark that our results fail to cover only one strongly connected
tournament of size~$5$, which is~$T_5^{10}$.

\medskip

The paper is organized as follows. We prove the lower bounds of
Theorems~\ref{thm:quasirandom}, \ref{thm:carousel} and~\ref{thm:triangle} in
Section~\ref{sec:LB}. The proof of the upper bounds are given in Section~\ref{sec:UB}
and are an application of Razborov's semidefinite method for flag
algebras~\cite{Raz10} (see also~\cite{BHLPUV:MinimumNumberOfMonotoneSubsequences,
  CKPSTY:MonochromaticTrianglesInThreeColouredGraphs,
  DHMNS:AProblemOfErdosOnTheMinimumNumberOfkCliques,
  FV:ApplicationsOfTheSemidefiniteMethod, PiVa13} for some examples). In
Section~\ref{sec:sdp}, we present a brief overview of flag algebras and this
method. The uniqueness proofs are presented in Section~\ref{sec:unique}. In
Section~\ref{sec:ext}, we show how to extract informations about extremal sequences
from the semidefinite method. We postpone the proof of the characterization of~$\vec
C_3$-decomposable tournaments to Section~\ref{sec:proofdec} and postpone the proof of
a technical lemma on quasi-triangular properties to Section~\ref{sec:prooftwononneg}.


\section{Lower bounds}
\label{sec:LB}
In this section, we prove the lower bounds in
Theorems~\ref{thm:carousel}, \ref{thm:quasirandom} and~\ref{thm:triangle}.

We start by recalling the definition of labelled density in tournaments.

\begin{definition}
  If~$T_1$ and~$T_2$ are tournaments with~$\lvert T_1\rvert\leq\lvert T_2\rvert$,
  then the labelled density of~$T_1$ in~$T_2$ (denoted~$\tind(T_1;T_2)$) is the
  probability that an injective mapping from~$V(T_1)$ to~$V(T_2)$ picked uniformly at
  random is an embedding of~$T_1$ in~$T_2$.
\end{definition}

It is easy to see that
\begin{align*}
  \tind(T_1;T_2) & = \frac{\lvert\Aut(T_1)\rvert}{\lvert T_1\rvert!}p(T_1;T_2),
\end{align*}
where~$\Aut(T_1)$ is the group of automorphisms of~$T_1$.

\begin{lemma}\label{lemma:lb} We have 
  \begin{align*}
    \lim_{n \to \infty}p(\tsete;R_{2n+1}) & = \frac{5}{16}; &
    \lim_{n \to \infty}p(\tdoze;R_{2n+1}) & = \frac{1}{16}.
  \end{align*}
\end{lemma}
\begin{proof}
  We will prove only the assertion for~$\tdoze$, since the proof for~$\tsete$ is very
  similar.


  Fix~$n\geq 2$ and let~$f\:V(R_5)\to V(R_{2n+1})$ be an embedding of~$R_5$
  in~$R_{2n+1}$.


  Suppose that the vertex~$0$ from~$R_5$ is mapped to the vertex~$0$ of~$R_n$. If
  vertex~$1$ is mapped to a vertex~$i$, then~$1\leq i\leq n$ and vertex~$2$ has to be
  mapped to a vertex~$j$ such that~$i+1\leq j\leq n$. Vertex~$3$ has to be mapped to
  a vertex~$k$ such that~$n+1\leq k\leq i+n$ (since~$(3,0)$ and~$(1,3)$ are arcs
  of~$R_5$). Finally, vertex~$4$ has to be mapped to a vertex~$\ell$ such
  that~$i+n\leq\ell\leq j+n$ (since~$(4,1)$ and~$(2,4)$ are arcs of~$R_5$). See
  Figure~\ref{f:embbed:T_5-12_to_R_n}.

  \begin{figure}[ht]
    \begin{center}
      \includegraphics[width=0.25\textwidth]{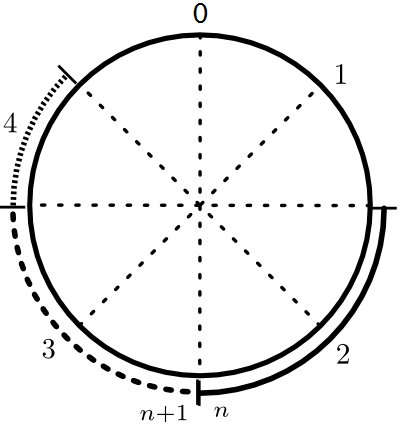}
      \caption{Possibilities of embedding~$\tdoze=R_5$ in~$R_{2n+1}$.}
      \label{f:embbed:T_5-12_to_R_n}
    \end{center}
  \end{figure}

  Note that, after we fix the images of the vertices~$0$, $1$ and~$2$, the number of
  choices for the vertex~$3$ becomes~$i$ and for the vertex~$4$ becomes~$j-i+1$.

  From the symmetry of~$R_{2n+1}$, we know that this is also the case for every other
  choice of the image of the vertex~$0$ of~$R_5$. Thus, we
  have
  \begin{align*}
    \tind(R_5;R_{2n+1})
    & = \frac{1}{(2n+1)_5}\cdot (2n+1)\sum_{i=1}^n\sum_{j=i+1}^n i(j-i+1)\\
    & = \frac{1}{(2n)_4}\sum_{i=1}^n\frac{i^3 - (2n-2)i^2 + n(n-2)i}{2}\\
    & =
    \frac{1}{(2n)_4}  \left(\frac{n^4}{24} + O(n^3) \right)\\
	& = \frac{1}{3\cdot 2^7} + O\left(\frac{1}{n}\right),
  \end{align*}
  where~$(\ell)_k=\ell(\ell-1)\cdots(\ell-k+1)$ denotes the falling factorial.
	
  Therefore
  \begin{align*}
	\lim_{n\to\infty}p(\tdoze;R_{2n+1}) & =
    \lim_{n\to\infty}\frac{5!}{5}\tind(R_5; R_n) =
    \frac{1}{16}.
    \qedhere
  \end{align*}
\end{proof}

We now prove the lower bounds in Theorem~\ref{thm:triangle}.

\begin{lemma}\label{lemma:lb_triangle} We have
  \begin{align*}
    \lim_{n \to \infty}p(\tnove;\trianglen[3^n]) & = \frac{3}{8}; &
    \lim_{n \to \infty}p(\tonze; \trianglen[3^n]) = \frac{1}{16}.
  \end{align*}
\end{lemma}
\begin{proof}
  Again, we will prove only the assertion for~$\tnove$, since the proof for~$\tonze$
  is very similar.

  Let~$T$ denote the tournament in Figure~\ref{f:t9label}, which is isomorphic
  to~$\tnove$.
  \begin{figure}[ht]
    \centering    
    \tikzsetnextfilename{T9}
\tikzstyle{every node}=[circle, draw, fill=white, inner sep=1, minimum width=5pt]
	\begin{tikzpicture}[thick,scale=0.8,shorten >=8pt,shorten <=5pt]
	    \draw \foreach \x in {18,90,162,...,306}
	    {
		(\x:2) node {}  edge[->,>=stealth] (\x-72:2)
	    };

	    \draw 
	    {
		(90:2) node {}  edge[->,>=stealth] (90-2*72:2)
		(90:2) node {}  edge[->,>=stealth] (90-3*72:2)
		(18:2) node {}  edge[->,>=stealth] (18-2*72:2)
		(306:2) node {}  edge[->,>=stealth] (306-2*72:2)
		(162:2) node {}  edge[->,>=stealth] (162-2*72:2)
	    };
        \draw
        \foreach \x [count=\xi] in {90,378,306,234,162}
                 {
                   (\x:2) node {\xi} 
                 };
	\end{tikzpicture}
    \caption{Tournament isomorphic to~$\tnove$.}
    \label{f:t9label}
  \end{figure}
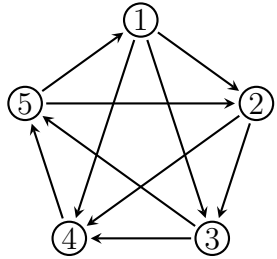

  Recall the definition of~$\trianglen[3^n]$ and let~$A_0=[3^{n-1}]$,
  $A_1=\{3^{n-1}+1,3^{n-1}+2,\ldots,2\cdot 3^{n-1}\}$ and~$A_2=\{2\cdot
  3^{n-1}+1,2\cdot 3^{n-1}+2,\ldots,3^n\}$.

  Let~$F(n)$ be the number of embeddings of~$T$ in~$\trianglen[3^n]$. Every such
  embedding either maps all vertices of~$T$ to a single~$A_i$ or it maps~$1$ and~$5$
  to some part~$A_i$, $3$ and~$4$ to~$A_{(i+1)\bmod 3}$ and~$2$
  to~$A_{(i+2)\bmod3}$. Thus we have~$F(1) = 0$ and, for every~$n\geq 2$, we have
  \begin{align*}
    F(n)
    & =
    3\binom{3^{n-1}}{2}^23^{n-1} + 3F(n-1)
    \leq
    \frac{3^{5n-4}}{4} + 3F(n-1).
  \end{align*}

  Therefore, it follows that
  \begin{align*}
    F(n)
    & \leq
    \sum_{i=0}^{n-1}3^i\frac{3^{5(n-i)-4}}{4}
    =
    \frac{3^{5n-4}}{4}\cdot\frac{1-3^{-4n}}{1-3^{-4}}
    =
    \frac{3^{5n}}{320} + O(3^{4n}).
  \end{align*}

  On the other hand, we have
  \begin{align*}
    F(n)
    & \geq
    \frac{3^{5n-4} - 3^{4n-2}}{4} + 3F(n-1),
  \end{align*}
  hence
  \begin{align*}
    F(n)
    & \geq
    \frac{3^{5n}}{320}
    - \sum_{i=0}^{n-1}3^i\frac{3^{2(n-i)-1}}{4}
    - O(3^{4n})
    =
    \frac{3^{5n}}{320} - O(3^{4n}).
  \end{align*}

  Therefore
  \begin{align*}
    \lim_{n\to\infty}p(\tnove;\trianglen[3^n]) & = 
    \lim_{n\to\infty}\frac{F(n)}{(3^n)_5}\cdot 5! =
    \frac{3}{8}.
    \qedhere
  \end{align*}
\end{proof}

Finally, we prove the lower bound for~$\toito$ in Theorem~\ref{thm:quasirandom}.

\begin{lemma}\label{lemma:lb_qr} We have
  \begin{align*}
    \lim_{n \to \infty}\expect{p(\toito;\bm{R_{n,1/2}})} & = \frac{15}{128}.
  \end{align*}
\end{lemma}
\begin{proof}
  From the definition of~$\bm{R_{n,1/2}}$, it follows that
  \begin{align*}
    \expect{\tind(\toito;\bm{R_{n,1/2}})} & = \frac{1}{2^{10}},
  \end{align*}
  for every~$n\geq 5$, hence
  \begin{align*}
    \lim_{n \to \infty}\expect{p(\toito;\bm{R_{n,1/2}})}
    & = 
    \frac{1}{2^{10}}\cdot 5!
    =
    \frac{15}{128}.
    \qedhere
  \end{align*}
\end{proof}



		    
\section{Razborov's semidefinite method for flag algebras}
\label{sec:sdp}

In this section, we briefly review the basics of the flag algebra theory and its
semidefinite method. Although we work here only with the theory of tournaments, we
remark that flag algebras can be defined in the general setting of any universal
theory of first-order (see~\cite{Raz07} and~\cite{Raz10}, see also the
surveys~\cite{R:FlagAlgebrasInterim} and~\cite{R:WhatIsAFlagAlgebra}).

\subsection{Basic definitions and properties}
\label{subsec:flagbasic}

First recall the definition of~$\mathcal{T}_n$ as the set of all tournaments of
on~$n$ vertices up to isomorphism and
define~$\mathcal{T}=\bigcup_{n\in\mathbb{N}}\mathcal{T}_n$ as the set of all
tournaments up to isomorphism on a finite number of vertices. For every
tournament~$T$, we will denote its \emph{size} by~$\lvert T\rvert = \lvert
V(T)\rvert$.

A \emph{type} is a tournament with vertex set~$[k]=\{1,2,\ldots,k\}$ for
some~$k\in\mathbb{N}$ and, given a type~$\sigma$ of size~$k=\lvert\sigma\rvert$,
a~\emph{$\sigma$-flag} is a partially labelled tournament such that the labelled part
is a copy of~$\sigma$. Formally, a~$\sigma$-flag is a pair~$(T,\theta)$, where~$T$ is
a tournament and~$\theta\:[k]\to V(T)$ is an embedding of~$\sigma$ into~$T$, that is,
the function~$\theta$ is an isomorphism between~$\sigma$ and the tournament induced
by~$\im(\theta)$ on~$T$. We define the \emph{size} of
the~$\sigma$-flag~$F=(T,\theta)$ as~$\lvert F\rvert=\lvert T\rvert$.

We extend the notion of isomorphism to~$\sigma$-flags declaring that a
function~$f\:V(T_1)\to V(T_2)$ is an \emph{isomorphism} between the
flags~$F_1=(T_1,\theta_1)$ and~$F_2=(T_2,\theta_2)$ if it is an isomorphism
between~$T_1$ and~$T_2$ and~$f\comp\theta_1 = \theta_2$ (i.e., the function~$f$
preserves labels). Naturally, we say that two flags~$F_1$ and~$F_2$ are
\emph{isomorphic} (denoted~$F_1\cong F_2$) if there exists an isomorphism between
them.

This allows us to define~$\mathcal{F}^\sigma_n$ as the set of all~$\sigma$-flags of
size~$n$ up to isomorphism
and~$\mathcal{F}^\sigma=\bigcup_{n\in\mathbb{N}}\mathcal{F}^\sigma_n$ as the set of
all finite~$\sigma$-flags up to isomorphism.

Let us denote the unique type of size~$0$ by~$0$ and note that a~$0$-flag can be
identified with a tournament. Let us also note that for every type~$\sigma$, the
set~$\mathcal{F}^\sigma_{\lvert\sigma\rvert}$ has only one element~$(\sigma,\id)$,
which we will denote by~$1_\sigma$.

If~$F=(T,\theta)$ is a~$\sigma$-flag and~$W\subset V(T)$ is such
that~$\im(\theta)\subset W$ (i.e., the set~$W$ contains all labelled vertices), then
we define the \emph{subflag induced by~$W$ on~$F$} as the
flag~$F\vert_W=(T\vert_W,\theta)$, where~$T\vert_W$ is the subtournament induced
by~$W$ on~$T$.

We now extend the notion of density to flags as well and also to a more general
setting of density of several flags.

\begin{definition}
  Let~$\sigma$ be a type of size~$k$ and~$\ell,\ell_1,\ell_2,\ldots,\ell_t\geq k$ be
  integers such that
  \begin{align*}
    \left(\sum_{i=1}^t\ell_i\right) - (t-1)k & \leq \ell.
  \end{align*}
  
  Let also~$F=(M,\theta),F_1,F_2,\ldots,F_t\in\mathcal{F}^\sigma$ be~$\sigma$-flags
  of sizes~$\ell,\ell_1,\ell_2,\ldots,\ell_t$ respectively.

  The \emph{joint density} of~$F_1,F_2,\ldots,F_t$ in~$F$, denoted
  by~$p(F_1,F_2,\ldots,F_t;F)$, is defined through the following random experiment.

  Pick uniformly at random pairwise disjoint
  subsets~$\bm{W_1},\bm{W_2},\ldots,\bm{W_t}$ of~$V(F)\setminus\im(\theta)$ subject
  to~$\lvert\bm{W_i}\rvert = \ell_i-k$ for every~$i\in[t]$ and define
  \begin{align*}
    p(F_1,F_2,\ldots,F_t;F) & = \prob{\forall i\in[t],
      F\vert_{\im(\theta)\cup\bm{W_i}}\cong F_i}.
  \end{align*}

  We also extend~$p$ linearly in each of its coordinates.
\end{definition}

We can (finally) present the \emph{flag algebra} of a type~$\sigma$.

\begin{proposition}[Razborov~{\cite[Lemma~2.4]{Raz07}}]\label{prop:flagalgebra}
  Let~$\sigma$ be a type of size~$k$
  and~$\mathcal{A}^\sigma=\mathbb{R}\mathcal{F}^\sigma/\mathcal{K}^\sigma$ denote the
  quotient of the set~$\mathbb{R}\mathcal{F}^\sigma$ of all formal linear
  combinations of elements of~$\mathcal{F}^\sigma$ by the linear
  subspace~$\mathcal{K}^\sigma$ generated by elements of the form
  \begin{align*}
    \widetilde{F} & - \sum_{F\in\mathcal{F}^\sigma_{\ell}}p(\widetilde{F};F)F,
  \end{align*}
  where~$\ell\geq\lvert\widetilde{F}\rvert$.

  Define also the linear
  product~${{}\cdot{}}\:\mathcal{A}^\sigma\times\mathcal{A}^\sigma\to\mathcal{A}^\sigma$
  through
  \begin{align*}
    F_1\cdot F_2 & = \sum_{F\in\mathcal{F}^\sigma_\ell}p(F_1,F_2;F)F,
  \end{align*}
  where~$F_1,F_2\in\mathcal{F}^\sigma$ and~$\ell\geq\lvert F_1\rvert+\lvert F_2\rvert
  - k$.

  Under these conditions, this product is well-defined and the
  set~$\mathcal{A}^\sigma$ equipped with this product (and the usual addition) is a
  commutative associative algebra over~$\mathbb{R}$ with unity~$1_\sigma$.
\end{proposition}

Let us denote by~$\Hom(\mathcal{A}^\sigma,\mathbb{R})$ the set of
all~$\mathbb{R}$-algebra homomorphisms from~$\mathcal{A}^\sigma$ to~$\mathbb{R}$ and
define the set of positive homomorphisms as
\begin{align*}
  \Homp{\sigma} & =
  \{\phi\in\Hom(\mathcal{A}^\sigma,\mathbb{R}) :
  \forall F\in\mathcal{F}^\sigma,\phi(F)\in[0,1]\}.
\end{align*}

We will now define the notion of a convergent sequence of flags.

\begin{definition}
  Let~$(F_n)_{n\in\mathbb{N}}$ be a sequence of~$\sigma$-flags.

  The sequence~$(F_n)_{n\in\mathbb{N}}$ is called \emph{increasing} if~$\lvert
  F_n\rvert<\lvert F_{n+1}\rvert$ for every~$n\in\mathbb{N}$.

  The sequence~$(F_n)_{n\in\mathbb{N}}$ is called \emph{convergent} if it is
  increasing and for every fixed~$\sigma$-flag~$F\in\mathcal{F}^\sigma$, the
  sequence~$(p(F;F_n))_{n\in\mathbb{N}}$ is convergent.

  If~$\phi\in\Homp{\sigma}$ is a homomorphism, we say that the
  sequence~$(F_n)_{n\in\mathbb{N}}$ \emph{converges to~$\phi$} if it is convergent
  and
  \begin{align*}
    \lim_{n\to\infty}p(F;F_n) & = \phi(F),
  \end{align*}
  for every~$\sigma$-flag~$F\in\mathcal{F}^\sigma$.
\end{definition}

It is easy to see (e.g., by a diagonalization argument) that every increasing
sequence of flags has a convergent subsequence. The next theorem says that the set of
positive homomorphisms~$\Homp{\sigma}$ captures precisely the limits of convergent
sequences of~$\sigma$-flags.

\begin{theorem}[Lovász--Szegedy~\cite{LoSz06}, Razborov~{\cite[Theorem~3.3]{Raz07}}]
  Every convergent sequence of~$\sigma$-flags converges to a positive homomorphism
  in~$\Homp{\sigma}$ and for every positive homomorphism~$\phi\in\Homp{\sigma}$ there
  exists a sequence of~$\sigma$-flags converging to~$\phi$.
\end{theorem}

Recall that we are interested in maximizing the density of a fixed tournament~$T$
asymptotically. This means that, in the language of flag algebras, we are interested
in the following problem.

\begin{problem}\label{prob:maxflag}
  Given a fixed tournament~$T\in\mathcal{T}$, compute
  \begin{align*}
    \max\{\phi(T) : \phi\in\Homp{\sigma}\}.
  \end{align*}
\end{problem}

\begin{remark}
  Here, we used~$\max$ instead of~$\sup$ because~$\Homp{\sigma}$ is compact.
\end{remark}

\subsection{Semidefinite method}
\label{subsec:semidefinite}

Providing lower bounds to Problem~\ref{prob:maxflag} is easy. Indeed, every
increasing sequence of tournaments~$(T_n)_{n\in\mathbb{N}}$ provides the lower bound
\begin{align*}
  \limsup_{n\to\infty}p(T;T_n).
\end{align*}

The hard part of this problem is to compute upper bounds. A first and na\"{i}ve way of
doing so is the following. If~$\phi\in\Homp{\sigma}$ is a homomorphism, then
Proposition~\ref{prop:flagalgebra} gives us
\begin{align}\label{eq:naive}
  \phi(T) & = \sum_{T'\in\mathcal{T}_\ell}p(T;T')\phi(T')\nonumber\\
  & \leq
  \left(\max_{T'\in\mathcal{T}_\ell}p(T;T')\right)\sum_{T'\in\mathcal{T}_\ell}\phi(T')
  = \left(\max_{T'\in\mathcal{T}_\ell}p(T;T')\right)\phi(1_0)
  = \left(\max_{T'\in\mathcal{T}_\ell}p(T;T')\right),
\end{align}
for every~$\ell\geq\lvert T\rvert$, since~$\phi$ is linear and
\begin{align*}
  1_0 & = \sum_{T'\in\mathcal{T}_\ell}T'.
\end{align*}

However, in general this bound is too weak to find extremal values. In what follows,
we will present the semidefinite method, which builds up on this simple argument but
can obtain much better bounds for Problem~\ref{prob:maxflag}.

Let us start by defining some flag algebra notation that will help us.

\begin{definition}
  Let~$\sigma$ be a type. We define the \emph{semantic cone of type~$\sigma$} as the
  set
  \begin{align*}
    \Csem{\sigma} & = \{f\in\mathcal{A}^\sigma : \forall\phi\in\Homp{\sigma},
    \phi(f)\geq 0\},
  \end{align*}
  that is, the semantic cone is the set of all ``positive'' elements
  of~$\mathcal{A}^\sigma$ with respect to positive homomorphisms.

  We define also the \emph{ordinary cone of type~$\sigma$} as the set
  \begin{align*}
    \cC{\sigma} & = \left\{\sum_{i=1}^t F_i\cdot f_i^2 : t\in\mathbb{N}
    \land F_1,F_2,\ldots,F_t\in\mathcal{F}^\sigma
    \land f_1,f_2,\ldots,f_t\in\mathcal{A}^\sigma
    \right\}.
  \end{align*}

  Finally, we define the preorder relation~$\leq_\sigma$ over~$\mathcal{A}^\sigma$
  through
  \begin{align*}
    f\leq_\sigma g & \iff g-f\in\Csem{\sigma}.
  \end{align*}
\end{definition}

Trivially we have~$\mathcal{F}^\sigma\subset\cC{\sigma}\subset\Csem{\sigma}$.

The idea of the semidefinite method is to use elements of~$\Csem{0}$ to
compensate large~$p(T;T')$ in~\eqref{eq:naive} as follows. If~$g\in\Csem{0}$, then
\begin{align}\label{eq:withg}
  T & \leq_0 T+g 
  = \sum_{T'\in\mathcal{T}_\ell}(p(T;T') + p(g;T'))T'\nonumber\\
  & \leq_0
  \left(\max_{T'\in\mathcal{T}_\ell}p(T;T') + p(g;T')\right)\sum_{T'\in\mathcal{T}_\ell}T'
  = \left(\max_{T'\in\mathcal{T}_\ell}p(T;T') + p(g;T')\right)1_0,
\end{align}
where~$\ell\in\mathbb{N}$ is large enough (so that we can write~$g$ as a combination
of tournaments of size smaller than~$\ell$). Our hope in doing so is to be able to
choose~$g$ such that~$p(g;T')$ is negative when~$p(T;T')$ is large, but
taking~$p(g;T')$ positive enough to ensure~$g\in\Csem{0}$ when~$p(T;T')$ is
small.

However, deciding whether an arbitrary~$g$ is an element of~$\Csem{\sigma}$ is hard
(in fact, our problem is exactly to prove that~$c1_0 - T\in\Csem{0}$ for a
certain~$c$).

We will now define the downward operator, which will help us in obtaining elements
of~$\Csem{0}$ from elements of~$\Csem{\sigma}$.

\begin{definition}
  Let~$\sigma$ be a type of size~$k$ and~$F=(T,\theta)$ be a~$\sigma$-flag. We denote
  the \emph{underlying tournament} of~$F$ by~$F\vert_0=T$ and we define the
  \emph{normalizing factor} of~$F$ (denoted~$q_\sigma(F)$) through the following
  random experiment.

  We pick uniformly at random an injective function~$\bm{\theta}\:[k]\to V(F\vert_0)$
  and let
  \begin{align*}
    q_\sigma(F) & = \prob{(F\vert_0,\bm{\theta})\cong F}.
  \end{align*}

  We also define the downward operator~$\llbracket{{}\cdot{}}\rrbracket_\sigma$ by
  letting
  \begin{align*}
    \llbracket F\rrbracket_\sigma & = q_\sigma(F)F\vert_0\quad\in\mathcal{A}^0,
  \end{align*}
  and extending it linearly to combinations of~$\sigma$-flags.
\end{definition}

\begin{theorem}[Razborov~{\cite[Theorems~2.5 and~3.1a]{Raz07}}]
  The downward operator~$\llbracket{{}\cdot{}}\rrbracket_\sigma$ is well-defined as
  an operator~$\mathcal{A}^\sigma\to\mathcal{A}^0$ and we have
  \begin{align*}
    \llbracket\Csem{\sigma}\rrbracket_\sigma & \subset \Csem{0}.
  \end{align*}
\end{theorem}

This theorem allows us to choose~$g$ of~\eqref{eq:withg} in the easier
set~$\llbracket\cC{\sigma}\rrbracket_\sigma$ for some type~$\sigma$. This reduces the
problem to finding a positive semidefinite matrix in the following way.

Fix a type~$\sigma$ of size~$k$, a~$\sigma$-flag~$F'$ and let~$\widetilde{\ell}$
and~$\ell$ be integers such that~$k\leq\widetilde{\ell}$ and~$\lvert F'\rvert +
2\widetilde{\ell}-2k\leq\ell$.

If~$v\in\mathbb{R}^{\mathcal{F}^\sigma_{\widetilde{\ell}}}$ is a vector indexed
by~$\mathcal{F}^\sigma_{\widetilde{\ell}}$, then let~$F(v)$ denote the element
\begin{align*}
  \sum_{F\in\mathcal{F}^\sigma_{\widetilde{\ell}}}v_FF
  & \in \mathcal{A}^\sigma.
\end{align*}

Analogously, if~$Q$ is a matrix indexed by
by~$\mathcal{F}^\sigma_{\widetilde{\ell}}\times\mathcal{F}^\sigma_{\widetilde{\ell}}$, let~$F(Q)$
denote the element
\begin{align*}
  \sum_{F_1,F_2\in\mathcal{F}^\sigma_{\widetilde{\ell}}}Q_{F_1F_2}F_1F_2
  & \in \mathcal{A}^\sigma.
\end{align*}

Note that if~$Q$ is positive semidefinite ($Q\succeq 0$), then by the Spectral
Theorem there exist
vectors~$v_1,v_2,\ldots,v_r\in\mathbb{R}^{\mathcal{F}^\sigma_{\widetilde{\ell}}}$
such that
\begin{align*}
  Q & = \sum_{i=1}^r v_iv_i^\top,
\end{align*}
which means that
\begin{align*}
  F(Q) & = \sum_{i=1}^r F(v_i)^2.
\end{align*}

Hence we have~$F'\cdot F(Q)\in\cC{\sigma}$ and we can take~$g$ in~\eqref{eq:withg} to
be equal to
\begin{align*}
  \llbracket F'\cdot F(Q)\rrbracket_\sigma.
\end{align*}
This yields the following semidefinite program
\begin{align*}
  \min\; & y \\
  \text{s.t.} & \quad
  p(T,T') +
  \sum_{F\in\mathcal{F}^\sigma_\ell}\sum_{F_1,F_2\in\mathcal{F}^\sigma_{\widetilde{\ell}}}
  Q_{F_1F_2}p(F',F_i,F_j;F)p(\llbracket F\rrbracket_\sigma;T')
  \leq y \qquad \forall T'\in\mathcal{T}_\ell;\\
  & \quad
  Q \in
  \mathbb{R}^{\mathcal{F}^\sigma_{\widetilde{\ell}}\times\mathcal{F}^\sigma_{\widetilde{\ell}}}
  \text{ is positive semidefinite};
\end{align*}
whose solutions have values that are upper bounds to the value in
Problem~\ref{prob:maxflag}.

In fact, we can even take~$g = \sum_{i=1}^mg_i$ in~\eqref{eq:withg}, where each~$g_i$
is of the form
\begin{align*}
  \llbracket F'_i\cdot F(Q_i)\rrbracket_{\sigma_i},
\end{align*}
for some type~$\sigma_i$, some~$\sigma_i$-flag~$F'_i$ and some positive semidefinite
matrix~$Q_i$ indexed
by~$\mathcal{F}^{\sigma_i}_{\ell_i}\times\mathcal{F}^{\sigma_i}_{\ell_i}$.

We state the resulting semidefinite program in the proposition below.

\begin{proposition}[\cite{Raz10}]\label{prop:semidefinite}
  Let~$T\in\mathcal{T}$ be a tournament, let~$\sigma_1,\sigma_2,\ldots,\sigma_m$ be
  types of sizes~$k_1,k_2,\ldots,k_m$ respectively and for each~$t\in[m]$,
  let~$F'_t\in\mathcal{F}^{\sigma_t}$ be a~$\sigma_t$-flag. Let
  also~$\ell_1,\ell_2,\ldots,\ell_m,\ell$ be integers such that
  \begin{align*}
    k_t & \leq \ell_t; & \lvert F'_t\rvert + 2\ell_t - 2k_t \leq \ell;
  \end{align*}
  for every~$t\in[m]$ and such that~$\lvert T\rvert\leq\ell$.

  Under these circumstances, the every value of every solution of the semidefinite
  program
  \begin{align}
    \min\; & y \nonumber\\
    \text{s.t.} & \quad
    p(T,T') +
    \sum_{t=1}^m\sum_{F\in\mathcal{F}^{\sigma_t}_\ell}\sum_{F_1,F_2\in\mathcal{F}^{\sigma_t}_{\ell_t}}
    Q^{(t)}_{F_1F_2}p(F'_t,F_i,F_j;F)p(\llbracket F\rrbracket_{\sigma_t};T')
    \leq y \qquad \forall T'\in\mathcal{T}_\ell;\nonumber\\
    & \quad
    Q^{(t)} \in
    \mathbb{R}^{\mathcal{F}^{\sigma_t}_{\ell_t}\times\mathcal{F}^{\sigma_t}_{\ell_t}}
    \text{ is positive semidefinite} \qquad \forall t\in[m];
    \label{eq:semidefinite}
  \end{align}
  is an upper bound to the value in Problem~\ref{prob:maxflag}, that is, if~$V$ is
  the value of a solution of~\eqref{eq:semidefinite}, then
  \begin{align*}
    \max\{\phi(T) : \phi\in\Homp{0}\} & \leq V.
  \end{align*}
\end{proposition}

In this text, all instances of~\eqref{eq:semidefinite} will be
with~$F'_t=1_{\sigma_t}$ for every~$t\in[m]$. Furthermore, when we use
Proposition~\ref{prop:semidefinite} to give upper bounds to
Problem~\ref{prob:maxflag}, we will denote each of the~$Q^{(t)}$
in~\eqref{eq:semidefinite} by~$Q(T,\sigma_t)$ as a reminder of which problem we are
solving and of what is the type involved. Moreover, for each~$T'\in\mathcal{T}_\ell$,
we define
\begin{align*}
  c(Q(T,\sigma_t);T') & =
  p(\llbracket F(Q(T,\sigma_t))\rrbracket_{\sigma_t};T') =
  \sum_{F\in\mathcal{F}^{\sigma_t}_\ell}\sum_{F_1,F_2\in\mathcal{F}^{\sigma_t}_{\ell_t}}
  Q(T,\sigma_t)_{F_1F_2}p(F_i,F_j;F)p(\llbracket F\rrbracket_{\sigma_t};T')
\end{align*}
and let
\begin{align*}
  c(T;T') & = \sum_{t=1}^m c(Q(T,\sigma_t);T'),
\end{align*}
so that~\eqref{eq:semidefinite} becomes
\begin{align}
  \min\; & y
  \nonumber\\
  \text{s.t.} & \quad
  p(T,T') + c(T;T') \leq y \qquad \forall T'\in\mathcal{T}_\ell;
  \nonumber\\
  & \quad
  c(T;T') = \sum_{t=1}^m c(Q(T,\sigma_t);T') \qquad \forall T'\in\mathcal{T}_\ell;
  \nonumber\\
  & \quad
  c(Q(T,\sigma_t);T') =
  \sum_{F\in\mathcal{F}^{\sigma_t}_\ell}\sum_{F_1,F_2\in\mathcal{F}^{\sigma_t}_{\ell_t}}
  Q(T,\sigma_t)_{F_1F_2}p(F_i,F_j;F)p(\llbracket F\rrbracket_{\sigma_t};T');
  \nonumber\\
  & \quad
  Q(T,\sigma_t) \in
  \mathbb{R}^{\mathcal{F}^{\sigma_t}_{\ell_t}\times\mathcal{F}^{\sigma_t}_{\ell_t}}
  \text{ is positive semidefinite} \qquad \forall t\in[m].\label{eq:semidefiniteshort}
\end{align}

\subsection{Tournaments, types and flags used}
\label{subsec:typesandflags}

Throughout this text, we denote the transitive tournament of size~$k$ by~$\Tr_k$. We
also denote (see Figure~\ref{fig:typesandflags}).
\begin{itemize}
\item the~$3$-cycle by~$\vec C_3$;
\item the only tournament of size~$4$ that has a~$4$-cycle by~$R_4$;
\item the only tournament with outdegree sequence~$(1,1,1,3)$ by~$W_4$;
\item the only tournament with outdegree sequence~$(0,2,2,2)$ by~$L_4$.
\end{itemize}

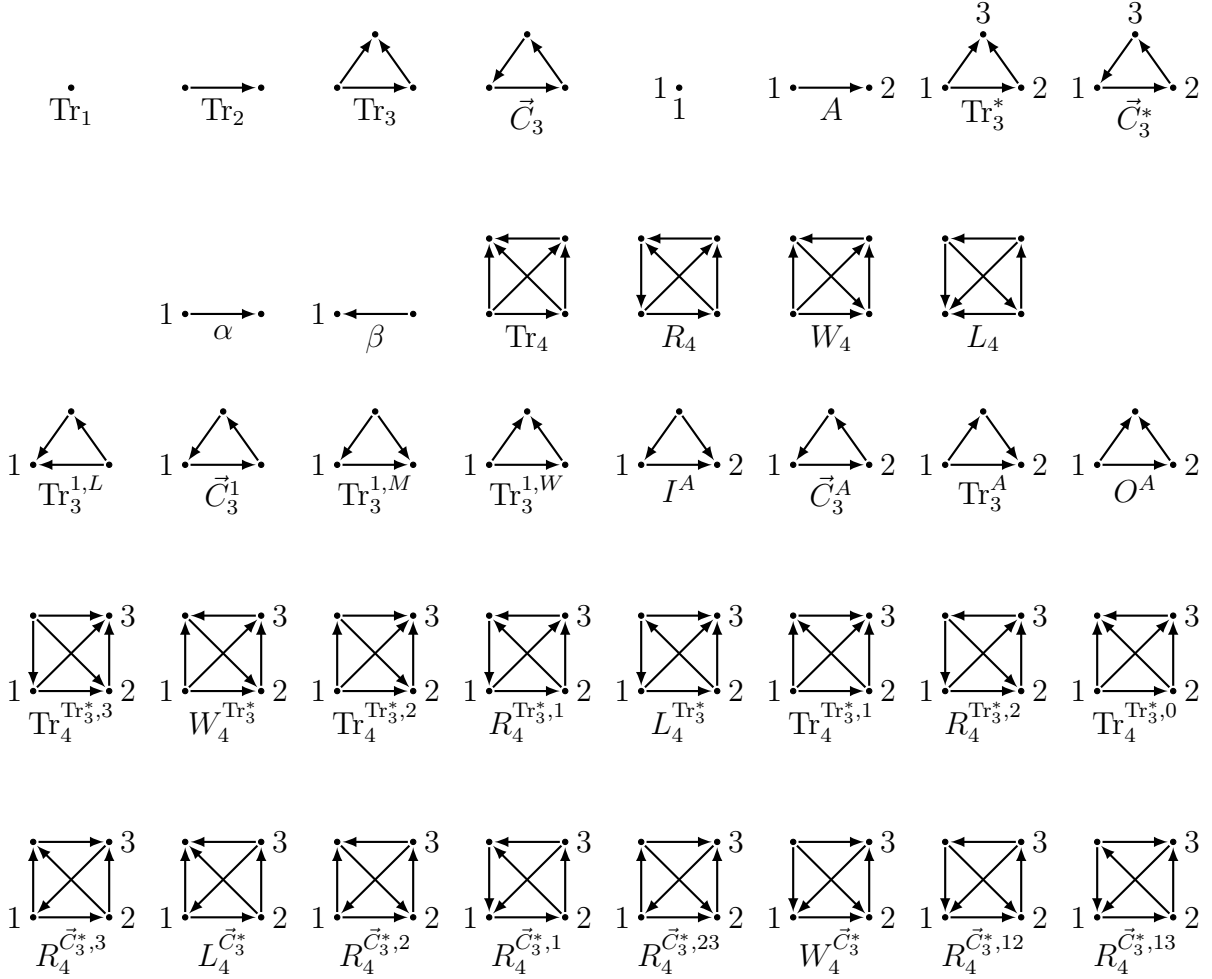
\begin{figure}[ht]
  \tikzsetnextfilename{typesandflags}
\begin{tikzpicture}
  \def\psize{1pt}
  \def\shorten{2pt}

  \foreach \i in {0,...,7}{
    \pgfmathsetmacro{\base}{2*\i}
    \pgfmathsetmacro{\next}{\base+1}

    \coordinate (T3A\i) at (\base cm, 0cm);
    \coordinate (T3B\i) at (\next cm, 0cm);
    \coordinate (T3M\i) at ($1/2*(T3A\i) + 1/2*(T3B\i)$);
    \coordinate (T3C\i) at ($(T3M\i) + (0cm,0.707106cm)$);

    \coordinate (T4A\i) at (\base cm,-3cm);
    \coordinate (T4B\i) at (\next cm,-3cm);
    \coordinate (T4C\i) at (\next cm,-2cm);
    \coordinate (T4D\i) at (\base cm,-2cm);
    \coordinate (T4M\i) at ($1/2*(T4A\i)+1/2*(T4B\i)$);

    \coordinate (F2A\i) at (\base cm, -5cm);
    \coordinate (F2B\i) at (\next cm, -5cm);
    \coordinate (F2M\i) at ($1/2*(F2A\i) + 1/2*(F2B\i)$);
    \coordinate (F2C\i) at ($(F2M\i) + (0cm,0.707106cm)$);

    \coordinate (FTr3A\i) at (\base cm,-8cm);
    \coordinate (FTr3B\i) at (\next cm,-8cm);
    \coordinate (FTr3C\i) at (\next cm,-7cm);
    \coordinate (FTr3D\i) at (\base cm,-7cm);
    \coordinate (FTr3M\i) at ($1/2*(FTr3A\i)+1/2*(FTr3B\i)$);

    \coordinate (FC3A\i) at (\base cm,-11cm);
    \coordinate (FC3B\i) at (\next cm,-11cm);
    \coordinate (FC3C\i) at (\next cm,-10cm);
    \coordinate (FC3D\i) at (\base cm,-10cm);
    \coordinate (FC3M\i) at ($1/2*(FC3A\i)+1/2*(FC3B\i)$);
  }

  \foreach \i in {0,4}
  \filldraw (T3M\i) circle (\psize);

  \foreach \i in {1,2,3,5,6,7}{
    \filldraw (T3A\i) circle (\psize);
    \filldraw (T3B\i) circle (\psize);
    \draw[thick,arrows={-latex},shorten <=\shorten,shorten >=\shorten]
    (T3A\i) -- (T3B\i);
  }

  \foreach \i in {2,3,6,7}{
    \filldraw (T3C\i) circle (\psize);
    \draw[thick, arrows={-latex},shorten <=\shorten,shorten >=\shorten]
    (T3B\i) -- (T3C\i);
  }

  \foreach \i in {2,6}
  \draw[thick, arrows={-latex},shorten <=\shorten,shorten >=\shorten]
  (T3A\i) -- (T3C\i);

  \foreach \i in {3,7}
  \draw[thick, arrows={-latex},shorten <=\shorten,shorten >=\shorten]
  (T3C\i) -- (T3A\i);

  \node[left] at (T3M4) {$1$};
  \foreach \i in {5,...,7}{
    \node[left] at (T3A\i) {$1$};
    \node[right] at (T3B\i) {$2$};
  }
  \foreach \i in {6,7}
  \node[above] at (T3C\i) {$3$};

  \foreach \i/\L in {%
    0/\Tr_1, 1/\Tr_2, 2/\Tr_3, 3/\vec C_3, 4/1, 5/A, 6/\Tr_3^*, 7/\vec C_3^*%
  }
  \node[below] at (T3M\i) {$\L$};


  \foreach \i in {1,2}{
    \filldraw (T4A\i) circle (\psize);
    \filldraw (T4B\i) circle (\psize);
    \node[left] at (T4A\i) {$1$};
  }
  \draw[thick, arrows={-latex},shorten <=\shorten,shorten >=\shorten]
  (T4A1) -- (T4B1);
  \draw[thick, arrows={-latex},shorten <=\shorten,shorten >=\shorten]
  (T4B2) -- (T4A2);
  \foreach \i/\L in {%
    1/\alpha, 2/\beta%
  }
  \node[below] at (T4M\i) {$\L$};


  \foreach \i in {3,...,6}{
    \filldraw (T4A\i) circle (\psize);
    \filldraw (T4B\i) circle (\psize);
    \filldraw (T4C\i) circle (\psize);
    \filldraw (T4D\i) circle (\psize);
  }

  \foreach \i in {3,4,5,6}{
    \draw[thick, arrows={-latex},shorten <=\shorten,shorten >=\shorten]
    (T4B\i) -- (T4C\i);
    \draw[thick, arrows={-latex},shorten <=\shorten,shorten >=\shorten]
    (T4C\i) -- (T4D\i);
  }

  \foreach \i in {3,4,5}{
    \draw[thick, arrows={-latex},shorten <=\shorten,shorten >=\shorten]
    (T4A\i) -- (T4B\i);
    \draw[thick, arrows={-latex},shorten <=\shorten,shorten >=\shorten]
    (T4A\i) -- (T4C\i);
  }
  \draw[thick, arrows={-latex},shorten <=\shorten,shorten >=\shorten]
  (T4B6) -- (T4A6);
  \draw[thick, arrows={-latex},shorten <=\shorten,shorten >=\shorten]
  (T4C6) -- (T4A6);

  \foreach \i in {3,4}
  \draw[thick, arrows={-latex},shorten <=\shorten,shorten >=\shorten]
  (T4B\i) -- (T4D\i);
  \foreach \i in {5,6}
  \draw[thick, arrows={-latex},shorten <=\shorten,shorten >=\shorten]
  (T4D\i) -- (T4B\i);

  \foreach \i in {3,5}
  \draw[thick, arrows={-latex},shorten <=\shorten,shorten >=\shorten]
  (T4A\i) -- (T4D\i);
  \foreach \i in {4,6}
  \draw[thick, arrows={-latex},shorten <=\shorten,shorten >=\shorten]
  (T4D\i) -- (T4A\i);
  
  \foreach \i/\L in {%
    3/\Tr_4, 4/R_4, 5/W_4, 6/L_4%
  }
  \node[below] at (T4M\i) {$\L$};


  \foreach \i in {0,...,7}{
    \filldraw (F2A\i) circle (\psize);
    \filldraw (F2B\i) circle (\psize);
    \filldraw (F2C\i) circle (\psize);
  }

  \foreach \i in {1,...,7}
  \draw[thick, arrows={-latex},shorten <=\shorten,shorten >=\shorten]
  (F2A\i) -- (F2B\i);
  \draw[thick, arrows={-latex},shorten <=\shorten,shorten >=\shorten]
  (F2B0) -- (F2A0);

  \foreach \i in {3,6,7}
  \draw[thick, arrows={-latex},shorten <=\shorten,shorten >=\shorten]
  (F2A\i) -- (F2C\i);
  \foreach \i in {0,1,2,4,5}
  \draw[thick, arrows={-latex},shorten <=\shorten,shorten >=\shorten]
  (F2C\i) -- (F2A\i);

  \foreach \i in {0,1,3,5,7}
  \draw[thick, arrows={-latex},shorten <=\shorten,shorten >=\shorten]
  (F2B\i) -- (F2C\i);
  \foreach \i in {2,4,6}
  \draw[thick, arrows={-latex},shorten <=\shorten,shorten >=\shorten]
  (F2C\i) -- (F2B\i);

  \foreach \i in {0,...,7}
  \node[left] at (F2A\i) {$1$};

  \foreach \i in {4,...,7}
  \node[right] at (F2B\i) {$2$};

  \foreach \i/\L in {%
    0/\Tr_3^{1,L}, 1/\vec C_3^1, 2/\Tr_3^{1,M}, 3/\Tr_3^{1,W},%
    4/I^A, 5/\vec C_3^A, 6/\Tr_3^A, 7/O^A%
  }
  \node[below] at (F2M\i) {$\L$};


  \foreach \i in {0,...,7}{
    \filldraw (FTr3A\i) circle (\psize);
    \filldraw (FTr3B\i) circle (\psize);
    \filldraw (FTr3C\i) circle (\psize);
    \filldraw (FTr3D\i) circle (\psize);
    \draw[thick, arrows={-latex},shorten <=\shorten,shorten >=\shorten]
    (FTr3A\i) -- (FTr3B\i);
    \draw[thick, arrows={-latex},shorten <=\shorten,shorten >=\shorten]
    (FTr3A\i) -- (FTr3C\i);
    \draw[thick, arrows={-latex},shorten <=\shorten,shorten >=\shorten]
    (FTr3B\i) -- (FTr3C\i);
    \node[left] at (FTr3A\i) {$1$};
    \node[right] at (FTr3B\i) {$2$};
    \node[right] at (FTr3C\i) {$3$};
  }

  \foreach \i in {1,2,5,7}
  \draw[thick, arrows={-latex},shorten <=\shorten,shorten >=\shorten]
  (FTr3A\i) -- (FTr3D\i);
  \foreach \i in {0,3,4,6}
  \draw[thick, arrows={-latex},shorten <=\shorten,shorten >=\shorten]
  (FTr3D\i) -- (FTr3A\i);

  \foreach \i in {3,4,5,7}
  \draw[thick, arrows={-latex},shorten <=\shorten,shorten >=\shorten]
  (FTr3B\i) -- (FTr3D\i);
  \foreach \i in {0,1,2,6}
  \draw[thick, arrows={-latex},shorten <=\shorten,shorten >=\shorten]
  (FTr3D\i) -- (FTr3B\i);

  \foreach \i in {1,3,6,7}
  \draw[thick, arrows={-latex},shorten <=\shorten,shorten >=\shorten]
  (FTr3C\i) -- (FTr3D\i);
  \foreach \i in {0,2,4,5}
  \draw[thick, arrows={-latex},shorten <=\shorten,shorten >=\shorten]
  (FTr3D\i) -- (FTr3C\i);

  \foreach \i/\L in {%
    0/\Tr_4^{\Tr_3^*,3}, 1/W_4^{\Tr_3^*},%
    2/\Tr_4^{\Tr_3^*,2}, 3/R_4^{\Tr_3^*,1},%
    4/L_4^{\Tr_3^*}, 5/\Tr_4^{\Tr_3^*,1},%
    6/R_4^{\Tr_3^*,2}, 7/\Tr_4^{\Tr_3^*,0}%
  }
  \node[below] at (FTr3M\i) {$\L$};


  \foreach \i in {0,...,7}{
    \filldraw (FC3A\i) circle (\psize);
    \filldraw (FC3B\i) circle (\psize);
    \filldraw (FC3C\i) circle (\psize);
    \filldraw (FC3D\i) circle (\psize);
    \draw[thick, arrows={-latex},shorten <=\shorten,shorten >=\shorten]
    (FC3A\i) -- (FC3B\i);
    \draw[thick, arrows={-latex},shorten <=\shorten,shorten >=\shorten]
    (FC3C\i) -- (FC3A\i);
    \draw[thick, arrows={-latex},shorten <=\shorten,shorten >=\shorten]
    (FC3B\i) -- (FC3C\i);
    \node[left] at (FC3A\i) {$1$};
    \node[right] at (FC3B\i) {$2$};
    \node[right] at (FC3C\i) {$3$};
  }

  \foreach \i in {0,1,2,4}
  \draw[thick, arrows={-latex},shorten <=\shorten,shorten >=\shorten]
  (FC3A\i) -- (FC3D\i);
  \foreach \i in {3,5,6,7}
  \draw[thick, arrows={-latex},shorten <=\shorten,shorten >=\shorten]
  (FC3D\i) -- (FC3A\i);

  \foreach \i in {0,1,3,7}
  \draw[thick, arrows={-latex},shorten <=\shorten,shorten >=\shorten]
  (FC3B\i) -- (FC3D\i);
  \foreach \i in {2,4,5,6}
  \draw[thick, arrows={-latex},shorten <=\shorten,shorten >=\shorten]
  (FC3D\i) -- (FC3B\i);

  \foreach \i in {1,2,3,6}
  \draw[thick, arrows={-latex},shorten <=\shorten,shorten >=\shorten]
  (FC3C\i) -- (FC3D\i);
  \foreach \i in {0,4,5,7}
  \draw[thick, arrows={-latex},shorten <=\shorten,shorten >=\shorten]
  (FC3D\i) -- (FC3C\i);

  \foreach \i/\L in {%
    0/R_4^{\vec C_3^*,3}, 1/L_4^{\vec C_3^*},%
    2/R_4^{\vec C_3^*,2}, 3/R_4^{\vec C_3^*,1},%
    4/R_4^{\vec C_3^*,23}, 5/W_4^{\vec C_3^*},%
    6/R_4^{\vec C_3^*,12}, 7/R_4^{\vec C_3^*,13}%
  }
  \node[below] at (FC3M\i) {$\L$};
\end{tikzpicture}

  \caption{Types and flags of size at most~$4$ used.}
  \label{fig:typesandflags}
\end{figure}

We will also use the notation of Figure~\ref{fig:tour5} for the non isomorphic tournaments of
size~$5$.

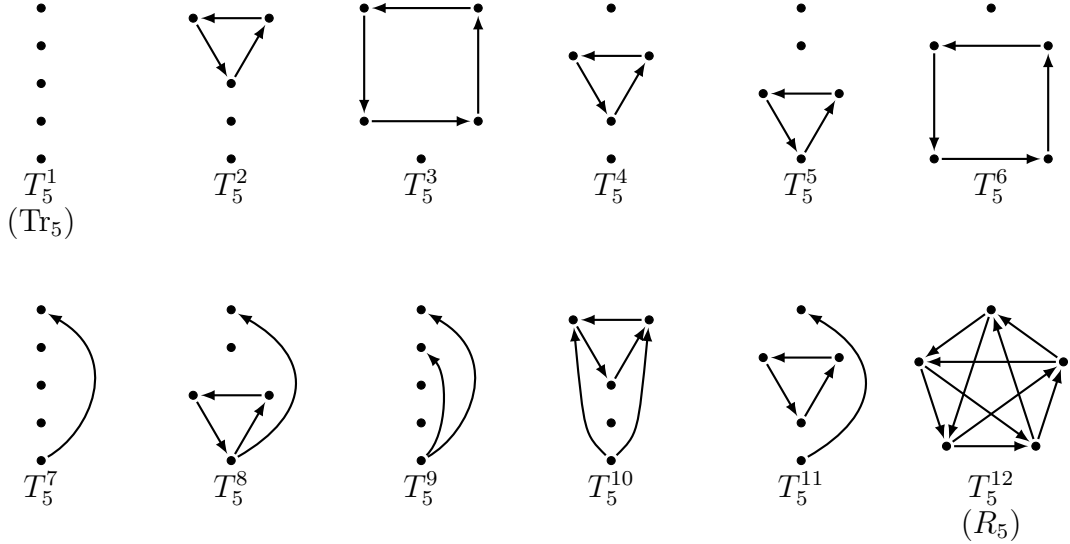
\begin{figure}[ht]
  \tikzsetnextfilename{tour5}
\begin{tikzpicture}[scale=0.5]
  \def\psize{3pt}
  \def\dist{1cm}
  \def\tdist{2cm}
  \def\sdist{3cm}
  \def\prad{2cm}
  \def\latdist{5cm}
  \def\tinycontout{1cm}
  \def\smallcontout{1.5cm}
  \def\contout{2.5cm}
  \def\sh{0.1cm}
  \def\altldist{1cm}

  \foreach \i in {1,...,6}{
    \coordinate (M\i) at ($\i*(\latdist,0cm)$);
    \coordinate (MA\i) at ($(\i*\latdist,-\altldist)$);
    \node[below] at (M\i) {$T_5^{\i}$};

    \foreach \j in {0,...,4}{
      \coordinate (\i-\j) at ($(M\i) + \j*(0cm,\dist)$);
      \coordinate (\i-\j TL) at ($(\i-\j) + (120:\tdist)$);
      \coordinate (\i-\j TR) at ($(\i-\j) + (60:\tdist)$);
      \coordinate (\i-\j SL0) at ($(\i-\j) + 1/2*(-\sdist,0cm)$);
      \coordinate (\i-\j SR0) at ($(\i-\j) + 1/2*(\sdist,0cm)$);
      \coordinate (\i-\j SL1) at ($(\i-\j) + 1/2*(-\sdist,0cm) + (0cm,\sdist)$);
      \coordinate (\i-\j SR1) at ($(\i-\j) + 1/2*(\sdist,0cm) + (0cm,\sdist)$);
    }
  }

  \begin{scope}[shift={($-8*(0cm,\dist) + -6*(\latdist,0cm)$)}]
    \foreach \i in {7,...,11}{
      \coordinate (M\i) at ($\i*(\latdist,0cm)$);
      \coordinate (MA\i) at ($\i*(\latdist,-\altldist)$);
      \node[below] at (M\i) {$T_5^{\i}$};

      \foreach \j in {0,...,4}{
        \coordinate (\i-\j) at ($(M\i) + \j*(0cm,\dist)$);
        \coordinate (\i-\j TL) at ($(\i-\j) + (120:\tdist)$);
        \coordinate (\i-\j TR) at ($(\i-\j) + (60:\tdist)$);
        \coordinate (\i-\j SL0) at ($(\i-\j) + 1/2*(-\sdist,0cm)$);
        \coordinate (\i-\j SR0) at ($(\i-\j) + 1/2*(\sdist,0cm)$);
        \coordinate (\i-\j SL1) at ($(\i-\j) + 1/2*(-\sdist,0cm) + (0cm,\sdist)$);
        \coordinate (\i-\j SR1) at ($(\i-\j) + 1/2*(\sdist,0cm) + (0cm,\sdist)$);
      }
    }
    \node[below] at (MA1) {$(\Tr_5)$};

    \coordinate (M12) at ($12*(\latdist,0cm)$);
    \coordinate (MA12) at ($(12*\latdist,-\altldist)$);
    \node[below] at (M12) {$T_5^{12}$};
    \node[below] at (MA12) {$(R_5)$};
    \foreach \i in {0,...,6}{
      \pgfmathsetmacro{\angle}{90 + \i * 72}
      \coordinate (12-\i) at ($(M12) + 2*(0cm,\dist) + (\angle:\prad)$);
    }
  \end{scope}

  \foreach \i in {1,7,9}{
    \foreach \j in {0,...,4}
    \filldraw (\i-\j) circle (\psize);
  }

  \foreach \i in {2,10}{
    \foreach \j in {0,...,2}
    \filldraw (\i-\j) circle (\psize);
    \filldraw (\i-2TL) circle (\psize);
    \filldraw (\i-2TR) circle (\psize);
  }

  \filldraw (3-0) circle (\psize);
  \filldraw (3-1SL0) circle (\psize);
  \filldraw (3-1SR0) circle (\psize);
  \filldraw (3-1SL1) circle (\psize);
  \filldraw (3-1SR1) circle (\psize);

  \foreach \i in {4,11}{
    \foreach \j in {0,1,4}
    \filldraw (\i-\j) circle (\psize);
    \filldraw (\i-1TL) circle (\psize);
    \filldraw (\i-1TR) circle (\psize);
  }

  \foreach \i in {5,8}{
    \foreach \j in {0,3,4}
    \filldraw (\i-\j) circle (\psize);
    \filldraw (\i-0TL) circle (\psize);
    \filldraw (\i-0TR) circle (\psize);
  }

  \filldraw (6-0SL0) circle (\psize);
  \filldraw (6-0SR0) circle (\psize);
  \filldraw (6-0SL1) circle (\psize);
  \filldraw (6-0SR1) circle (\psize);
  \filldraw (6-4) circle (\psize);


  \draw[thick,arrows={-latex},shorten <=\sh,shorten >=\sh] (2-2) -- (2-2TR);
  \draw[thick,arrows={-latex},shorten <=\sh,shorten >=\sh] (2-2TR) -- (2-2TL);
  \draw[thick,arrows={-latex},shorten <=\sh,shorten >=\sh] (2-2TL) -- (2-2);

  \draw[thick,arrows={-latex},shorten <=\sh,shorten >=\sh] (3-1SL0) -- (3-1SR0);
  \draw[thick,arrows={-latex},shorten <=\sh,shorten >=\sh] (3-1SR0) -- (3-1SR1);
  \draw[thick,arrows={-latex},shorten <=\sh,shorten >=\sh] (3-1SR1) -- (3-1SL1);
  \draw[thick,arrows={-latex},shorten <=\sh,shorten >=\sh] (3-1SL1) -- (3-1SL0);

  \draw[thick,arrows={-latex},shorten <=\sh,shorten >=\sh] (4-1) -- (4-1TR);
  \draw[thick,arrows={-latex},shorten <=\sh,shorten >=\sh] (4-1TR) -- (4-1TL);
  \draw[thick,arrows={-latex},shorten <=\sh,shorten >=\sh] (4-1TL) -- (4-1);

  \draw[thick,arrows={-latex},shorten <=\sh,shorten >=\sh] (5-0) -- (5-0TR);
  \draw[thick,arrows={-latex},shorten <=\sh,shorten >=\sh] (5-0TR) -- (5-0TL);
  \draw[thick,arrows={-latex},shorten <=\sh,shorten >=\sh] (5-0TL) -- (5-0);

  \draw[thick,arrows={-latex},shorten <=\sh,shorten >=\sh] (6-0SL0) -- (6-0SR0);
  \draw[thick,arrows={-latex},shorten <=\sh,shorten >=\sh] (6-0SR0) -- (6-0SR1);
  \draw[thick,arrows={-latex},shorten <=\sh,shorten >=\sh] (6-0SR1) -- (6-0SL1);
  \draw[thick,arrows={-latex},shorten <=\sh,shorten >=\sh] (6-0SL1) -- (6-0SL0);

  \draw[thick,arrows={-latex},shorten <=\sh,shorten >=\sh] (7-0)
  .. controls +(30:\smallcontout) and ($(7-4)+(-30:\contout)$)
  .. (7-4);

  \draw[thick,arrows={-latex},shorten <=\sh,shorten >=\sh] (8-0) -- (8-0TR);
  \draw[thick,arrows={-latex},shorten <=\sh,shorten >=\sh] (8-0TR) -- (8-0TL);
  \draw[thick,arrows={-latex},shorten <=\sh,shorten >=\sh] (8-0TL) -- (8-0);
  \draw[thick,arrows={-latex},shorten <=\sh,shorten >=\sh] (8-0)
  .. controls +(30:\contout) and ($(8-4)+(-30:\contout)$)
  .. (8-4);

  \draw[thick,arrows={-latex},shorten <=\sh,shorten >=\sh] (9-0)
  .. controls +(30:\smallcontout) and ($(9-4)+(-30:\contout)$)
  .. (9-4);
  \draw[thick,arrows={-latex},shorten <=\sh,shorten >=\sh] (9-0)
  .. controls +(45:\tinycontout) and ($(9-3)+(-45:\tinycontout)$)
  .. (9-3);

  \draw[thick,arrows={-latex},shorten <=\sh,shorten >=\sh] (10-2) -- (10-2TR);
  \draw[thick,arrows={-latex},shorten <=\sh,shorten >=\sh] (10-2TR) -- (10-2TL);
  \draw[thick,arrows={-latex},shorten <=\sh,shorten >=\sh] (10-2TL) -- (10-2);
  \draw[thick,arrows={-latex},shorten <=\sh,shorten >=\sh] (10-0)
  .. controls +(135:\tinycontout)
  .. (10-2TL);
  \draw[thick,arrows={-latex},shorten <=\sh,shorten >=\sh] (10-0)
  .. controls +(45:\tinycontout)
  .. (10-2TR);

  \draw[thick,arrows={-latex},shorten <=\sh,shorten >=\sh] (11-1) -- (11-1TR);
  \draw[thick,arrows={-latex},shorten <=\sh,shorten >=\sh] (11-1TR) -- (11-1TL);
  \draw[thick,arrows={-latex},shorten <=\sh,shorten >=\sh] (11-1TL) -- (11-1);
  \draw[thick,arrows={-latex},shorten <=\sh,shorten >=\sh] (11-0)
  .. controls +(30:\contout) and ($(11-4)+(-30:\contout)$)
  .. (11-4);

  \foreach \i in {0,...,4}{
    \pgfmathsetmacro{\n}{\i+1}
    \pgfmathsetmacro{\nn}{\i+2}
    \filldraw (12-\i) circle (\psize);
    \draw[thick,arrows={-latex},shorten <=\sh,shorten >=\sh] (12-\i) -- (12-\n);
    \draw[thick,arrows={-latex},shorten <=\sh,shorten >=\sh] (12-\i) -- (12-\nn);
  }
\end{tikzpicture}

  \caption{Tournaments of size~$5$. The arcs omitted are all oriented downward.}
  \label{fig:tour5}
\end{figure}

Furthermore, we define the following types (see Figure~\ref{fig:typesandflags}).
\begin{itemize}
\item The only type of size~$1$ is denoted by~$1$;
\item The type of size~$2$ where the vertex with label~$1$ beats the vertex with
  label~$2$ is denoted~$A$;
\item The type of size~$3$ isomorphic to~$\Tr_3$ such that the winner has label~$1$
  and the loser has label~$3$ is denoted~$\Tr_3^*$;
\item The type of size~$3$ isomorphic to~$\vec C_3$ such that the vertex with
  label~$1$ beats the vertex with label~$2$ is denoted~$\vec C_3^*$.
\end{itemize}

If~$T$ is a tournament and~$\sigma$ is a type such that there exists exactly
one~$\sigma$-flag~$F$ such that~$F\vert_0=T$, then we denote such flag
by~$T^\sigma$. Note that this uniquely defines the following flags.
\begin{align*}
  \vec C_3^1, \vec C_3^A, W_4^{\Tr_3^*}, L_4^{\Tr_3^*}, W_4^{\vec C_3^*}, L_4^{\vec C_3^*}.
\end{align*}

For the remaining flags, we use the notation of Figure~\ref{fig:typesandflags}. Let
us only comment the reasoning behind our notation.
\begin{itemize}
\item The notation for the flags~$O^A$ and~$I^A$ are meant to be a mnemonic for common
  outneighbourhood and common inneighbourhood respectively;
\item The flag~$\Tr_3^A$ is not the only~$A$-flag over~$\Tr_3$, but this notation is
  nevertheless used since~$\Tr_3^A$ is the only remaining~$A$-flag over~$\Tr_3$;
\item The~$\Tr_3^*$-flags over~$\Tr_4$ and~$R_4$ are uniquely determined by the
  outdegree~$d$ of the unlabelled vertex and as such, we denote them accordingly
  by~$\Tr_4^{\Tr_3^*,d}$ and~$R_4^{\Tr_3^*,d}$;
\item The~$\vec C_3^*$-flags over~$R_4$ are uniquely determined by the
  outneighbourhood of the unlabelled vertex and as such, we denote the accordingly by
  listing the vertices in the outneighbourhood of the unlabelled vertex in the
  superscript.
\end{itemize}


		   
\section{Upper bounds}\label{sec:UB}

In this section we prove the upper bounds in Theorems~\ref{thm:carousel},
\ref{thm:quasirandom} and \ref{thm:triangle}.  We use the semidefinite method of flag
algebras as presented in Section~\ref{sec:sdp}.

\begin{lemma}\label{lemma:ub} For every~$n$-vertex tournament~$T_n$, 
  \begin{align*}
	\lim_{n \to \infty}p(\tdoze;T_n) & \leq \frac{1}{16}.
  \end{align*}
\end{lemma}
\begin{proof}
  In order to use the semidefinite method, we need to fix~$\ell$, which is used to
  define set~$\calt_\ell$. Then we need to define~$c(T;T')$ for every~$T'\in
  \calt_\ell$ as in~\eqref{eq:semidefiniteshort}. To define~$c(T;T')$, we choose how
  many types~$m$ we will use and the types~$\sigma_t$ we want to use. For each
  type~$\sigma_t$, we choose an integer~$\ell_t$ satisfying~$\ell_t \leq (\ell +
  \lvert\sigma_t\rvert)/2$ and a positive
  semidefinite~$\lvert\calf_{\ell_t}^{\sigma_t}\rvert\times
  \lvert\calf_{\ell_t}^{\sigma_t}\rvert$ matrix~$Q(T,\sigma_t)$.

  Fix~$m = 3$, $\ell=5$, $\ell_1 = 3$, $\ell_2 = \ell_3 = 4$ and let~$\sigma_1 = 1$,
  $\sigma_2 = \Tr_3^*$ and~$\sigma_3 = \vec C_3^*$ be types as defined in
  Section~\ref{subsec:typesandflags} (see
  Figure~\eqref{fig:typesandflags}).

  Let~$Q(\tdoze,1)$, $Q(\tdoze,\Tr_3^*)$ and~$Q(\tdoze,\vec C_3^*)$ be the positive
  semidefinite matrices of orders~$\calf_3^1\times\calf_3^1$,
  $\calf_4^{\Tr_3^*}\times\calf_4^{\Tr_3^*}$ and~$\calf_4^{\vec
    C_3^*}\times\calf_4^{\vec C_3^*}$ respectively shown in
  Appendix~\ref{appendix:matrix} (note that~$\lvert\calf_3^1\rvert = 4$
  and~$\lvert\calf_4^{\Tr_3^*}\rvert = \lvert\calf_4^{\vec C_3^*}\rvert = 8$).

  To see that~$Q(\tdoze,1)$, $Q(\tdoze,\Tr_3^*)$ and~$Q(\tdoze,\vec C_3^*)$ are
  positive semidefinite, we analyse their characteristic
  polynomials~$p_{Q(\tdoze,1)}(x)$, $p_{Q(\tdoze,\Tr_3^*)}(x)$ and~$p_{Q(\tdoze,\vec
    C_3^*)}(x)$ shown in Appendix~\ref{appendix:polchar}. Since the only negative
  coefficients of these polynomials are all of odd order, it follows that all of
  their roots are non-negative, hence the matrices are positive semidefinite.

  We then compute~$p(\tdoze,T)$ and~$c(Q(\tdoze,\sigma_t);T)$ for
  every~$T\in\calt_5$ (see Figure~\ref{fig:tour5}) and every~$t\in[3]$.


  Finally, by Proposition~\ref{prop:semidefinite}, we have
  \begin{align*}
    \lim_{n \to \infty}p(\tdoze;T_n)
    & \leq \max_{T \in \calt_5}\{p(\tdoze;T)+c(\tdoze;T)\}
    = \frac{1}{16},
  \end{align*}
  where~$c(\tdoze;T)=c(Q(\tdoze,1);T)+c(Q(\tdoze,\Tr_3^*);T)+c(Q(\tdoze,\vec
  C_3^*);T)$ for every~$T\in\calt_5$.
\end{proof}

\begin{remark}
  All of the matrices in Appendix~\ref{appendix:matrix} were found with the aid of
  semidefinite programming solvers \hbox{CSDP}~\cite{CSDP} and
  \hbox{SDPA}~\cite{SDPAFamily}.

  Furthermore, the solution provided by these solvers was rounded to an exact
  solution using the rounding method described by Baber~\cite{thesiBaber}.

  Finally, the characteristic polynomials in Appendix~\ref{appendix:polchar} were
  found with the aid of the symbolic mathematics software
  \hbox{Maxima}~\cite{maxima}.
\end{remark}

The proofs of the upper bounds for~$\tsete$, $\toito$, $\tnove$ and~$\tonze$ are very
similar to the proof of Lemma~\ref{lemma:ub}. We choose how many types~$m$ we will
use and the types~$\sigma_i$ we want to use. For each type~$\sigma_i$, we choose an
integer~$\ell_i$ satisfying~$\ell_i\leq(\ell + \lvert\sigma_i\rvert)/2$ and find positive
semidefinite matrices~$Q_i=Q(T^j_5,\sigma_i)$.

The matrices and their characteristic polynomials are shown in
Appendix~\ref{appendix}. As in the proof of Lemma~\ref{lemma:ub}, the matrices are
easily seen to be positive semidefinite since the only negative coefficients of their
characteristic polynomials are all of odd order.


For each~$j\in\{7,8,9,11,12\}$, we then compute~$c(T_5^j;T) = \sum_{i=i}^m
c(Q(T_5^j,\sigma_i);T)$ and~$p(T_5^j;T)$, for every~$T\in\calt_\ell$, and obtain
the desired bounds according to the following tables.

\begin{footnotesize}
\noindent
  \begin{tabular}[t]{|c|c|}
  \hline  
  $\tsete$ & \\
  \hline
  $m$ & $3$ \\
  $\sigma_1$& $ 1 $\\
  $\sigma_2$  & $\Tr_3^*$\\
  $\sigma_3$ & $\vec C_3^*$\\
  $\ell$ & $5$\\
  $\ell_1$ & $3$\\
  $\ell_2$ & $4$ \\
  $\ell_3$ & $4$\\
  $Q_1$ & $Q(\tsete,1)$\\
  $Q_2$ & $Q(\tsete,\Tr_3^*)$\\
  $Q_3$ & $Q(\tsete,\vec C_3^*)$\\\hline
 \end{tabular}
  \begin{tabular}[t]{|c|c|}
  \hline  
  $\toito$ & \\
  \hline
  $m$ & $4$ \\
  $\sigma_1$& $ 1 $\\
  $\sigma_2$ & $A$\\
  $\sigma_3$  & $\Tr_3^*$\\
  $\sigma_4$ & $\vec C_3^*$\\
  $\ell$ & $5$\\
  $\ell_1$ & $3$\\
  $\ell_2$ & $3$ \\
  $\ell_3$ & $4$\\
  $\ell_4$ & $4$\\
  $Q_1$ & $Q(\toito,1)$\\
  $Q_2$ & $Q(\toito,A)$\\
  $Q_3$ & $Q(\toito,\Tr_3^*)$\\
  $Q_4$ & $Q(\toito,\vec C_3^*)$\\\hline
 \end{tabular}
 \begin{tabular}[t]{|c|c|}
  \hline  
  $\tnove$ & \\
  \hline
  $m$ & $2$ \\
  $\sigma_1$& $ 1 $\\
  $\sigma_2$ & $\vec C_3^*$\\
  $\ell$ & $5$\\
  $\ell_1$ & $3$\\
  $\ell_2$ & $4$ \\
  $Q_1$ & $Q(\tnove,1)$\\
  $Q_2$ & $Q(\tnove,\vec C_3^*)$\\\hline
 \end{tabular}
 \begin{tabular}[t]{|c|c|}
  \hline  
  $\tonze$ & \\
  \hline
  $m$ & $2$ \\
  $\sigma_1$& $ 1 $\\
  $\sigma_2$ & $\vec C_3^*$\\
  $\ell$ & $5$\\
  $\ell_1$ & $3$\\
  $\ell_2$ & $4$ \\
  $Q_1$ & $Q(\tonze,1)$\\
  $Q_2$ & $Q(\tonze,\vec C_3^*)$\\\hline
 \end{tabular}
  \begin{tabular}[t]{|c|c|}
  \hline  
  $\tdoze$ & \\
  \hline
  $m$ & $3$ \\
  $\sigma_1$& $ 1 $\\
  $\sigma_2$  & $\Tr_3^*$\\
  $\sigma_3$ & $\vec C_3^*$\\
  $\ell$ & $5$\\
  $\ell_1$ & $3$\\
  $\ell_2$ & $4$ \\
  $\ell_3$ & $4$\\
  $Q_1$ & $Q(\tdoze,1)$\\
  $Q_2$ & $Q(\tdoze,\Tr_3^*)$\\
  $Q_3$ & $Q(\tdoze,\vec C_3^*)$\\\hline
 \end{tabular}
\end{footnotesize}

		   
\section{Extracting more information from the semidefinite method}
\label{sec:ext}

In this section, we review some techniques in flag algebras to extract information
about extremal homomorphisms of Problem~\ref{prob:maxflag} from a tight solution of
the semidefinite program~\eqref{eq:semidefinite} (see also the more general
version~\eqref{eq:semidefiniteshort}). Again, we will work here only with the theory
of tournaments, but these techniques can be used in a more general setting.

The first technique is used to prove that the tournaments~$T'$ corresponding to
non-tight restrictions in~\eqref{eq:semidefinite} must have zero density in
the extremal homomorphisms.

\begin{proposition}\label{prop:slack}
  Let~$T\in\mathcal{T}$ be a tournament and let
  \begin{align*}
    c & = \max\{\phi(T) : \phi\in\Homp{0}\}.
  \end{align*}

  If~$\ell\geq\lvert T\rvert$ and~$g\in\Csem{0}$ are such that
  \begin{align*}
    \max\{p(T+g;T') : T'\in\mathcal{T}_\ell\} & = c,
  \end{align*}
  and~$\phi\in\Homp{0}$ is extremal for~$T$ (that is, if~$\phi(T) = c$), then
  \begin{align*}
    \phi(T') & = 0,
  \end{align*}
  for every~$T'\in\mathcal{T}_\ell$ such that~$p(T+g;T') < c$.
\end{proposition}

\begin{proof}
  Recall the semidefinite method from Subsection~\ref{subsec:semidefinite}. We know
  that
  \begin{align*}
    c & = \phi(T) \leq \phi(T+g) = \sum_{T'\in\mathcal{T}_\ell}p(T+g;T')\phi(T')\\
    & \leq
    \left(\max_{T'\in\mathcal{T}_\ell}p(T+g;T')\right)\sum_{T'\in\mathcal{T}_\ell}\phi(T')
    = c\phi(1_0) = c.
  \end{align*}
  
  Hence, we must have equality throughout. In particular, equality in the last
  inequality implies that
  \begin{align*}
    \sum_{T'\in\mathcal{T}_\ell}p(T+g;T')\phi(T') & = c\sum_{T'\in\mathcal{T}_\ell}\phi(T'),
  \end{align*}
  and since~$\phi(T')\geq 0$ for every~$T'\in\mathcal{T}_\ell$, we have
  \begin{align*}
    \phi(T')(c - p(T+g;T')) & = 0,
  \end{align*}
  for every~$T'\in\mathcal{T}_\ell$. Therefore, the result follows.
\end{proof}

For the next technique, we will need the notion of a homomorphism extension, so we
recall below the main theorem on the matter.

\begin{theorem}[Razborov~{\cite[Theorem~3.5]{Raz07}}]\label{thm:ext}
  If~$\sigma$ is a type and~$\phi\in\Homp{0}$ is a homomorphism such
  that~$\phi(\sigma)>0$, then there exists a random element~$\bm{\phi^\sigma}$
  of~$\Homp{\sigma}$ (called \emph{homomorphism extension}) such that
  \begin{align*}
    \expect{\bm{\phi^\sigma}(f)} & =
    \frac{\phi(\llbracket f\rrbracket_\sigma)}
         {\phi(\llbracket 1_\sigma\rrbracket_\sigma)},
  \end{align*}
  for every~$f\in\mathcal{A}^\sigma$.
\end{theorem}

The next technique says that if the element~$\llbracket F\cdot f^2\rrbracket_\sigma$
was used in a tight solution of~\eqref{eq:semidefinite}, then we must
have~$\bm{\phi^\sigma}(F\cdot f) = 0$ almost surely for every extremal
homomorphism~$\phi\in\Homp{0}$.

\begin{proposition}\label{prop:eigen}
  With the definitions and notation of Proposition~\ref{prop:semidefinite}, let
  \begin{align*}
    c & = \max\{\phi(T) : T\in\Homp{0}\},
  \end{align*}
  suppose that the optimum solution~$(Q^{(t)})_{t=1}^m$ of~\eqref{eq:semidefinite}
  has value~$c$ and write
  \begin{align*}
    Q^{(t)} & = \sum_{i=1}^{r_t}v^{(t)}_i(v^{(t)}_i)^\top,
  \end{align*}
  for every~$t\in[m]$.

  Under these circumstances, if~$\phi\in\Homp{0}$ is extremal for~$T$, that is,
  if~$\phi(T)=c$, then for every~$t\in[m]$ with~$\phi(\sigma_t)>0$ and
  every~$i\in[r_t]$, we have
  \begin{align*}
    \bm{\phi^{\sigma}}(F'_t\cdot F(v^{(t)}_i)) & = 0
  \end{align*}
  almost surely.
\end{proposition}

\begin{proof}
  Recall the semidefinite method from Subsection~\ref{subsec:semidefinite}. We know
  that
  \begin{align*}
    c = \phi(T) & \leq
    \phi(T) + \sum_{i=1}^m\phi(\llbracket F'_t\cdot F(Q^{(t)})\rrbracket_{\sigma_t})
    \\
    & =
    \sum_{T'\in\mathcal{T}_\ell}\left(
    p(T;T') +
    \sum_{i=1}^mp(\llbracket F'_t\cdot F(Q^{(t)})\rrbracket_{\sigma_t};T')
    \right)
    \phi(T')\\
    & \leq
    \max_{T'\in\mathcal{T}_\ell}\left(
    p(T;T') +
    \sum_{i=1}^mp(\llbracket F'_t\cdot F(Q^{(t)})\rrbracket_{\sigma_t};T')
    \right)
    \phi(1_0)\\
    & =
    c.
  \end{align*}

  Hence, we must have equality throughout. In particular, equality in the first
  inequality implies that
  \begin{align*}
    \sum_{i=1}^m\phi(\llbracket F'_t\cdot F(Q^{(t)})\rrbracket_{\sigma_t}) & = 0,
  \end{align*}
  and since~$\llbracket F'_t\cdot F(Q^{(t)})\rrbracket_{\sigma_t}\in\Csem{0}$ for
  every~$t\in[m]$, we get that
  \begin{align}\label{eq:brackQ}
    \phi(\llbracket F'_t\cdot F(Q^{(t)})\rrbracket_{\sigma_t}) & = 0,
  \end{align}
  for every~$t\in[m]$.

  Fix now~$t\in[m]$ such that~$\phi(\sigma_t)>0$ and recall that
  \begin{align*}
    \llbracket F'_t\cdot F(Q^{(t)})\rrbracket_{\sigma_t} & =
    \sum_{i=1}^{r_t}\llbracket F'_t\cdot F(v^{(t)}_i)^2\rrbracket_{\sigma_t}.
  \end{align*}
  This along with~\eqref{eq:brackQ} implies that
  \begin{align*}
    \phi(\llbracket F'_t\cdot F(v^{(t)}_i)^2\rrbracket_{\sigma_t}) & = 0.
  \end{align*}

  From Theorem~\ref{thm:ext}, we have
  \begin{align*}
    \expect{\bm{\phi^{\sigma_t}}(F'_t\cdot F(v^{(t)}_i)^2)} & = 0,
  \end{align*}
  and since this variable is (almost surely) non-negative, we get
  \begin{align*}
    \bm{\phi^{\sigma_t}}(F'_t\cdot F(v^{(t)}_i)) & = 0
  \end{align*}
  almost surely, as desired.
\end{proof}


	        
\section{Uniqueness}
\label{sec:unique}

In this section, we will prove the uniqueness results. Namely, we will prove that a
homomorphism~$\phi\in\Homp{0}$ maximizes the density of~$T_5^8$ if and only if~$\phi$
is the quasi-random homomorphism~$\phiqr$. We will also prove that~$\phi\in\Homp{0}$
maximizes the density of~$T_5^7$ or of~$T_5^{12}$ ($R_5$) if and only if~$\phi$ is
the carousel homomorphism~$\phiR$. Finally, we will also prove that~$\phi \in
\Homp{0}$ maximizes the density of~$T_5^9$ or of~$T_5^{11}$ if and only if~$\phi$ is
the limit of the sequence~$(\trianglen)_{n\in\mathbb{N}}$.

\subsection{Quasi-random uniqueness}
\label{subsec:qrunique}

First we recall the definition of the quasi-random homomorphism~$\phiqr\in\Homp{0}$
as the almost sure limit of the sequence of random
tournaments~$(\bm{R_{n,1/2}})_{n\in\mathbb{N}}$.
Alternatively, the quasi-random homomorphism is defined by
\begin{align*}
  \phiqr(T) & = \frac{\ell!}{\lvert\Aut(T)\rvert 2^{\binom{\ell}{2}}},
\end{align*}
for every tournament~$T$ of size~$\ell\in\mathbb{N}$, where~$\Aut(T)$ denotes the
group of automorphisms of the tournament~$T$.

We also recall the equivalence of the following quasi-random properties in the lemma
below.

\begin{lemma}[Chung--Graham~{\cite[Theorem~1]{CG:QuasiRandomTournaments}}]
  \label{lem:ChungGraham}
  Let~$\phi\in\Homp{0}$ be a homomorphism. The following are equivalent.
  \begin{enumerate}[label={$P_{\arabic*}$:}, ref={\ensuremath{P_{\arabic*}}}]
  \item $\phi = \phiqr$;
    \label{it:phiqr}
    \setcounter{enumi}{3}
  \item $\bm{\phi^A}(O^A+I^A) = 1/2$ a.s.
    \label{it:sameness}
  \end{enumerate}
\end{lemma}

\begin{remark}
  Although we will only use two quasi-random properties, let us mention that Chung
  and~Graham proved equivalence of a total of~11 quasi-random properties ($P_1$
  to~$P_{11}$).
\end{remark}

We are now in condition of proving that the density of~$T_5^8$ is maximized only by
the quasi-random homomorphism.

\begin{theorem}\label{thm:T8unique}
  If~$\phi\in\Homp{0}$ is a homomorphism, then
  \begin{align*}
    \phi(T_5^8) & \leq \frac{15}{128},
  \end{align*}
  with equality if and only if~$\phi=\phiqr$.
\end{theorem}

\begin{proof}
  By Lemma~\ref{lemma:lb_qr} and by Proposition~\ref{prop:semidefinite} (see also
  Section~\ref{sec:UB}), we know that
  \begin{align*}
    \max\{\phi(T_5^8) : \phi\in\Homp{0}\} & = \frac{15}{128} = \phiqr(T_5^8).
  \end{align*}

  Furthermore, we know that the matrices~$Q(T_5^8,1)$, $Q(T_5^8,A)$, $Q(T_5^8,\Tr_3^*)$
  and~$Q(T_5^8,\vec C_3^*)$ from the semidefinite method are an optimum solution with
  value~$15/128$.

  Since
  \begin{align*}
    Q(T_5^8,A) & = \frac{99}{3200}vv^\top,
  \end{align*}
  where~$v = (1,-1,-1,1)$ (indexed by~$(I^A,\vec C_3^A,\Tr_3^A,O^A)$),
  Proposition~\ref{prop:eigen} implies that if~$\phi\in\Homp{0}$ is such
  that~$\phi(T_5^8)=15/128$, then
  \begin{align*}
    \bm{\phi^A}(F(v)) & = \bm{\phi^A}(I^A - \vec C_3^A - \Tr_3^A + O^A) = 0\as
  \end{align*}

  Since~$\vec C_3^A+\Tr_3^A = 1_0 - O^A + I^A$, we get
  \begin{align*}
    \bm{\phi^A}(O^A+I^A) & = \frac{1}{2}\as,
  \end{align*}
  hence~$\phi$ satisfies Property~\ref{it:sameness} from
  Lemma~\ref{lem:ChungGraham}. Therefore~$\phi=\phiqr$.
\end{proof}

\subsection{Quasi-carousel uniqueness}
\label{subsec:qcunique}

First we recall the definition of the carousel homomorphism~$\phiR\in\Homp{0}$ as the
limit of the sequence~$(R_{2n+1})_{n\in\mathbb{N}}$ of carousel
tournaments. Analogously to quasi-random properties, the quasi-carousel
properties~\cite{Na15} are equivalent properties over a
homomorphism~$\phi\in\Homp{0}$ that force~$\phi=\phiR$. We recall two of the carousel
properties below.

\begin{lemma}[{\cite[Lemma~3.2]{Na15}}]\label{lem:QC}
  Let~$\phi\in\Homp{0}$ be a homomorphism. The following are equivalent.
  \begin{enumerate}[label={$S_{\arabic*}$:}, ref={\ensuremath{S_{\arabic*}}}]
  \item $\phi = \phiR$;
    \label{it:phiR}
  \item $\phi$ is balanced and locally transitive, that is, we have
    \begin{align*}
      \bm{\phi^1}(\alpha) & = \bm{\phi^1}(\beta)\as; &
      \phi(W_4+L_4) & = 0.
    \end{align*}
    \label{it:balloctran}
  \end{enumerate}
\end{lemma}

Furthermore, we will need an equivalence regarding balanced homomorphisms.

\begin{lemma}[Chung--Graham~{\cite[Theorem~2]{CG:QuasiRandomTournaments}}]
  \label{lem:balanced}
  Let~$\phi\in\Homp{0}$ be a homomorphism. The following are equivalent.
  \begin{enumerate}[label={$Q_{\arabic*}$:}, ref={\ensuremath{Q_{\arabic*}}}]
  \item $\phi(\Tr_3)=3/4$ and~$\phi(\vec C_3)=1/4$;
    \setcounter{enumi}{3}
  \item $\phi$ is balanced, that is, we have~$\bm{\phi^1}(\alpha) =
    \bm{\phi^1}(\beta)$ a.s.
  \end{enumerate}
\end{lemma}

Analogously to Theorem~\ref{thm:T8unique}, uniqueness for the carousel homomorphism
will follow from quasi-carousel Property~\ref{it:balloctran}.

\begin{theorem}\label{thm:T7unique}
  If~$\phi\in\Homp{0}$ is a homomorphism, then
  \begin{align*}
    \phi(T_5^7) & \leq \frac{5}{16},
  \end{align*}
  with equality if and only if~$\phi=\phiR$.
\end{theorem}

\begin{proof}
  By Lemma~\ref{lemma:lb} and by Proposition~\ref{prop:semidefinite} (see also
  Section~\ref{sec:UB}), we know that
  \begin{align*}
    \max\{\phi(T_5^7) : \phi\in\Homp{0}\} & = \frac{5}{16} = \phiR(T_5^7).
  \end{align*}

  Our goal is to prove that every~$\phi\in\Homp{0}$ such that~$\phi(T_5^7)=5/16$ is
  balanced and locally transitive.

  To prove that such~$\phi$ is balanced, we note that the matrices~$Q(T_5^7,1)$,
  $Q(T_5^7,\Tr_3^*)$ and~$Q(T_5^7,\vec C_3^*)$ from the semidefinite method are an
  optimum solution with value~$5/16$, and since
  \begin{align*}
    Q(T_5^7,1) & = \frac{35}{48}vv^\top,
  \end{align*}
  where~$v = (1,-1,-1,1)$ (indexed by~$(\Tr_3^{1,L},\vec
  C_3^1,\Tr_3^{1,M},\Tr_3^{1,W})$), Proposition~\ref{prop:eigen} implies that
  \begin{align*}
    \bm{\phi^1}(F(v)) & =
    \bm{\phi^1}(\Tr_3^{1,L} - \vec C_3^1 - \Tr_3^{1,M} + \Tr_3^{1,W}) =
    \bm{\phi^1}((\alpha-\beta)^2) =
    0\as
  \end{align*}
  Therefore~$\bm{\phi^1}(\alpha)=\bm{\phi^1}(\beta)$ a.s., that is, the
  homomorphism~$\phi$ is balanced.

  To prove that~$\phi$ is also locally transitive, we will use
  Proposition~\ref{prop:slack}. Table~\ref{tab:T7slack} has the values
  of~$p(T_5^7+g;T')$ for~$T'\in\mathcal{T}_5$ and where
  \begin{align*}
    g & = \llbracket F(Q(T_5^7,1))\rrbracket_1
    + \llbracket F(Q(T_5^7,\Tr_3^*))\rrbracket_{\Tr_3^*}
    + \llbracket F(Q(T_5^7,\vec C_3^*))\rrbracket_{\vec C_3^*}.
  \end{align*}

  \begin{table}[ht]
    \begin{center}
      \begin{tabular}{%
          >{$\displaystyle}c<{$}|*{12}{>{$\displaystyle}c<{$}}
        }
        T' & T_5^1 & T_5^2 & T_5^3 & T_5^4 & T_5^5 & T_5^6
        & T_5^7 & T_5^8 & T_5^9 & T_5^{10} & T_5^{11} & T_5^{12}
        \\
        \hline\noalign{\smallskip}
        p(T_5^7+g;T') &
        \frac{5}{16} & -\frac{7}{80} & \frac{11}{48} &
        -\frac{29}{240} & -\frac{7}{80} & \frac{11}{48} &
        \frac{5}{16} & -\frac{13}{48} & \frac{5}{16} &
        \frac{1}{16} & -\frac{109}{240} & \frac{5}{16}
      \end{tabular}
      \caption{Values~$p(T_5^7+g;T')$ for~$T'\in\mathcal{T}_5$ and where~$g =
        \llbracket F(Q(T_5^7,1))\rrbracket_1 + \llbracket
        F(Q(T_5^7,\Tr_3^*))\rrbracket_{\Tr_3^*} + \llbracket F(Q(T_5^7,\vec
        C_3^*))\rrbracket_{\vec C_3^*}$.}
      \label{tab:T7slack}
    \end{center}
  \end{table}

  Proposition~\ref{prop:slack} implies that if~$\phi(T')>0$
  for~$T'\in\mathcal{T}_5$, then~$T'\in\{T_5^1,T_5^7,T_5^9,T_5^{12}\}$, and since
  these four tournaments are the only locally transitive tournaments of size~$5$
  (i.e., the only tournaments~$T'\in\mathcal{T}_5$ with~$p(W_4+L_4;T') = 0$), we
  have~$\phi(W_4+L_4) = 0$, that is, the homomorphism~$\phi$ is locally transitive.

  Therefore~$\phi$ satisfies quasi-carousel Property~\ref{it:balloctran},
  hence~$\phi=\phiR$ by Lemma~\ref{lem:QC}.
\end{proof}

\begin{theorem}\label{thm:T12unique}
  If~$\phi\in\Homp{0}$ is a homomorphism, then
  \begin{align*}
    \phi(T_5^{12}) & \leq \frac{1}{16},
  \end{align*}
  with equality if and only if~$\phi=\phiR$.
\end{theorem}

\begin{proof}
  By Lemma~\ref{lemma:lb} and by Proposition~\ref{prop:semidefinite} (see also
  Section~\ref{sec:UB}), we know that
  \begin{align*}
    \max\{\phi(T_5^{12}) : \phi\in\Homp{0}\} & = \frac{1}{16} = \phiR(T_5^{12}).
  \end{align*}

  Again, our goal is to prove that every~$\phi\in\Homp{0}$ such
  that~$\phi(T_5^{12})=5/16$ is balanced and locally transitive.

  To prove that such~$\phi$ is balanced, we note that the matrices~$Q(T_5^{12},1)$,
  $Q(T_5^{12},\Tr_3^*)$ and~$Q(T_5^{12},\vec C_3^*)$ from the semidefinite method are an
  optimum solution with value~$1/16$, and since
  \begin{align*}
    Q(T_5^{12},1) & = \frac{1}{16}vv^\top,
  \end{align*}
  where~$v = (1,-1,-1,1)$ (indexed by~$(\Tr_3^{1,L},\vec
  C_3^1,\Tr_3^{1,M},\Tr_3^{1,W})$), Proposition~\ref{prop:eigen} implies that
  \begin{align*}
    \bm{\phi^1}(F(v)) & =
    \bm{\phi^1}(\Tr_3^{1,L} - \vec C_3^1 - \Tr_3^{1,M} + \Tr_3^{1,W}) =
    \bm{\phi^1}((\alpha-\beta)^2)
    0\as
  \end{align*}
  Therefore~$\bm{\phi^1}(\alpha)=\bm{\phi^1}(\beta)$ a.s., that is, the
  homomorphism~$\phi$ is balanced.

  To prove that~$\phi$ is also locally transitive, we will use again
  Proposition~\ref{prop:slack}. Table~\ref{tab:T12slack} has the values
  of~$p(T_5^{12}+g;T')$ for~$T'\in\mathcal{T}_5$ and where
  \begin{align*}
    g & = \llbracket F(Q(T_5^{12},1))\rrbracket_1
    + \llbracket F(Q(T_5^{12},\Tr_3^*))\rrbracket_{\Tr_3^*}
    + \llbracket F(Q(T_5^{12},\vec C_3^*))\rrbracket_{\vec C_3^*}.
  \end{align*}

  \begin{table}[ht]
    \begin{center}
      \begin{tabular}{%
          >{$\displaystyle}c<{$}|*{12}{>{$\displaystyle}c<{$}}
        }
        T' & T_5^1 & T_5^2 & T_5^3 & T_5^4 & T_5^5 & T_5^6
        & T_5^7 & T_5^8 & T_5^9 & T_5^{10} & T_5^{11} & T_5^{12}
        \\
        \hline\noalign{\smallskip}
        p(T_5^{12}+g;T')
        & \frac{1}{16} & \frac{1}{80} & \frac{1}{16}
        & -\frac{3}{16} & \frac{1}{80} & \frac{1}{16}
        & \frac{1}{16} & \frac{1}{16} & \frac{1}{16}
        & \frac{1}{16} & -\frac{39}{80} & \frac{1}{16}
      \end{tabular}
      \caption{Values~$p(T_5^{12}+g;T')$ for~$T'\in\mathcal{T}_5$ and where~$g =
        \llbracket F(Q(T_5^{12},1))\rrbracket_1 + \llbracket
        F(Q(T_5^{12},\Tr_3^*))\rrbracket_{\Tr_3^*} + \llbracket F(Q(T_5^{12},\vec
        C_3^*))\rrbracket_{\vec C_3^*}$.}
      \label{tab:T12slack}
    \end{center}
  \end{table}

  Proposition~\ref{prop:slack} implies that~$\phi(T_5^2+T_5^5) = 0$.

  Now, since we have
  \begin{align*}
    \llbracket (L_4^{\vec C_3^*})^2\rrbracket_{\vec C_3^*} & = \frac{1}{20}T_5^2; &
    \llbracket (W_4^{\vec C_3^*})^2\rrbracket_{\vec C_3^*} & = \frac{1}{20}T_5^5;
  \end{align*}
  and since~$\phi$ is balanced, by Lemma~\ref{lem:balanced}, we have~$\phi(\vec
  C_3)=1/4$, hence
  \begin{align*}
    \expect{\bm{\phi^{\vec C_3^*}}(W_4^{\vec C_3^*})^2 +
      \bm{\phi^{\vec C_3^*}}(L_4^{\vec C_3^*})^2} & =
    \frac{1}{10}\cdot\frac{\phi(T_5^2 + T_5^5)}{\phi(\vec C_3)} = 0,
  \end{align*}
  which implies that~$\bm{\phi^{\vec C_3^*}}(W_4^{\vec C_3^*}+L_4^{\vec C_3^*})=0$
  a.s.

  This in turn implies that
  \begin{align*}
    0 & = \expect{\bm{\phi^{\vec C_3^*}}(W_4^{\vec C_3^*}+L_4^{\vec C_3^*})}
    = \frac{1}{4}\cdot\frac{\phi(W_4+L_4)}{\phi(\vec C_3)},
  \end{align*}
  hence~$\phi(W_4+L_4)=0$, that is, the homomorphism~$\phi$ is locally transitive.

  Therefore~$\phi$ satisfies quasi-carousel Property~\ref{it:balloctran},
  hence~$\phi=\phiR$ by Lemma~\ref{lem:QC}.
\end{proof}

\subsection{Quasi-triangular uniqueness}
\label{subsec:qtunique}

We start by defining a~$\vec C_3$-decomposable tournament inductively, which
intuitively are tournaments similar in structure to~$\trianglen$, but without
requiring the ``blow-up'' to have parts as balanced as possible.

\begin{definition}\label{def:C3dec}
  Define the sequence of sets~$(\mathcal{B}_n)_{n\in\mathbb{N}}$ inductively as
  follows.

  Let~$\mathcal{B}_0=\mathcal{T}_0$ and~$\mathcal{B}_1=\mathcal{T}_1$ and for~$n\geq
  2$, let~$\mathcal{B}_n\subset\mathcal{T}_n$ be the set of all tournaments~$T$ of
  size~$n$ such that there exist sets~$A$, $B$ and~$C$ such
  that
  \begin{enumerate}[label={\roman*.}, ref={(\roman*)}]
  \item The sets~$A$, $B$ and~$C$ are strictly contained in~$V(T)$, that is, we
    have~$A,B,C\subsetneq V(T)$;
    \label{it:strictsubset}
  \item The sets~$A$, $B$ and~$C$ are pairwise disjoint;
    \label{it:pairwisedisjoint}
  \item We have~$V(T) = A\cup B\cup C$;
    \label{it:union}
  \item We have~$T\vert_A\in\mathcal{B}_{\lvert A\rvert}$,
    $T\vert_B\in\mathcal{B}_{\lvert B\rvert}$ and~$T\vert_C \in\mathcal{B}_{\lvert
    C\rvert}$;
    \label{it:recursion}
  \item We have~$A\times B, B\times C, C\times A\subset A(T)$.
    \label{it:C3}
  \end{enumerate}

  Finally, we say that a tournament~$T$ of size~$n$ is~\emph{$\vec
    C_3$-decomposable} if~$T\in\mathcal{B}_n$.
\end{definition}

\begin{remark}
  Note that items~\ref{it:strictsubset}, \ref{it:pairwisedisjoint} and~\ref{it:union}
  together say that~$\{A,B,C\}\setminus\{\varnothing\}$ is a partition of~$V(T)$ into
  either two or three sets.

  Furthermore, note that item~\ref{it:pairwisedisjoint} actually follows from
  item~\ref{it:C3}.

  Finally, note that item~\ref{it:recursion} is well-defined since~$\max\{\lvert
  A\rvert,\lvert B\rvert, \lvert C\rvert\} < n$ (due to item~\ref{it:strictsubset}).
\end{remark}

The next theorem provides a characterization of~$\vec C_3$-decomposable theorems as
the class of tournaments avoiding~$T_5^8$, $T_5^{10}$ and~$T_5^{12}$. We defer the
proof of this theorem to Section~\ref{sec:proofdec}.

\begin{theorem}\label{thm:C3deccharac}
  A tournament~$T$ is~$\vec C_3$-decomposable if and only if it has no copies
  of~$T_5^8$, $T_5^{10}$ nor of~$T_5^{12}$.
\end{theorem}

Motivated by the theorem above let us say that a homomorphism~$\phi\in\Homp{0}$ in
the theory of tournaments is~$\vec C_3$-decomposable
if~$\phi(T_5^8+T_5^{10}+T_5^{12})=0$.

Note that the fact that a sequence of tournaments~$(T_n)_{n\in\mathbb{N}}$ converges
to a~$\vec C_3$-decomposable homomorphism does \emph{not} imply that any of the
tournaments is~$\vec C_3$-decomposable. Rather, it only implies that the densities of
the tournaments~$T_5^8$, $T_5^{10}$ and~$T_5^{12}$ in~$T_n$ go to zero as~$n$ goes to
infinity.

We now define the notion of a~$k$-equally~$\vec C_3$-decomposable tournament
inductively.

\begin{definition}
  A tournament~$T$ is~\emph{$0$-equally~$\vec C_3$-decomposable} if it is~$\vec
  C_3$-decomposable.

  For~$k>0$, a tournament~$T$ is~\emph{$k$-equally~$\vec C_3$-decomposable} if
  either~$\lvert T\rvert\leq 1$ or there exists~$(A,B,C)$ as in
  Definition~\ref{def:C3dec} satisfying also the following properties.
  \begin{enumerate}[label={\alph*.}, ref={(\alph*)}]
  \item We have
    \begin{align*}
        \max\{\lvert A\rvert, \lvert B\rvert, \lvert C\rvert\} -
        \min\{\lvert A\rvert, \lvert B\rvert, \lvert C\rvert\} & \leq
        1;
    \end{align*}
      \label{it:topequally}
    \item The tournaments~$T\vert_A$, $T\vert_B$ and~$T\vert_C$
      are~$(k-1)$-equally~$\vec C_3$-decomposable.
      \label{it:recursionequally}
  \end{enumerate}

\end{definition}

Trivially, every~$k$-equally~$\vec C_3$-decomposable tournament is
also~$(k-1)$-equally~$\vec C_3$-decomposable.

Note also that if~$n\leq 3^k$, then the only~$k$-equally~$\vec C_3$-decomposable
tournament of size~$n$ is~$\trianglen$.
We claim now that the sequence~$(\trianglen)_{n\in\mathbb{N}}$ is convergent, but we
defer the proof of this claim. We will call the limit of this sequence
the~\emph{triangular homomorphism} and denote it by~$\phiC$.

The next theorem states the equivalence of what we could call \emph{quasi-triangular
  properties}. The equivalence of Properties~\ref{it:QTphiC3}, \ref{it:QTmaxT9}
and~\ref{it:QTmaxT11} imply that~$\phiC$ is the only homomorphism that maximizes the
density of~$T_5^9$ and is the only homomorphism that maximizes the density
of~$T_5^{11}$.

\begin{theorem}\label{thm:QT}
  If~$\phi\in\Homp{0}$ is a homomorphism in the theory of tournaments, then the
  following are equivalent.
  \begin{enumerate}[label={$L_{\arabic*}:$}, ref={\ensuremath{L_{\arabic*}}}]
  \item $\phi=\phiC$;
    \label{it:QTphiC3}
  \item $\phi$ maximizes the density of~$T_5^9$, that is, we have
    \begin{align*}
      \phi(T_5^9) & = \max\{\psi(T_5^9) : \psi\in\Homp{0}\};
    \end{align*}
    \label{it:QTmaxT9}
  \item $\phi$ maximizes the density of~$T_5^{11}$, that is, we have
    \begin{align*}
      \phi(T_5^{11}) & = \max\{\psi(T_5^{11}) : \psi\in\Homp{0}\};
    \end{align*}
    \label{it:QTmaxT11}
  \item $\phi$ is balanced and~$\vec C_3$-decomposable, that is, we have
    \begin{align*}
      \bm{\phi^1}(\alpha) & = \bm{\phi^1}(\beta)\as; &
      \phi(T_5^8+T_5^{10}+T_5^{12}) & = 0;
    \end{align*}
    \label{it:QTbalC3dec}
  \item For every~$k\in\mathbb{N}$, there exists a
    sequence~$(T^{(k)}_n)_{n\in\mathbb{N}}$ of~$k$-equally~$\vec C_3$-decomposable
    tournaments that converges to~$\phi$.
    \label{it:QTseqskequallyC3dec}
  \end{enumerate}
\end{theorem}

We will prove Theorem~\ref{thm:QT} through a series of lemmas. We have already proved
in Sections~\ref{sec:LB} and~\ref{sec:UB}
that~$\ref{it:QTphiC3}\implies\ref{it:QTmaxT9}\land\ref{it:QTmaxT11}$.

The next two lemmas follow from the techniques presented in Section~\ref{sec:ext}.

\begin{lemma}
  We have~$\ref{it:QTmaxT9}\implies\ref{it:QTbalC3dec}$.
\end{lemma}

\begin{proof}
  By Lemma~\ref{lemma:lb_triangle} and by Proposition~\ref{prop:semidefinite} (see
  also Section~\ref{sec:UB}), we know that
  \begin{align*}
    \max\{\psi(T_5^9) : \psi\in\Homp{0}\} & = \frac{3}{8},
  \end{align*}
  and that the matrices~$Q(T_5^9,1)$ and~$Q(T_5^9,\vec C_3^*)$ from the semidefinite
  method are an optimum solution with value~$3/8$.

  Let then~$\phi\in\Homp{0}$ be a homomorphism that maximizes the density
  of~$T_5^9$.

  Let us prove that~$\phi$ is~$\vec C_3$-decomposable. To do this, we will use
  Proposition~\ref{prop:slack}. Table~\ref{tab:T9slack} has the values
  of~$p(T_5^9+g;T')$ for~$T'\in\mathcal{T}_5$ and where
  \begin{align*}
    g & = \llbracket F(Q(T_5^9,1))\rrbracket_1
    + \llbracket F(Q(T_5^9,\vec C_3^*))\rrbracket_{\vec C_3^*}.
  \end{align*}

  \begin{table}[ht]
    \begin{center}
      \begin{tabular}{%
          >{$\displaystyle}c<{$}|*{12}{>{$\displaystyle}c<{$}}
        }
        T' & T_5^1 & T_5^2 & T_5^3 & T_5^4 & T_5^5 & T_5^6
        & T_5^7 & T_5^8 & T_5^9 & T_5^{10} & T_5^{11} & T_5^{12}
        \\
        \hline\noalign{\smallskip}
        p(T_5^9+g;T')
        & \frac{3}{8} & \frac{3}{8} & \frac{3}{8}
        & \frac{3}{8} & \frac{3}{8} & \frac{3}{8}
        & \frac{3}{8} & -\frac{289}{200} & \frac{3}{8}
        & \frac{3}{200} & \frac{3}{8} & \frac{11}{40}
      \end{tabular}
      \caption{Values~$p(T_5^9+g;T')$ for~$T'\in\mathcal{T}_5$ and where~$g =
        \llbracket F(Q(T_5^9,1))\rrbracket_1 + \llbracket F(Q(T_5^9,\vec
        C_3^*))\rrbracket_{\vec C_3^*}$.}
      \label{tab:T9slack}
    \end{center}
  \end{table}

  Proposition~\ref{prop:slack} implies that~$\phi(T_5^8+T_5^{10}+T_5^{12})=0$, that
  is, the homomorphism~$\phi$ is~$\vec C_3$-decomposable.

  It remains only to prove that~$\phi$ is balanced.

  We first note that the matrix~$Q(T_5^9,1)$ has eigenvectors
  \begin{align*}
    v_1 & = \left(
    \begin{array}{*{4}{>{\displaystyle}c}}
      1, &
      \frac{-16+\sqrt{179}}{7}, &
      \frac{2-\sqrt{179}}{7}, &
      1
    \end{array}
    \right);
    \\
    v_2 & = \left(
    \begin{array}{*{4}{>{\displaystyle}c}}
    1, &
    \frac{-16-\sqrt{179}}{7}, &
    \frac{2+\sqrt{179}}{7}, &
    1
    \end{array}
    \right)
  \end{align*}
  (indexed by~$(\Tr_3^{1,L},\vec C_3^1,\Tr_3^{1,M},\Tr_3^{1,W})$) with the
  eigenvalues~$12(16 + \sqrt{179})$ and~$12(16 - \sqrt{179})$ respectively.

  By Proposition~\ref{prop:eigen}, we know that~$\bm{\phi^1}(F(v_1)) =
  \bm{\phi^1}(F(v_2)) = 0$ a.s. This implies that
  \begin{align*}
    0 & = \expect{\bm{\phi^1}(F(v_1) - F(v_2))}
    =
    \frac{2\sqrt{179}}{7}
    \expect{\bm{\phi^1}(\vec C_3^1 - \Tr_3^{1,M})}
    =
    \frac{2\sqrt{179}}{7}
    \left(
    \phi(\vec C_3)
    - \frac{1}{3}\phi(\Tr_3)
    \right).
  \end{align*}

  Hence~$\phi(\vec C_3) = 1/4$ (since~$\vec C_3 + \Tr_3 = 1_0$), which by
  Lemma~\ref{lem:balanced}, implies that~$\phi$ is balanced.

  Therefore~$\phi$ satisfies~\ref{it:QTbalC3dec}.
\end{proof}

\begin{lemma}
  We have~$\ref{it:QTmaxT11}\implies\ref{it:QTbalC3dec}$.
\end{lemma}

\begin{proof}
  By Lemma~\ref{lemma:lb_triangle} and by Proposition~\ref{prop:semidefinite} (see
  also Section~\ref{sec:UB}), we know that
  \begin{align*}
    \max\{\psi(T_5^{11}) : \psi\in\Homp{0}\} & = \frac{1}{16},
  \end{align*}
  and that the matrices~$Q(T_5^{11},1)$ and~$Q(T_5^{11},\vec C_3^*)$ from the
  semidefinite method are an optimum solution with value~$1/16$.

  Let then~$\phi\in\Homp{0}$ be a homomorphism that maximizes the density
  of~$T_5^{11}$.

  Since
  \begin{align*}
    Q(T_5^{11},1) & = \frac{5}{16}vv^\top,
  \end{align*}
  where~$v = (1,-1,-1,1)$ (indexed by~$(\Tr_3^{1,L},\vec
  C_3^1,\Tr_3^{1,M},\Tr_3^{1,W})$), Proposition~\ref{prop:eigen} implies that
  \begin{align*}
    \bm{\phi^1}(F(v)) & =
    \bm{\phi^1}(\Tr_3^{1,L} - \vec C_3^1 - \Tr_3^{1,M} + \Tr_3^{1,W}) =
    \bm{\phi^1}((\alpha-\beta)^2)
    0\as
  \end{align*}
  Therefore~$\bm{\phi^1}(\alpha)=\bm{\phi^1}(\beta)$ a.s., that is, the
  homomorphism~$\phi$ is balanced.

  It remains only to prove that~$\phi$ is~$\vec C_3$-decomposable. To do this, we
  will use Proposition~\ref{prop:slack}. Table~\ref{tab:T11slack} has the values
  of~$p(T_5^{11}+g;T')$ for~$T'\in\mathcal{T}_5$ and where
  \begin{align*}
    g & = \llbracket F(Q(T_5^{11},1))\rrbracket_1
    + \llbracket F(Q(T_5^{11},\vec C_3^*))\rrbracket_{\vec C_3^*}.
  \end{align*}

  \begin{table}[ht]
    \begin{center}
      \begin{tabular}{%
          >{$\displaystyle}c<{$}|*{12}{>{$\displaystyle}c<{$}}
        }
        T' & T_5^1 & T_5^2 & T_5^3 & T_5^4 & T_5^5 & T_5^6
        & T_5^7 & T_5^8 & T_5^9 & T_5^{10} & T_5^{11} & T_5^{12}
        \\
        \hline\noalign{\smallskip}
        p(T_5^{11}+g;T')
        & \frac{1}{16} & \frac{1}{16} & \frac{1}{16}
        & \frac{1}{16} & \frac{1}{16} & \frac{1}{16}
        & \frac{1}{16} & -\frac{3}{400} & \frac{1}{16} 
        & -\frac{23}{400} & \frac{1}{16} & \frac{1}{80}
      \end{tabular}
      \caption{Values~$p(T_5^{11}+g;T')$ for~$T'\in\mathcal{T}_5$ and where~$g =
        \llbracket F(Q(T_5^{11},1))\rrbracket_1 + \llbracket F(Q(T_5^{11},\vec
        C_3^*))\rrbracket_{\vec C_3^*}$.}
      \label{tab:T11slack}
    \end{center}
  \end{table}

  Proposition~\ref{prop:slack} implies that~$\phi(T_5^8+T_5^{10}+T_5^{12})=0$, that
  is, the homomorphism~$\phi$ is~$\vec C_3$-decomposable.

  Therefore~$\phi$ satisfies~\ref{it:QTbalC3dec}.
\end{proof}

For the next two implications, we will need to use the notion of a~$\vec
C_3$-decomposition of a~$\vec C_3$-decomposable tournament. To make it precise, let
us first fix some notation. Let
\begin{align*}
  \Sigma^* & = \{(\sigma_i)_{i=1}^k :
  k\in\mathbb{N}\land\forall i\in[k],\sigma_i\in[3]\}
\end{align*}
denote the set of all finite sequences of elements in~$[3] = \{1,2,3\}$ (and let us denote the
empty sequence by~$\oldepsilon$).

As usual, we will denote by~$\sigma\tau$ the sequence obtained by concatenating~$\tau\in\Sigma^*$
to the end of~$\sigma\in\Sigma^*$ and we will denote the length of a
sequence~$\sigma\in\Sigma^*$ by~$\lvert\sigma\rvert$.

\begin{definition}
  Let~$T$ be a~$\vec C_3$-decomposable tournament. A~\emph{$\vec C_3$-decomposition}
  of~$T$ is a family of sets~$A = (A_\sigma)_{\sigma\in\Sigma^*}$ indexed by~$\Sigma^*$
  such that
  \begin{enumerate}[label={\roman*.}, ref={(\roman*)}]
  \item We have~$A_\oldepsilon = V(T)$;
  \item For every~$\sigma\in\Sigma^*$ such that~$\lvert A_\sigma\rvert\geq 2$, the
    triple~$(A_{\sigma1},A_{\sigma2},A_{\sigma3})$ satisfies the items in
    Definition~\ref{def:C3dec} for~$T\vert_{A_\sigma}$;
  \item For every~$\sigma\in\Sigma^*$ such that~$\lvert A_\sigma\rvert\leq 1$, the
    sets~$A_{\sigma1}$, $A_{\sigma2}$ and~$A_{\sigma3}$ are pairwise disjoint
    and~$A_{\sigma1}\cup A_{\sigma2}\cup A_{\sigma3} = A_\sigma$.
  \end{enumerate}

  For every~$k\in\mathbb{N}$, the~\emph{$k$-th level} of the~$\vec
  C_3$-decomposition~$A$ is the family of sets~$A_\sigma$ such
  that~$\lvert\sigma\rvert=k$. The \emph{skewness} of the~$k$-th level of~$A$
  (denoted~$\Delta_k(A)$) is defined as
  \begin{align*}
    \Delta_k(A) & = 
    \max\{\lvert A_\sigma\rvert : \sigma\in\Sigma^*\land\lvert\sigma\rvert=k\} -
    \min\{\lvert A_\sigma\rvert : \sigma\in\Sigma^*\land\lvert\sigma\rvert=k\}.
  \end{align*}
\end{definition}

Note that a tournament is~$k$-equally~$\vec C_3$-decomposable if and only if it has
a~$\vec C_3$-decomposition~$A = (A_\sigma)_{\sigma\in\Sigma^*}$ such
that~$\Delta_\ell(A)\leq 1$ for every~$\ell\leq k$.

Let us now define some notation on tournaments.

\begin{definition}
  Let~$T$ be a tournament and~$A\subset V(T)$. We define
  \begin{align*}
    N^+(A) & = \{v\in V(T) : \forall a\in A, av\in A(T)\};\\
    N^-(A) & = \{v\in V(T) : \forall a\in A, va\in A(T)\}.
  \end{align*}
\end{definition}

Note that~$N^+(A)\cup N^-(A)$ is always a subset of~$V(T)\setminus A$ and may be a
proper subset.

We now prove two basic facts about tournaments.

\begin{lemma}\label{lem:weakbal}
  If~$(T_n)_{n\in\mathbb{N}}$ is a sequence of tournaments
  with~$\lim_{n\to\infty}\lvert T_n\rvert=\infty$ and~$c\geq 1/2$ is a constant such
  that all but~$o(\lvert T_n\rvert)$ vertices of~$T_n$ have indegree greater
  than~$(c+o(1))\lvert T_n\rvert$, then~$c=1/2$.
\end{lemma}

\begin{proof}
  Let~$(T_n')_{n\in\mathbb{N}}$ be a convergent subsequence
  of~$(T_n)_{n\in\mathbb{N}}$ and let~$\phi\in\Homp{0}$ be its limit and note
  that~$\bm{\phi^1}(\beta)\geq c$ a.s.

  Since~$\expect{\bm{\phi^1}(\beta)} = 1/2$, we get~$c\leq 1/2$.
\end{proof}

\begin{lemma}\label{lem:balneigh}
  Let~$(T_n)_{n\in\mathbb{N}}$ be a sequence of tournaments converging to a balanced
  homomorphism~$\phi$ and for every~$n\in\mathbb{N}$, let~$A_n\subset V(T_n)$ be such
  that~$\lvert A_n\rvert = \Omega(\lvert T_n\rvert)$.

  Under these circumstances, if~$N^+(A_n)\cup N^-(A_n) = V(T_n)\setminus A_n$ for
  every~$n\in\mathbb{N}$, then
  \begin{align*}
    \lvert N^+(A_n)\rvert - \lvert N^-(A_n)\rvert & = o(\lvert T_n\rvert).
  \end{align*}
\end{lemma}

\begin{proof}
  Suppose not. This means that by passing to a subsequence and possibly flipping all
  arcs, we may suppose that there exists~$\epsilon > 0$ such that
  \begin{align*}
    \lvert N^+(A_n)\rvert - \lvert N^-(A_n)\rvert & \geq \epsilon \lvert T_n\rvert,
  \end{align*}
  for every~$n\in\mathbb{N}$.

  Note that if~$v\in A_n$, then we have
  \begin{align*}
    \lvert T_n\rvert - 2d^-(v)
    & \geq
    d^+(v) - d^-(v) + 1
    \\
    & =
    \lvert N^+(A_n)\rvert + d^+_{A_n}(v) - \lvert N^-(A_n)\rvert - d^-_{A_n}(v) + 1
    \\
    & \geq
    \epsilon\lvert T_n\rvert + d_{A_n}^+(v) - d^-_{A_n}(v) + 1
    \\
    & =
    \epsilon\lvert T_n\rvert + \lvert A_n\rvert - 2d^-_{A_n}(v)
    \\
    & \geq
    (1+\epsilon)\lvert A_n\rvert - 2d^-_{A_n}(v),
  \end{align*}
  where~$d^+_A(v) = \lvert N^+(v)\cap A\rvert$ and~$d^-_A(v) = \lvert N^-(v)\cap
  A\rvert$.

  Since~$\phi$ is balanced, we know that all but~$o(\lvert T_n\rvert)$ vertices
  of~$A_n$ have outdegree~$(1/2+o(1))\lvert T_n\rvert$, hence, since~$\lvert
  A_n\rvert=\Omega(\lvert T_n\rvert)$, if~$v$ is one such vertex, we have
  \begin{align*}
    d_{A_n}^-(v) & \geq \left(\frac{1 + \epsilon}{2} + o(1)\right)\lvert A_n\rvert.
  \end{align*}

  But this contradicts Lemma~\ref{lem:weakbal} for the
  sequence~$(T_n\vert_{A_n})_{n\in\mathbb{N}}$.
\end{proof}

The next technical lemma says that if a sequence of~$\vec C_3$-decomposable
tournaments converges to a balanced homomorphism, then we may suppose that at least
two of~$A^{(n)}_1$, $A^{(n)}_2$ or~$A^{(n)}_3$ have non-negligible size. We defer the
proof of this lemma to Section~\ref{sec:prooftwononneg}.

\begin{lemma}\label{lem:twononneg}
  If~$(T_n)_{n\in\mathbb{N}}$ is a sequence of~$\vec C_3$-decomposable tournaments
  that converges to a balanced homomorphism~$\phi\in\Homp{0}$, then there
  exists a sequence~$(T_n')_{n\in\mathbb{N}}$ of~$\vec C_3$-decomposable tournaments
  and for every~$n\in\mathbb{N}$ a~$\vec
  C_3$-decomposition~$(A^{(n)}_\sigma)_{\sigma\in\Sigma^*}$ of~$T_n'$ such that
  \begin{itemize}
  \item There exists a subsequence~$(T_{k_n})_{n\in\mathbb{N}}$
    of~$(T_n)_{n\in\mathbb{N}}$ such that the tournament~$T_n'$ can be obtained
    from~$T_{k_n}$ by flipping~$o(\lvert T_{k_n}\rvert^2)$ arcs
    (hence~$(T_n')_{n\in\mathbb{N}}$ also converges to~$\phi$);
  \item We have~$\lvert A^{(n)}_1\rvert = \Omega(\lvert T_n'\rvert)$ and~$\lvert
    A^{(n)}_2\rvert = \Omega(\lvert T_n'\rvert)$.
  \end{itemize}
\end{lemma}

\begin{lemma}\label{lem:QTbalC3dec->QTseqskequallyC3dec}
  We have~$\ref{it:QTbalC3dec}\implies\ref{it:QTseqskequallyC3dec}$.
\end{lemma}

\begin{proof}
  Suppose that~$\phi$ is balanced and~$\vec C_3$-decomposable. Let~$T_{\vec C_3}$ be
  the universal theory of~$\vec C_3$-decomposable tournaments, that is, the theory of
  tournaments that have no copy of~$T_5^8$, $T_5^{10}$ nor of~$T_5^{12}$ and note
  that~$\phi$ can also be thought of as an element of~$\Hom^+(\mathcal{A}^0[T_{\vec
      C_3}],\mathbb{R})$. This means that there exists a
  sequence~$(T_n^{(0)})_{n\in\mathbb{N}}$ of~$\vec C_3$-decomposable tournaments that
  converges to~$\phi$ (which is, by definition, a sequence of~$0$-equally~$\vec
  C_3$-decomposable tournaments). Furthermore, we may also suppose without loss of
  generality that~$\lvert T_n^{(0)}\rvert$ is a power of~$3$ for
  every~$n\in\mathbb{N}$.

  Let us now construct by induction in~$k$ the
  sequences~$(T_n^{(k)})_{n\in\mathbb{N}}$ of~$k$-equally~$\vec C_3$-decomposable
  tournaments converging to~$\phi$ and preserving the property that~$\lvert
  T_n^{(k)}\rvert$ is a power of~$3$ for every~$n\in\mathbb{N}$.

  Suppose~$k>0$ and that we have already
  constructed~$(T_n^{(k-1)})_{n\in\mathbb{N}}$. Applying Lemma~\ref{lem:twononneg} a
  total of~$3^{k-1}$ times to the tournaments induced by the~$(k-1)$-th level of
  the~$\vec C_3$-decompositions of the~$T_n^{(k-1)}$, we know that there exists a
  sequence~$(T_n')_{n\in\mathbb{N}}$ of~$\vec C_3$-decomposable tournaments and for
  every~$n\in\mathbb{N}$ there is a~$\vec C_3$-decomposition~$A^{(n)} =
  (A^{(n)}_\sigma)_{\sigma\in\Sigma^*}$ of~$T_n'$ such that
  \begin{itemize}
  \item For every~$t\leq k-1$, we have~$\Delta_t(A^{(n)}) = 0$;
  \item There exists a subsequence~$(T_{m_n}^{(k-1)})_{n\in\mathbb{N}}$
    of~$(T_n^{(k-1)})_{n\in\mathbb{N}}$ such that the tournament~$T_n'$ can be
    obtained from~$T_{m_n}^{(k-1)}$ by flipping~$o(\lvert T_{m_n}^{(k-1)}\rvert^2)$
    arcs, all completely contained within one of the sets~$A^{(n)}_\sigma$ for
    some~$\sigma\in\Sigma^*$ with~$\lvert\sigma\rvert = k-1$;
  \item For every~$\sigma\in\Sigma^*$ with~$\lvert\sigma\rvert = k-1$, we have
    \begin{align*}
      \lvert A^{(n)}_{\sigma1}\rvert & = \Omega(\lvert T_n'\rvert);
      &
      \lvert A^{(n)}_{\sigma2}\rvert & = \Omega(\lvert T_n'\rvert);
    \end{align*}
  \end{itemize}

  Fix~$\sigma\in\Sigma^*$ with~$\lvert\sigma\rvert = k-1$ and note that
  \begin{align*}
    N^+(A^{(n)}_{\sigma1})\cup N^-(A^{(n)}_{\sigma1}) & = V(T_n')\setminus A^{(n)}_{\sigma1};\\
    N^+(A^{(n)}_{\sigma2})\cup N^-(A^{(n)}_{\sigma2}) & = V(T_n')\setminus A^{(n)}_{\sigma2}.
  \end{align*}
  Furthermore, since~$\Delta_t(A^{(n)}) = 0$ for every~$t\leq k-1$, we also have
  \begin{align*}
    \lvert N^+(A^{(n)}_{\sigma1})\rvert - \lvert N^-(A^{(n)}_{\sigma1})\rvert & =
    \lvert A^{(n)}_{\sigma2}\rvert - \lvert A^{(n)}_{\sigma3}\rvert;
    \\
    \lvert N^+(A^{(n)}_{\sigma2})\rvert - \lvert N^-(A^{(n)}_{\sigma2})\rvert & =
    \lvert A^{(n)}_{\sigma3}\rvert - \lvert A^{(n)}_{\sigma1}\rvert.
  \end{align*}

  Applying Lemma~\ref{lem:balneigh} to~$(A^{(n)}_{\sigma1})_{n\in\mathbb{N}}$
  and~$(A^{(n)}_{\sigma2})_{n\in\mathbb{N}}$, we get
  \begin{align*}
    \lvert A^{(n)}_{\sigma2}\rvert - \lvert A^{(n)}_{\sigma3}\rvert & =
    o(\lvert T_n'\rvert);
    \\
    \lvert A^{(n)}_{\sigma3}\rvert - \lvert A^{(n)}_{\sigma1}\rvert & =
    o(\lvert T_n'\rvert).
  \end{align*}

  Since~$\sigma$ was chosen arbitrarily, we conclude that~$\Delta_k(A^{(n)})=o(\lvert
  T_n'\rvert)$. This means that we can edit~$\lvert T_n'\rvert^2$ arcs of~$T_n'$ and
  obtain a~$k$-equally~$\vec C_3$-decomposable tournament~$T_n^{(k)}$ (note that it
  is crucial that~$\lvert T_n'\rvert$ is a power of~$3$) and since this doesn't
  affect the convergence of the sequence~$(T_n')_{n\in\mathbb{N}}$, the
  sequence~$(T_n^{(k)})_{n\in\mathbb{N}}$ also converges to~$\phi$ and we still have
  that~$\lvert T_n^{(k)}\rvert$ is a power of~$3$.
\end{proof}

\begin{lemma}\label{lem:QTseqskequallyC3dec->QTphiC3}
  We have~$\ref{it:QTseqskequallyC3dec}\implies\ref{it:QTphiC3}$.
\end{lemma}

\begin{proof}
  For every~$k\in\mathbb{N}$, let~$(T^{(k)}_n)_{n\in\mathbb{N}}$ be a sequence
  of~$k$-equally~$\vec C_3$-decomposable tournaments converging to~$\phi$.

  Our objective is to diagonalize the family of
  sequences~$(T_n^{(k)})_{n\in\mathbb{N}}$ in a way that the resulting sequence still
  converges to~$\phi$. To do this, we let~$(D_t)_{t\in\mathbb{N}}$ be an enumeration
  of the set of all finite tournaments~$\mathcal{T}$, we set~$f(0)=0$, and for
  every~$k>0$, we let
  \begin{align*}
    f(k) & = \min\left\{
    u\in\mathbb{N} : \lvert T_u^{(k)}\rvert > \lvert T_{f(k-1)}^{(k-1)}\rvert
    \;\land\;
    \forall t\leq k, \forall m \geq u, \lvert p(D_t;T_m^{(k)}) - \phi(D_t)\rvert
    < \frac{1}{k}
    \right\}.
  \end{align*}

  Note that the fact that~$(T_n^{(k)})_{n\in\mathbb{N}}$ converges to~$\phi$
  guarantees that~$f(k)<\infty$ for every~$k\in\mathbb{N}$.

  Define now the sequence of tournaments~$(U_n)_{n\in\mathbb{N}}$ by
  letting~$U_n=T^{(n)}_{f(n)}$ for every~$n\in\mathbb{N}$.

  We claim that~$(U_n)_{n\in\mathbb{N}}$ also converges to~$\phi$. Indeed,
  if~$T'\in\mathcal{T}$ is a tournament, then there exists~$t\in\mathbb{N}$ such
  that~$D_t=T'$, hence, for every~$n>t$, we have
  \begin{align*}
    \lvert p(D_t;U_n) - \phi(D_t)\rvert < \frac{1}{n},
  \end{align*}
  which implies
  that~$\lim_{n\to\infty}p(D_t;U_n)=\phi(D_t)$. Therefore~$(U_n)_{n\in\mathbb{N}}$
  converges to~$\phi$.

  By construction, we know~$U_n$ is~$n$-equally~$\vec C_3$-decomposable; this
  means that we can obtain~$\trianglen[\lvert U_n\rvert]$ from~$U_n$ by editing at
  most
  \begin{align*}
    \binom{\lvert U_n\rvert}{2}
    - \sum_{i=1}^n3^i\left(\frac{\lvert U_n\rvert}{3^i}\right)^2
    & =
    \binom{\lvert U_n\rvert}{2} - \lvert U_n\rvert^2\left(\frac{1 - 3^{-n}}{2}\right)
    = o(\lvert U_n\rvert^2)
  \end{align*}
  arcs of~$U_n$.

  Therefore the sequence~$(\trianglen[\lvert U_n\rvert])_{n\in\mathbb{N}}$ also
  converges to~$\phi$, and since it is a subsequence
  of~$(\trianglen)_{n\in\mathbb{N}}$, we have~$\phi=\phiC$.
\end{proof}

Finally, we prove the convergence of the
sequence~$(\trianglen)_{n\in\mathbb{N}}$. This proof can be obtained by
reinterpreting the proofs of Lemmas~\ref{lem:QTbalC3dec->QTseqskequallyC3dec}
and~\ref{lem:QTseqskequallyC3dec->QTphiC3}.

\begin{proposition}\label{prop:triangleconv}
  The sequence~$(\trianglen)_{n\in\mathbb{N}}$ is convergent.
\end{proposition}

\begin{proof}
  Let
  \begin{align*}
    \mathcal{C} & =
    \{I\subset\mathbb{N} : (\trianglen[i])_{i\in I}\text{ is convergent}\},
  \end{align*}
  and for every~$I\in\mathcal{C}$, let~$\phi_I$ denote the limit
  of~$(\trianglen[i])_{i\in I}$.

  From compactness of~$[0,1]^{\mathcal{T}}$, we know
  that~$\mathcal{C}\neq\varnothing$. Even more, from compactness
  of~$[0,1]^{\mathcal{T}}$, we know that there exists~$I_0\in\mathcal{C}$ such that
  \begin{align*}
    I_0\subset \{3^n : n\in\mathbb{N}\}.
  \end{align*}

  Note that if~$I\in\mathcal{C}$, then~$\phi_I$ is~$\vec C_3$-decomposable (since
  the~$\trianglen$ are~$\vec C_3$-decomposable) and balanced (since all vertices of
  the~$\vec C_3$ have outdegree either~$\lfloor\lvert\trianglen\rvert/2\rfloor$
  or~$\lceil\lvert\trianglen\rvert/2\rceil$). Therefore~$\phi_I$
  satisfies~\ref{it:QTbalC3dec}, for every~$I\in\mathcal{C}$.

  Now we repeat the proof of Lemma~\ref{lem:QTbalC3dec->QTseqskequallyC3dec} for
  each~$I\in\mathcal{C}$ to obtain sequences~$(T_n^{(k)})_{n\in\mathbb{N}}$
  of~$k$-equally~$\vec C_3$-decomposable tournaments converging to~$\phi_I$ for
  each~$k\in\mathbb{N}$. However, we require that these sequences are such
  that~$\lvert T_n^{(k)}\rvert\in I_0$ for every~$n,k\in\mathbb{N}$.

  We proceed then to the proof of Lemma~\ref{lem:QTseqskequallyC3dec->QTphiC3} and we
  get that~$\phi_I$ is also the limit of a subsequence
  of~$(\trianglen[i])_{i\in I_0}$, hence~$\phi_I=\phi_{I_0}$ for every~$I\in\mathcal{C}$.

  Therefore, every convergent subsequence of~$(\trianglen)_{n\in\mathbb{N}}$
  converges to the same limit~$\phi_{I_0}$. By compactness of~$[0,1]^{\mathcal{T}}$
  again, this implies that~$(\trianglen)_{n\in\mathbb{N}}$ is convergent.
\end{proof}


		
\section{Proof of Theorem~\protect\ref{thm:C3deccharac}}
\label{sec:proofdec}

For convenience of the reader we state the theorem again below.

\begin{theorem*}
  A tournament~$T$ is~$\vec C_3$-decomposable if and only if it has no copies
  of~$T_5^8$, $T_5^{10}$ nor of~$T_5^{12}$.
\end{theorem*}

\begin{proof}
  It is straightforward to check that~$T_5^8$, $T_5^{10}$ and~$T_5^{12}$ are
  not~$\vec C_3$-decomposable and that the property of~$\vec C_3$-decomposability is
  hereditary (i.e., every subtournament of a~$\vec C_3$-decomposable tournament is
  also~$\vec C_3$-decomposable). This concludes the proof for one direction.

  \medskip

  We will prove the other direction by induction in the size~$n$ of the
  tournament~$T$ with no copies of~$T_5^8$, $T_5^{10}$ nor of~$T_5^{12}$.

  If~$n\leq 2$, then trivially~$T$ is~$\vec C_3$-decomposable. So let~$n\geq 3$ and
  suppose the assertion is true for tournaments of size smaller than~$n$.

  If~$T$ is transitive, then we can let~$A$ be the singleton consisting of the vertex
  of~$T$ with maximum outdegree, let~$B=V(T)\setminus A$ and~$C=\varnothing$ and note
  that~$(A,B,C)$ satisfies the items in Definition~\ref{def:C3dec} (using inductive
  hypothesis for item~\ref{it:recursion}), hence~$T$ is~$\vec C_3$-decomposable.

  Suppose then that~$T$ is not transitive and let~$a,b,c\in V(T)$ be such
  that~$ab,bc,ca\in A(T)$.

  Define the following sets
  \begin{align*}
    V_{abc} & = \{v\in V(T) : va, vb, vc\in A(T)\}; &
    V_{ab} & = \{v\in V(T) : va, vb, cv\in A(T)\};
    \\
    V_{bc} & = \{v\in V(T) : av, vb, vc\in A(T)\}; &
    V_{ac} & = \{v\in V(T) : va, bv, vc\in A(T)\};
    \\
    V_a & = \{v\in V(T) : va, bv, cv\in A(T)\}; &
    V_b & = \{v\in V(T) : av, vb, cv\in A(T)\};
    \\
    V_c & = \{v\in V(T) : av, bv, vc\in A(T)\}; &
    V_\varnothing & = \{v\in V(T) : av, bv, cv\in A(T)\};
  \end{align*}
  and note that these sets form a partition of~$V(T)\setminus\{a,b,c\}$.

  We may suppose furthermore that~$(a,b,c)$ is chosen in such a way as to
  minimize~$\lvert V_{abc}\cup V_\varnothing\rvert$.

  We claim now the following assertions (see Figure~\ref{fig:arcsC3dec}).
  \begin{enumerate}[label={\alph*.}, ref={(\alph*)}]
  \item $V_{ab}\times V_{bc}, V_{bc}\times V_{ac}, V_{ac}\times V_{ab}\subset A(T)$,
    otherwise there would exist a copy of~$T_5^{10}$ in~$T$;
    \label{it:arcs22}
  \item $V_a\times V_b, V_b\times V_c, V_c\times V_a\subset A(T)$, otherwise there
    would exist a copy of~$T_5^{10}$ in~$T$;
    \label{it:arcs11}
  \item $V_a\times V_{ab}, V_b\times V_{bc}, V_c\times V_{ac}\subset A(T)$, otherwise
    there would exist a copy of~$T_5^8$ in~$T$;
    \label{it:arcs12}
  \item $V_{ab}\times V_c, V_{bc}\times V_a, V_{ac}\times V_b\subset A(T)$, otherwise
    there would exist a copy of~$T_5^{12}$ in~$T$;
    \label{it:arcs21}
  \item $V_{abc}\times(V_a\cup V_b\cup V_c)\subset A(T)$, otherwise there would exist
    a copy of~$T_5^{10}$ in~$T$;
    \label{it:arcs31}
  \item $V_{abc}\times(V_{ab}\cup V_{bc}\cup V_{ac})\subset A(T)$, otherwise there
    would exist a copy of~$T_5^8$ in~$T$;
    \label{it:arcs32}
  \item $(V_a\cup V_b\cup V_c)\times V_\varnothing\subset A(T)$, otherwise there
    would exist a copy of~$T_5^8$ in~$T$;
    \label{it:arcs10}
  \item $(V_{ab}\cup V_{bc}\cup V_{ac})\times V_\varnothing\subset A(T)$, otherwise there
    would exist a copy of~$T_5^{10}$ in~$T$.
    \label{it:arcs20}
  \end{enumerate}

  \begin{figure}[ht]
    \begin{center}
      \input{arcsC3dec}
      \caption{Contradictions of the proof of Theorem~\ref{thm:C3deccharac} involving
        arcs between the sets~$V_{abc}$, $V_{ab}$, $V_{bc}$, $V_{ac}$, $V_a$, $V_b$,
        $V_c$ and~$V_\varnothing$ and forbidden tournaments~$T_5^8$, $T_5^{10}$
        and~$T_5^{12}$. The arcs omitted are all oriented downward.}
      \label{fig:arcsC3dec}
    \end{center}
  \end{figure}

  Now we claim that~$V_{abc}\times V_\varnothing\subset A(T)$. Suppose not, that is,
  suppose that~$v_{abc}\in V_{abc}$ and~$v_\varnothing\in V_\varnothing$ are such
  that~$v_\varnothing v_{abc}\in A(T)$. Since~$v_{abc}a,av_\varnothing\in A(T)$, we
  have
  \begin{align*}
    \{v\in V(T) : va, vv_{abc}, vv_\varnothing\in A(T)\}\cup
    \{v\in V(T) : av, v_{abc}v, v_\varnothing v\in A(T)\} & \subset
    (V_{abc}\cup V_\varnothing)\setminus\{v_{abc},v_\varnothing\},
  \end{align*}
  contradicting the choice of~$(a,b,c)$ such as to minimize~$\lvert V_{abc}\cup
  V_\varnothing\rvert$. Therefore we must have~$V_{abc}\times V_\varnothing\subset
  A(T)$.


  Figure~\ref{fig:strucC3dec} shows all arcs of~$T$ proven so far.

  \begin{figure}[ht]
    \begin{center}
      \tikzsetnextfilename{strucC3dec}
\begin{tikzpicture}
  \def\psize{2pt}
  \def\externrad{5}
  \def\circumrad{5}
  \def\abcrad{3}
  \def\setrad{0.4}
  \def\abclabeldist{0.4}
  \def\lwidth{3}
  \def\shortenpoint{0.2}
  \def\shorten{0.2}

  \pgfmathsetmacro{\shortenset}{\setrad + \shorten}
  \pgfmathsetmacro{\shortencircum}{\circumrad + \shorten}

  \begin{scope}[xshift={-\externrad cm}, rotate={-90}]
    \pgfmathsetmacro{\outrad}{\circumrad - \setrad}
    \pgfmathsetmacro{\inrad}{2*\abcrad - \outrad}

    \coordinate (abcCenter) at (0cm, 0cm);

    \def\pp{c}
    \def\po{a}
    \def\pi{ac}
    \def\pa{-30}
    \coordinate (\pp) at (\pa:\abcrad cm);
    \coordinate (v\po) at (\pa:\outrad cm);
    \coordinate (v\pi) at (\pa:\inrad cm);
    \foreach[%
      remember=\p as \pp,
      remember=\o as \po,
      remember=\i as \pi,
      remember=\a as \pa%
    ] \p/\o/\i/\a in {%
      a/b/ab/90,
      b/c/bc/210,
      c/a/ac/330%
    }{
      \coordinate (\p) at (\a:\abcrad cm);
      \coordinate (v\o) at (\a:\outrad cm);
      \coordinate (v\i) at (\a:\inrad cm);

      \filldraw (\p) circle (\psize);
      \draw (v\o) circle (\setrad cm);
      \draw (v\i) circle (\setrad cm);

      \node at (v\o) {$V_{\o}$};
      \node at (v\i) {$V_{\i}$};
      \node at ($(\p) + (\a+45:\abclabeldist)$) {$\p$};

      \draw[thick, arrows={-latex},
        shorten <=\shortenset cm, shorten >=\shortenpoint cm]
      (v\i) -- (\p);
      \draw[thick, arrows={-latex},
        shorten <=\shortenpoint cm, shorten >=\shortenset cm]
      (\p) -- (v\o);

      \pgfmathsetmacro{\ta}{\a-90}
      \pgfmathsetmacro{\pta}{\pa+90}
      \pgfmathsetmacro{\ma}{(\a+\pa)/2}

      \pgfmathsetmacro{\inabcrad}{(\inrad+\abcrad)/2}
      \pgfmathsetmacro{\outabcrad}{(\outrad+\abcrad)/2}

      \foreach \orig/\origsiz/\dest/\destsiz/\contrad in {%
        v\pi/\shortenset/v\i/\shortenset/\inrad,
        v\pi/\shortenset/\p/\shortenpoint/\inabcrad,
        v\pi/\shortenset/v\o/\shortenset/\abcrad,
        \pp/\shortenpoint/v\i/\shortenset/\inabcrad,
        \pp/\shortenpoint/\p/\shortenpoint/\abcrad,
        \pp/\shortenpoint/v\o/\shortenset/\outabcrad,
        v\po/\shortenset/v\i/\shortenset/\abcrad,
        v\po/\shortenset/\p/\shortenpoint/\outabcrad,
        v\po/\shortenset/v\o/\shortenset/\outrad%
      }{
        \draw[thick, arrows={-latex}, shorten <=\origsiz cm, shorten >=\destsiz cm]
        (\orig) .. controls (\ma:\contrad cm) .. (\dest);

      }
    }
  \end{scope}

  \draw (abcCenter) circle (\circumrad cm);

  \coordinate (vabc) at (60:\externrad cm);
  \coordinate (vempty) at (300:\externrad cm);

  \node at (vabc) {$V_{abc}$};
  \node at (vempty) {$V_\varnothing$};

  \draw (vabc) circle (\setrad);
  \draw (vempty) circle (\setrad);

  \draw[line width=\lwidth, arrows={-latex},
    shorten <=\shortenset cm, shorten >=\shortencircum cm]
  (vabc) -- (abcCenter);
  \draw[line width=\lwidth, arrows={-latex},
    shorten <=\shortencircum cm, shorten >=\shortenset cm]
  (abcCenter) -- (vempty);
  \draw[line width=\lwidth, arrows={-latex},
    shorten <=\shortenset cm, shorten >=\shortenset cm]
  (vabc) -- (vempty);
\end{tikzpicture}
      \caption{Typical structure of~$T$ in the proof of
        Theorem~\ref{thm:C3deccharac}.}
      \label{fig:strucC3dec}
    \end{center}
  \end{figure}
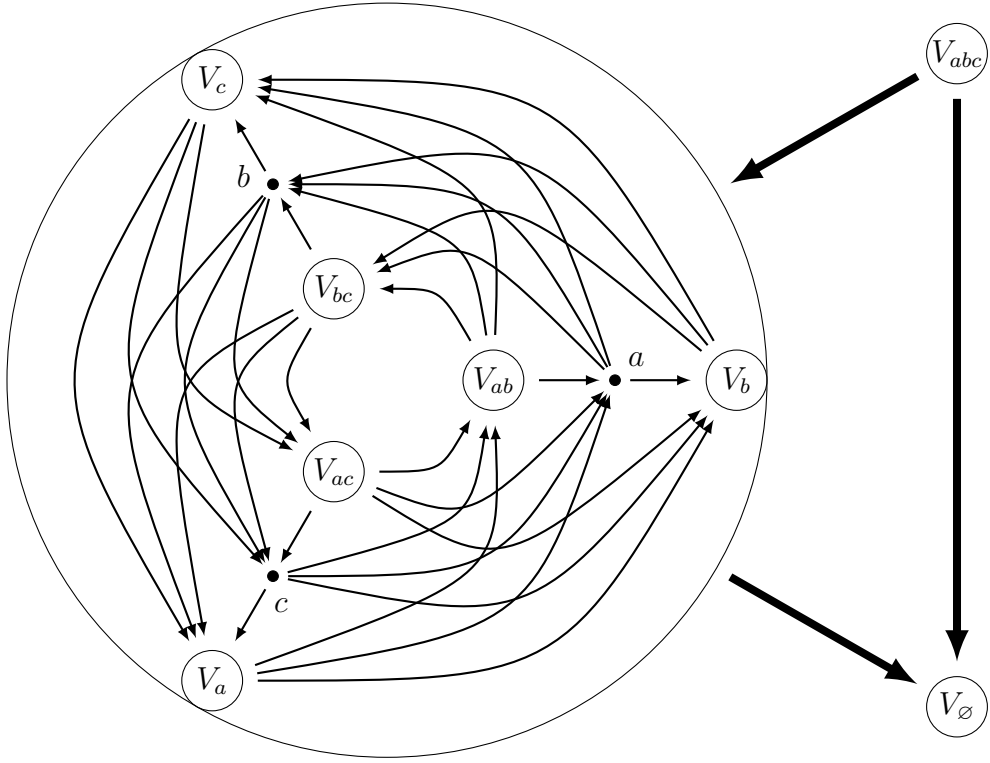

  Finally, we consider three cases.

  If~$V_{abc}\neq\varnothing$, let~$A = V_{abc}$, $B = V(T) \setminus V_{abc}$
  and~$C=\varnothing$ and note that~$(A,B,C)$ satisfies the items in
  Definition~\ref{def:C3dec} (using inductive hypothesis for
  item~\ref{it:recursion}), hence~$T$ is~$\vec C_3$-decomposable.
  
  If~$V_\varnothing\neq\varnothing$, let~$A = V(T) \setminus V_\varnothing$,
  $B=V_\varnothing$ and~$C=\varnothing$ and note that~$(A,B,C)$ satisfies the items
  in Definition~\ref{def:C3dec} (using inductive hypothesis for
  item~\ref{it:recursion}), hence~$T$ is~$\vec C_3$-decomposable.

  And finally, if~$V_{abc}\cup V_\varnothing = \varnothing$, let
  \begin{align*}
    A & = \{a\}\cup V_b\cup V_{ab}; &
    B & = \{b\}\cup V_c\cup V_{bc}; &
    C & = \{c\}\cup V_a\cup V_{ac};
  \end{align*}
  and note that~$(A,B,C)$ satisfies the items in Definition~\ref{def:C3dec} (using
  inductive hypothesis for item~\ref{it:recursion}), hence~$T$ is~$\vec
  C_3$-decomposable.
\end{proof}


	
\section{Proof of Lemma~\protect\ref{lem:twononneg}}
\label{sec:prooftwononneg}

For convenience of the reader we state the lemma again below.

\begin{lemma*}
  If~$(T_n)_{n\in\mathbb{N}}$ is a sequence of~$\vec C_3$-decomposable tournaments
  that converges to a balanced homomorphism~$\phi\in\Homp{0}$, then there
  exists a sequence~$(T_n')_{n\in\mathbb{N}}$ of~$\vec C_3$-decomposable tournaments
  and for every~$n\in\mathbb{N}$ a~$\vec
  C_3$-decomposition~$(A^{(n)}_\sigma)_{\sigma\in\Sigma^*}$ of~$T_n'$ such that
  \begin{itemize}
  \item There exists a subsequence~$(T_{k_n})_{n\in\mathbb{N}}$
    of~$(T_n)_{n\in\mathbb{N}}$ such that the tournament~$T_n'$ can be obtained
    from~$T_{k_n}$ by flipping~$o(\lvert T_{k_n}\rvert^2)$ arcs
    (hence~$(T_n')_{n\in\mathbb{N}}$ also converges to~$\phi$);
  \item We have~$\lvert A^{(n)}_1\rvert = \Omega(\lvert T_n'\rvert)$ and~$\lvert
    A^{(n)}_2\rvert = \Omega(\lvert T_n'\rvert)$.
  \end{itemize}
\end{lemma*}

\begin{proof}
  Suppose the lemma is not true and let~$(T_n)_{n\in\mathbb{N}}$ be a counter-example
  sequence.

  For every~$n\in\mathbb{N}$, let~$(B^{(n)}_\sigma)_{\sigma\in\Sigma^*}$ be a~$\vec
  C_3$-decomposition of~$T_n$. Without loss of generality, we may suppose that
  \begin{align*}
    \forall n\in\mathbb{N}, \forall \sigma\in\Sigma^*,
    \lvert B^{(n)}_{\sigma1}\rvert \geq \lvert B^{(n)}_{\sigma2}\rvert
    \land
    \lvert B^{(n)}_{\sigma1}\rvert \geq \lvert B^{(n)}_{\sigma3}\rvert.
  \end{align*}

  \begin{claim}\label{clm:intervalsamesize}
    Suppose~$u,v\:\mathbb{N}\to\mathbb{N}$ are two functions such that~$u(n)\leq
    v(n)$ for every~$n\in\mathbb{N}$.

    If~$\lvert B^{(n)}_{1^{v(n)+1}}\rvert = \Omega(\lvert T_n\rvert)$, then
    \begin{align*}
      \left\lvert\bigcup_{t=u(n)}^{v(n)}B^{(n)}_{1^t2}\right\rvert -
      \left\lvert\bigcup_{t=u(n)}^{v(n)}B^{(n)}_{1^t3}\right\rvert & = o(\lvert
      T_n\rvert).
    \end{align*}
  \end{claim}

  \begin{proof}
    Note that
    \begin{align*}
      N^+(B^{(n)}_{1^{v(n)+1}}) & = \bigcup_{t=0}^{v(n)}B^{(n)}_{1^t2}; &
      N^-(B^{(n)}_{1^{v(n)+1}}) & = \bigcup_{t=0}^{v(n)}B^{(n)}_{1^t3}.
    \end{align*}

    Since~$\lvert B^{(n)}_{1^{v(n)+1}}\rvert = \Omega(\lvert T_n\rvert)$, by
    Lemma~\ref{lem:balneigh}, we have
    \begin{align}\label{eq:vsiz}
      \left\lvert\bigcup_{t=0}^{v(n)}B^{(n)}_{1^t2}\right\rvert -
      \left\lvert\bigcup_{t=0}^{v(n)}B^{(n)}_{1^t3}\right\rvert & =
      o(\lvert T_n\rvert).
    \end{align}

    Note that, since~$u(n)\leq v(n)$, we have~$B^{(n)}_{1^{v(n)+1}}\subset
    B^{(n)}_{1^{u(n)}}$, which implies~$\lvert B^{(n)}_{1^{u(n)}}\rvert=\Omega(\lvert
    T_n\rvert)$, hence we have
    \begin{align}\label{eq:usiz}
      \left\lvert\bigcup_{t=0}^{u(n)-1}B^{(n)}_{1^t2}\right\rvert -
      \left\lvert\bigcup_{t=0}^{u(n)-1}B^{(n)}_{1^t3}\right\rvert & =
      o(\lvert T_n\rvert),
    \end{align}
    analogously to the case with~$v(n)$.

    The result follows by subtracting equation~\eqref{eq:usiz} from
    equation~\eqref{eq:vsiz}.
  \end{proof}

  \begin{claim}\label{clm:nonnegpred}
    If~$u\:\mathbb{N}\to\mathbb{N}$ is a function such that~$\lvert
    B^{(n)}_{1^{u(n)+1}2}\rvert = \Omega(\lvert T_n\rvert)$, then
    \begin{align*}
	  \left\lvert\bigcup_{t=0}^{u(n)}B^{(n)}_{1^t2}\right\rvert & =
	  \Omega(\lvert T_n\rvert).
    \end{align*}
  \end{claim}

  \begin{proof}
    Suppose the claim is not true. This means that there is a
    subsequence~$(T_{k_n})_{n\in\mathbb{N}}$ of~$(T_n)_{n\in\mathbb{N}}$ such that
    \begin{align}\label{eq:falsenonnegpred}
	  \left\lvert\bigcup_{t=0}^{u(k_n)}B^{(k_n)}_{1^t2}\right\rvert & =
	  o(\lvert T_{k_n}\rvert).
    \end{align}

    Let~$T_n'$ be the tournament obtained from~$T_{k_n}$ by flipping all the arcs in
    \begin{align*}
      A(T_{k_n})\cap
      &
	  \left(
      \left(
      B^{(k_n)}_{1^{u(k_n)+1}3}\times
      \left(\bigcup_{t=0}^{u(k_n)}B^{(k_n)}_{1^t2}\right)
      \right)
      \right.\\ & \cup
      \left(
      \left(\bigcup_{t=0}^{u(k_n)}B^{(k_n)}_{1^t3}\right)\times
      \left(\bigcup_{t=0}^{u(k_n)}B^{(k_n)}_{1^t2}\right)
      \right)
      \\ &\left. \cup
      \left(
      \left(\bigcup_{t=0}^{u(k_n)}B^{(k_n)}_{1^t3}\right)\times
      B^{(k_n)}_{1^{u(k_n)+1}2}
      \right)
      \right)
    \end{align*}
    and note that equation~\eqref{eq:falsenonnegpred} and
    Claim~\ref{clm:intervalsamesize} imply that the total of arcs flipped is~$o(\lvert
    T_{k_n}\rvert^2)$.

    Let
    \begin{align*}
      A^{(n)}_1 & = B^{(k_n)}_{1^{u(k_n)+2}}; &
      A^{(n)}_2 & = \bigcup_{t=0}^{u(k_n)+1}B^{(k_n)}_{1^t2}; &
      A^{(n)}_3 & = \bigcup_{t=0}^{u(k_n)+1}B^{(k_n)}_{1^t3}; &
    \end{align*}
    and note that~$\lvert A^{(n)}_1\rvert\geq\lvert A^{(n)}_2\rvert = \Omega(\lvert
    T_n'\rvert)$.

    Completing~$(A^{(n)}_1,A^{(n)}_2,A^{(n)}_3)$ to a~$\vec C_3$-decomposition
    of~$T_n'$ contradicts the choice of~$(T_n)_{n\in\mathbb{N}}$ as a counter-example
    sequence.
  \end{proof}

  \begin{claim}\label{clm:nonnegpredstrong}
	If~$u\:\mathbb{N}\to\mathbb{N}$ is a function such that~$\lvert
    B^{(n)}_{1^{u(n)+1}}\rvert = \Omega(\lvert T_n\rvert)$ and
    \begin{align*}
      \left\lvert\bigcup_{t=0}^{u(n)}B^{(n)}_{1^t2}\right\rvert & =
	  \Omega(\lvert T_n\rvert),
    \end{align*}
    then there exists a function~$w\:\mathbb{N}\to\mathbb{N}$ such that~$w(n)\leq
    u(n)$ and~$\lvert B^{(n)}_{1^{w(n)}2}\rvert = \Omega(\lvert T_n\rvert)$.
  \end{claim}

  \begin{proof}
    For every~$n\in\mathbb{N}$, let
    \begin{align*}
      M(n) & = \max\{\lvert B^{(n)}_{1^w2}\rvert : w \leq u(n)\};
      \\
      R(n) & = \max\left\{
      \left\lvert\bigcup_{t=0}^r B^{(n)}_{1^t2}\right\rvert
      - \left\lvert\bigcup_{t=0}^r B^{(n)}_{1^t3}\right\rvert : r\leq u(n)
      \right\};
      \\
      S(n) & = \max\{
      \lvert B^{(n)}_{1^s3}\rvert
      - \lvert B^{(n)}_{1^s2}\rvert : s\leq u(n)
      \}.
    \end{align*}

    By Claim~\ref{clm:intervalsamesize}, we know that~$R(n) = o(\lvert T_n\rvert)$
    and~$S(n) = o(\lvert T_n\rvert)$.

    Suppose towards a contradiction that the claim is false. This means that we must
    have~$M(n)\neq\Omega(\lvert T_n\rvert)$, that is, there exists a
    subsequence~$(T_{k_n})_{n\in\mathbb{N}}$ of~$(T_n)_{n\in\mathbb{N}}$ such
    that~$M(k_n)=o(\lvert T_{k_n}\rvert)$.

    Note now that if~$t\leq u(k_n)$ and~$v_t\in B^{(k_n)}_{1^t2}$ (see
    Figure~\ref{fig:neighbourhoodpiece} further ahead for the neighbourhoods of the
    set~$B^{(k_n)}_{1^t2}$), then
    \begin{align*}
	  d^+(v_t) & =
      \left\lvert\bigcup_{t=0}^{t-1}B^{(k_n)}_{1^t2}\right\rvert +
      \left\lvert B^{(k_n)}_{1^t3}\right\rvert +
      d^+_{B^{(k_n)}_{1^t2}}(v_t);
      \\
	  d^-(v_t) & =
      \left\lvert\bigcup_{t=0}^{t-1}B^{(k_n)}_{1^t3}\right\rvert +
      \left\lvert B^{(k_n)}_{1^{t+1}}\right\rvert +
      d^-_{B^{(k_n)}_{1^t2}}(v_t),
    \end{align*}
    where~$d^+_A(v) = \lvert N^+(v)\cap A\rvert$ and~$d^-_A(v) = \lvert N^-(v)\cap
    A\rvert$.

    Since
    \begin{align*}
      \left\lvert\bigcup_{t=0}^{t-1}B^{(k_n)}_{1^t2}\right\rvert
      - \left\lvert\bigcup_{t=0}^{t-1}B^{(k_n)}_{1^t3}\right\rvert
      & \leq
      R(k_n);
      \\
      \left\lvert B^{(k_n)}_{1^t3}\right\rvert
      & \leq
      \left\lvert B^{(k_n)}_{1^t2}\right\rvert + S(k_n)
      \leq
      M(k_n) + S(k_n);
      \\
      d^+_{B^{(k_n)}_{1^t2}}(v_t) - d^-_{B^{(k_n)}_{1^t2}}(v_t)
      & \leq
      \lvert B^{(k_n)}_{1^t2}\rvert
      \leq
      M(k_n);
      \\
      \left\lvert B^{(k_n)}_{1^{t+1}}\right\rvert
      & \geq
      \left\lvert B^{(k_n)}_{1^{u(k_n)+1}}\right\rvert;
    \end{align*}
    we have
    \begin{align*}
      d^+(v_t) - d^-(v_t)
      & \leq
      R(k_n) + 2M(k_n) + S(k_n) - \left\lvert B^{(k_n)}_{1^{u(k_n)+1}}\right\rvert.
    \end{align*}
    Note that this bound does not depend on~$t$.

    Since~$R(k_n)=o(\lvert T_{k_n}\rvert)$, $M(k_n)=o(\lvert T_{k_n}\rvert)$,
    $S(k_n)=o(\lvert T_{k_n}\rvert)$ and~$\lvert
    B^{(k_n)}_{1^{u(k_n)+1}}\rvert=\Omega(\lvert T_{k_n}\rvert)$, this implies that
    \begin{align*}
      d^+(v) - d^-(v)
      & \leq
      R(k_n) + 2M(k_n) + S(k_n) - \left\lvert B^{(k_n)}_{1^{u(k_n)+1}}\right\rvert
      \leq
      -\epsilon\lvert T_{k_n}\rvert,    
    \end{align*}
    for every~$v\in\bigcup_{t=0}^{u(k_n)}B^{(k_n)}_{1^t2}$ and~$n\in\mathbb{N}$ large
    enough, which contradicts the fact that~$\phi$ is balanced
    (since~$\lvert\bigcup_{t=0}^{u(k_n)}B^{(k_n)}_{1^t2}\rvert = \Omega(\lvert
    T_{k_n}\rvert)$).
  \end{proof}

  \begin{claim}\label{clm:twononnegrel}
    Suppose~$u,v\:\mathbb{N}\to\mathbb{N}$ are two functions such that~$u(n) < v(n)$
    for every~$n\in\mathbb{N}$.

    If~$\lvert B^{(n)}_{1^{u(n)}2}\rvert = \Omega(\lvert T_n\rvert)$ and~$\lvert
    B^{(n)}_{1^{v(n)}2}\rvert = \Omega(\lvert T_n\rvert)$, then
    \begin{align*}
      \lvert B^{(n)}_{1^{u(n)}2}\rvert & =
      2\left\lvert\bigcup_{t=u(n)+1}^{v(n)-1}B^{(n)}_{1^t2}\right\rvert
      + 3\lvert B^{(n)}_{1^{v(n)}2}\rvert
      + o(\lvert T_n\rvert).
    \end{align*}
  \end{claim}

  \begin{proof}
    By Lemma~\ref{lem:balneigh}, we know that~$\lvert N^+(B^{(n)}_{1^{v(n)}2})\rvert -
    \lvert N^-(B^{(n)}_{1^{v(n)}2})\rvert = o(\lvert T_n\rvert)$, that is, we have
    (see Figure~\ref{fig:neighbourhoodpiece})
    \begin{align*}
      \left\lvert\bigcup_{t=0}^{v(n)-1}B^{(n)}_{1^t2}\right\rvert
      + \lvert B^{(n)}_{1^{v(n)}3}\rvert
      - \left\lvert\bigcup_{t=0}^{v(n)-1}B^{(n)}_{1^t3}\right\rvert
      - \lvert B^{(n)}_{1^{v(n)+1}}\rvert
      & =
      o(\lvert T_n\rvert).
    \end{align*}

    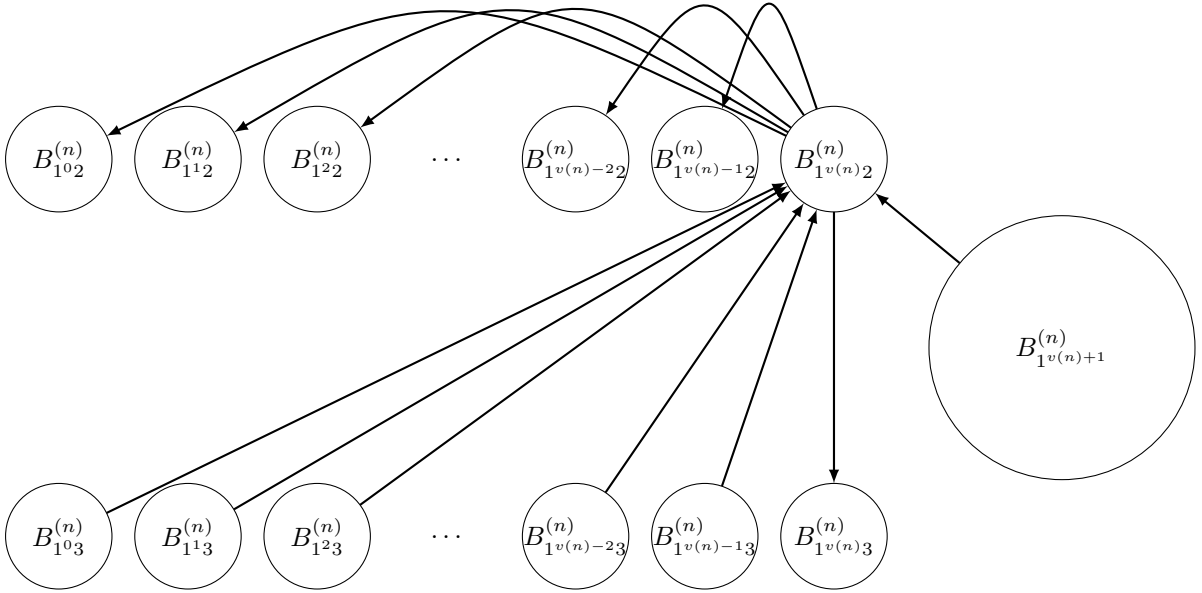
\begin{figure}[ht]
      \tikzsetnextfilename{neighbourhoodpiece}
\begin{footnotesize}
\begin{tikzpicture}
  \def\piecerad{0.7}
  \def\pieceheight{2.5}
  \def\piecehstep{1.7}
  \def\startendpieceextrahdist{\piecehstep}
  \def\lastpiecerad{1.75}
  \def\lastpiecehstep{3}
  \def\arrowcontrolheight{2.5}
  \def\piecenumberstartminusone{2}
  \def\piecenumberend{2}

  \coordinate (v-1) at ($\piecenumberstartminusone*(\piecehstep cm,0cm) +
  \piecenumberend*(\piecehstep cm,0cm) + (\startendpieceextrahdist cm,0cm) +
  (\piecehstep cm, 0cm) + (\lastpiecehstep cm, 0cm)$);
  \coordinate (v-2) at ($(v-1) + (-\lastpiecehstep cm,\pieceheight)$);
  \coordinate (v-3) at ($(v-1) + (-\lastpiecehstep cm,-\pieceheight)$);
  \draw (v-1) circle (\lastpiecerad cm);
  \draw (v-2) circle (\piecerad cm);
  \draw (v-3) circle (\piecerad cm);
  \node at (v-1) {$B^{(n)}_{1^{v(n)+1}}$};
  \node at (v-2) {$B^{(n)}_{1^{v(n)}2}$};
  \node at (v-3) {$B^{(n)}_{1^{v(n)}3}$};

  \draw[thick,arrows={-latex},shorten <=\piecerad cm, shorten >= \piecerad cm]
  (v-2) -- (v-3);
  \draw[thick,arrows={-latex},shorten <=\lastpiecerad cm, shorten >= \piecerad cm]
  (v-1) -- (v-2);

  \foreach \i in {0,...,\piecenumberstartminusone}{
    \coordinate (\i-2) at ($(\i*\piecehstep cm, \pieceheight cm)$);
    \coordinate (\i-3) at ($(\i*\piecehstep cm, -\pieceheight cm)$);
    \draw (\i-2) circle (\piecerad cm);
    \draw (\i-3) circle (\piecerad cm);
    \node at (\i-2) {$B^{(n)}_{1^{\i}2}$};
    \node at (\i-3) {$B^{(n)}_{1^{\i}3}$};

    \draw[thick,arrows={-latex},shorten <=\piecerad cm, shorten >= \piecerad cm]
    (v-2) .. controls ($1/2*(\i-2) + 1/2*(v-2) + (0cm, \arrowcontrolheight cm)$)
    .. (\i-2);
    \draw[thick,arrows={-latex},shorten <=\piecerad cm, shorten >= \piecerad cm]
    (\i-3) -- (v-2);
  }

  \foreach \i in {1,...,\piecenumberend}{
    \coordinate (e\i-2) at ($(v-2) + (-\i*\piecehstep cm, 0 cm)$);
    \coordinate (e\i-3) at ($(v-3) + (-\i*\piecehstep cm, 0 cm)$);
    \draw (e\i-2) circle (\piecerad cm);
    \draw (e\i-3) circle (\piecerad cm);
    \node at (e\i-2) {$B^{(n)}_{1^{v(n)-\i}2}$};
    \node at (e\i-3) {$B^{(n)}_{1^{v(n)-\i}3}$};

    \draw[thick,arrows={-latex},shorten <=\piecerad cm, shorten >= \piecerad cm]
    (v-2) .. controls ($1/2*(e\i-2) + 1/2*(v-2) + (0cm, \arrowcontrolheight cm)$)
    .. (e\i-2);
    \draw[thick,arrows={-latex},shorten <=\piecerad cm, shorten >= \piecerad cm]
    (e\i-3) -- (v-2);
  }

  \node at ($1/2*(e\piecenumberend-2) + 1/2*(\piecenumberstartminusone-2)$)
  {$\ldots$};
  \node at ($1/2*(e\piecenumberend-3) + 1/2*(\piecenumberstartminusone-3)$)
  {$\ldots$};
\end{tikzpicture}
\end{footnotesize}
      \caption{Neighbourhoods of~$B^{(n)}_{1^{v(n)}2}$.}
      \label{fig:neighbourhoodpiece}
    \end{figure}

    By Claim~\ref{clm:intervalsamesize}, this implies that
    \begin{align}\label{eq:vbal}
      \lvert B^{(n)}_{1^{v(n)}2}\rvert
      - \lvert B^{(n)}_{1^{v(n)+1}}\rvert
      & = o(\lvert T_n\rvert).
    \end{align}

    With an analogous argument for~$u(n)$, we get
    \begin{align*}
      \lvert B^{(n)}_{1^{u(n)}2}\rvert
      - \lvert B^{(n)}_{1^{u(n)+1}}\rvert & =
      o(\lvert T_n\rvert),
    \end{align*}
    which implies
    \begin{align*}
      \lvert B^{(n)}_{1^{u(n)}2}\rvert
      - \left\lvert
      \bigcup_{t=u(n)+1}^{v(n)-1}(B^{(n)}_{1^t2}\cup B^{(n)}_{1^t3})
      \right\rvert
      - \lvert B^{(n)}_{1^{v(n)+1}}\rvert
      - \lvert B^{(n)}_{1^{v(n)}2}\rvert
      - \lvert B^{(n)}_{1^{v(n)}3}\rvert
      & =
      o(\lvert T_n\rvert).
    \end{align*}

    Since~$\lvert B^{(n)}_{1^{v(n)}}\rvert\geq\lvert
    B^{(n)}_{1^{v(n)+1}}\rvert\geq\lvert B^{(n)}_{1^{v(n)}2}\rvert=\Omega(\lvert
    T_n\rvert)$, two more applications of Claim~\ref{clm:intervalsamesize} yield
    \begin{align}\label{eq:ubalfinal}
      \lvert B^{(n)}_{1^{u(n)}2}\rvert
      - 2\left\lvert\bigcup_{t=u(n)+1}^{v(n)-1}B^{(n)}_{1^t2}\right\rvert
      - \lvert B^{(n)}_{1^{v(n)+1}}\rvert
      - 2\lvert B^{(n)}_{1^{v(n)}2}\rvert    
      & = o(\lvert T_n\rvert).
    \end{align}

    Subtracting equation~\eqref{eq:vbal} from~\eqref{eq:ubalfinal}, we get
    \begin{align*}
      \lvert B^{(n)}_{1^{u(n)}2}\rvert
      - 2\left\lvert\bigcup_{t=u(n)+1}^{v(n)}B^{(n)}_{1^t2}\right\rvert
      - 3\lvert B^{(n)}_{1^{v(n)}2}\rvert    
      & = o(\lvert T_n\rvert).
      \qedhere
    \end{align*}
  \end{proof}

  We are now in condition of finishing the proof of the lemma. For
  every~$n\in\mathbb{N}$, let
  \begin{align*}
    v(n) & = \min\left\{v\in\mathbb{N} :
    \left\lvert\bigcup_{t=0}^v (B^{(n)}_{1^t2}\cup B^{(n)}_{1^t3})\right\rvert
    \geq \frac{2}{3}\lvert T_n\rvert
    \right\}.
  \end{align*}
  Note that~$v(n)$ is well-defined for~$n\geq 3$ and, by
  Claim~\ref{clm:intervalsamesize}, we have~$\lvert\bigcup_{t=0}^{v(n)}
  B^{(n)}_{1^t2}\rvert = \Omega(\lvert T_n\rvert)$.

  Furthermore, note that
  \begin{align*}
    \lvert B^{(n)}_{1^{v(n)+1}}\rvert & \geq \frac{1}{9}\lvert T_n\rvert;
  \end{align*}
  which, by Claim~\ref{clm:nonnegpredstrong}, implies that there exists a
  function~$w\:\mathbb{N}\to\mathbb{N}$ with~$w(n)\leq v(n)$ for
  every~$n\in\mathbb{N}$ and such that~$\lvert B^{(n)}_{1^{w(n)}2}\rvert =
  \Omega(\lvert T_n\rvert)$, that is, there exists~$n_0\in\mathbb{N}$
  and~$\epsilon>0$ such that
  \begin{align*}
    \lvert B^{(n)}_{1^{w(n)}2}\rvert & \geq \epsilon\lvert T_n\rvert
  \end{align*}
  for every~$n\geq n_0$.

  For every~$n\in\mathbb{N}$, let
  \begin{align*}
    w_0(n) & = \min\left\{w\in\mathbb{N} :
    \lvert B^{(n)}_{1^w2}\rvert
    \geq \epsilon\lvert T_n\rvert
    \right\}
  \end{align*}
  and note that, for every~$n\geq n_0$, we have that~$w_0(n)$ is well-defined
  and~$w_0(n)\leq w(n)$.

  Since~$\lvert B^{(n)}_{1^{w_0(n)}2}\rvert = \Omega(\lvert T_n\rvert)$, by
  Claim~\ref{clm:nonnegpred}, we know that
  \begin{align*}
    \left\lvert\bigcup_{t=0}^{w_0(n)-1}B^{(n)}_{1^t2}\right\rvert & =
	\Omega(\lvert T_n\rvert).
  \end{align*}

  Another application of Claim~\ref{clm:nonnegpredstrong} yields then a
  function~$u\:\mathbb{N}\to\mathbb{N}$ with~$u(n)\leq w_0(n)-1$ for
  every~$n\in\mathbb{N}$ and such that~$\lvert B^{(n)}_{1^{u(n)}2}\rvert =
  \Omega(\lvert T_n\rvert)$.

  Now, by Claim~\ref{clm:twononnegrel}, we have
  \begin{align*}
    \lvert B^{(n)}_{1^{u(n)}2}\rvert & =
    2\left\lvert\bigcup_{t=u(n)+1}^{w_0(n)-1}B^{(n)}_{1^t2}\right\rvert
    + 3\lvert B^{(n)}_{1^{w_0(n)}2}\rvert
    + o(\lvert T_n\rvert)
    \geq
    3\epsilon\lvert T_n\rvert + o(\lvert T_n\rvert),
  \end{align*}
  which implies that for~$n\in\mathbb{N}$ large enough, we have~$\lvert
  B^{(n)}_{1^{u(n)}2}\rvert\geq\epsilon\lvert T_n\rvert$, contradicting the definition
  of~$w_0(n)$.
\end{proof}


	
%
\bibliographystyle{amsplain}
\bibliography{bibliografia}

\providecommand{\bysame}{\leavevmode\hbox to3em{\hrulefill}\thinspace}
\providecommand{\MR}{\relax\ifhmode\unskip\space\fi MR }
\providecommand{\MRhref}[2]{%
  \href{http://www.ams.org/mathscinet-getitem?mr=#1}{#2}
}
\providecommand{\href}[2]{#2}
\begin{thebibliography}{10}

\bibitem{thesiBaber}
R.~Baber, \emph{Some results in extremal combinatorics}, Ph.D. thesis,
  University College London, 2011.

\bibitem{BHLPUV:MinimumNumberOfMonotoneSubsequences}
J.~Balogh, P.~Hu, B.~Lidick\`{y}, O.~Pikhurko, B.~Udvari, and J.~Volec,
  \emph{Minimum number of monotone subsequences of length 4 in permutations},
  (2013), Pre-print available at
  \texttt{http://homepages.warwick.ac.uk/~maskat/Papers/monoSeq.pdf}.

\bibitem{BeHa65}
L.~W. Beineke and F.~Harary, \emph{The maximum number of strongly connected
  subtournaments}, Canad. Math. Bull. \textbf{8} (1965), 491--498. \MR{0181581
  (31 \#5810)}

\bibitem{BR:NoteOnUpperDensityOfQuasiRandomHypergraphs}
V.~Bhat and V.~R{\"o}dl, \emph{Note on upper density of quasi-random
  hypergraphs}, Electron. J. Combin. \textbf{20} (2013), no.~2, Paper 59, 8.
  \MR{3084601}

\bibitem{CSDP}
B.~Borchers, \emph{Csdp, a c library for semidefinite programming},
  Optimization Methods and Software \textbf{11} (1999), no.~1-4, 613--623.

\bibitem{C:QuasiRandomHypergraphsRevisited}
F.~Chung, \emph{Quasi-random hypergraphs revisited}, Random Structures
  Algorithms \textbf{40} (2012), no.~1, 39--48. \MR{2864651 (2012m:05315)}

\bibitem{CG:QuasiRandomHypergraphs}
F.~Chung and R.~Graham, \emph{Quasi-random hypergraphs}, Random Structures
  Algorithms \textbf{1} (1990), no.~1, 105--124. \MR{1068494 (91h:05089)}

\bibitem{CG:QuasiRandomTournaments}
\bysame, \emph{Quasi-random tournaments}, J. Graph Theory \textbf{15} (1991),
  no.~2, 173--198.

\bibitem{ChGrWi88}
F.~Chung, R.~Graham, and R.~Wilson, \emph{Quasirandom graphs}, Proc. Nat. Acad.
  Sci. U.S.A. \textbf{85} (1988), no.~4, 969--970. \MR{MR928566 (89a:05116)}

\bibitem{C:SuiCircuitiNeiGrafiCompleti}
U.~Colombo, \emph{Sui circuiti nei grafi completi}, Boll. Un. Mat. Ital. (3)
  \textbf{19} (1964), 153--170. \MR{0172262 (30 \#2482)}

\bibitem{C:QuasirandomPermutations}
J.~N. Cooper, \emph{Quasirandom permutations}, J. Combin. Theory Ser. A
  \textbf{106} (2004), no.~1, 123--143. \MR{2050120 (2005f:05002)}

\bibitem{Na15}
L.~N. {Coregliano}, \emph{{Quasi-Carousel Tournaments}}, ArXiv e-prints (2015).

\bibitem{NaRaz15a}
L.~N. {Coregliano} and A.~A. {Razborov}, \emph{{On the Density of Transitive
  Tournaments}}, ArXiv e-prints (2015).

\bibitem{CKPSTY:MonochromaticTrianglesInThreeColouredGraphs}
J.~Cummings, D.~Kr{\'a}l', F.~Pfender, K.~Sperfeld, A.~Treglown, and M.~Young,
  \emph{Monochromatic triangles in three-coloured graphs}, J. Combin. Theory
  Ser. B \textbf{103} (2013), no.~4, 489--503. \MR{3071377}

\bibitem{DHMNS:AProblemOfErdosOnTheMinimumNumberOfkCliques}
S.~Das, H.~Huang, J.~Ma, H.~Naves, and B.~Sudakov, \emph{A problem of {E}rd{\H
  o}s on the minimum number of {$k$}-cliques}, J. Combin. Theory Ser. B
  \textbf{103} (2013), no.~3, 344--373. \MR{3048160}

\bibitem{DJ:GraphLimitsAndExchangeableRandomGraphs}
P.~Diaconis and S.~Janson, \emph{Graph limits and exchangeable random graphs},
  Rend. Mat. Appl. (7) \textbf{28} (2008), no.~1, 33--61. \MR{2463439
  (2010a:60127)}

\bibitem{ES:MeasureTheoreticApproach}
G.~Elek and B.~Szegedy, \emph{A measure-theoretic approach to the theory of
  dense hypergraphs}, Adv. Math. \textbf{231} (2012), no.~3-4, 1731--1772.
  \MR{2964622}

\bibitem{FV:ApplicationsOfTheSemidefiniteMethod}
V.~Falgas-Ravry and E.~R. Vaughan, \emph{Applications of the semi-definite
  method to the {T}ur\'an density problem for 3-graphs}, Combin. Probab.
  Comput. \textbf{22} (2013), no.~1, 21--54. \MR{3002572}

\bibitem{G:QuasiRandomOrientedGraphs}
S.~Griffiths, \emph{Quasi-random oriented graphs}, J. Graph Theory \textbf{74}
  (2013), no.~2, 198--209. \MR{3090716}

\bibitem{HKMRS:LimitsOfPermutationSequences}
C.~Hoppen, Y.~Kohayakawa, C.~G. Moreira, B.~R{\'a}th, and R.~M. Sampaio,
  \emph{Limits of permutation sequences}, J. Combin. Theory Ser. B \textbf{103}
  (2013), no.~1, 93--113. \MR{2995721}

\bibitem{KS:ANoteOnEvenCyclesAndQuasirandomTournaments}
S.~Kalyanasundaram and A.~Shapira, \emph{A note on even cycles and quasirandom
  tournaments}, J. Graph Theory \textbf{73} (2013), no.~3, 260--266.

\bibitem{KP:QuasirandomPermutationsAreCharacterized}
D.~Kr{\'a}l' and O.~Pikhurko, \emph{Quasirandom permutations are characterized
  by 4-point densities}, Geom. Funct. Anal. \textbf{23} (2013), no.~2,
  570--579. \MR{3053756}

\bibitem{KS:PseudoRandomGraphs}
M.~Krivelevich and B.~Sudakov, \emph{Pseudo-random graphs}, More Sets, Graphs
  and Numbers, Bolyai Society Mathematical Studies 15, Springer-Verlag, 2006,
  pp.~199--262.

\bibitem{LoSz06}
L.~Lov{\'a}sz and B.~Szegedy, \emph{Limits of dense graph sequences}, J.
  Combin. Theory Ser. B \textbf{96} (2006), no.~6, 933--957. \MR{2274085
  (2007m:05132)}

\bibitem{maxima}
Maxima, \emph{Maxima, a computer algebra system. version 5.34.1}, 2014.

\bibitem{PiVa13}
O.~Pikhurko and E.~R. Vaughan, \emph{Minimum number of {$k$}-cliques in graphs
  with bounded independence number}, Combin. Probab. Comput. \textbf{22}
  (2013), no.~6, 910--934. \MR{3111549}

\bibitem{Raz07}
A.~Razborov, \emph{Flag algebras}, J. Symbolic Logic \textbf{72} (2007), no.~4,
  1239--1282. \MR{2371204 (2008j:03040)}

\bibitem{Raz10}
\bysame, \emph{On 3-hypergraphs with forbidden 4-vertex configurations}, SIAM
  J. Discrete Math. \textbf{24} (2010), no.~3, 946--963. \MR{2680226
  (2011k:05171)}

\bibitem{R:FlagAlgebrasInterim}
A.~A. Razborov, \emph{Flag algebras: an interim report}, 2013.

\bibitem{R:WhatIsAFlagAlgebra}
\bysame, \emph{What is{$\ldots$}a flag algebra?},  \textbf{60} (2013), no.~10,
  1324--1327. \MR{3135939}

\bibitem{T:PseudoRandomGraphs}
A.~Thomason, \emph{Pseudo-random graphs}, Ann. of Discrete Math. \textbf{33}
  (1987), 307--331.

\bibitem{SDPAFamily}
M.~Yamashita, K.~Fujisawa, M.~Fukuda, K.~Kobayashi, K.~Nakata, and M.~Nakata,
  \emph{Latest developments in the {SDPA} family for solving large-scale
  {SDP}s}, Handbook on semidefinite, conic and polynomial optimization,
  Internat. Ser. Oper. Res. Management Sci., vol. 166, Springer, New York,
  2012, pp.~687--713. \MR{2894706}

\end{thebibliography}

\appendix
\section{Appendix}
\label{appendix}

\subsection{Positive semi-definite matrices used}
\label{appendix:matrix}

\begin{align*}
  Q(T_5^7,1) & =
  \frac{35}{48}\cdot
  \!\!\!\kbordermatrix{
    & \Tr_3^{1,L} & \vec C_3^1 & \Tr_3^{1,M} & \Tr_3^{1,W}\\
    & 1 & -1 & -1 & 1\\
    & -1 & 1 & 1 & -1\\
    & -1 & 1 & 1 & -1\\
    & 1 & -1 & -1 & 1
  }
  =
  \frac{35}{48}\cdot
  \left(
  \begin{array}{c}
    1\\ -1\\ -1\\ 1\\
  \end{array}
  \right)
  \cdot
  \left(
  \begin{array}{c}
    1\\ -1\\ -1\\ 1\\
  \end{array}
  \right)^\top.
\end{align*}

\begin{align*}
  Q(T_5^7,\Tr_3^*) & =
  5\cdot
  \!\!\!\kbordermatrix{
    & \Tr_4^{\Tr_3^*,3} & W_4^{\Tr_3^*} & 
    \Tr_4^{\Tr_3^*,2} & R_4^{\Tr_3^*,1} &
    L_4^{\Tr_3^*} & \Tr_4^{\Tr_3^*,1} &
    R_4^{\Tr_3^*,2} & \Tr_4^{\Tr_3^*,0}\\
    & 1 & 0 & -1 & 1 & 0 & 1 & -1 & -1\\
    & 0 & 0 & 0 & 0 & 0 & 0 & 0 & 0\\
    & -1 & 0 & 1 & -1 & 0 & -1 & 1 & 1\\
    & 1 & 0 & -1 & 1 & 0 & 1 & -1 & -1\\
    & 0 & 0 & 0 & 0 & 0 & 0 & 0 & 0\\
    & 1 & 0 & -1 & 1 & 0 & 1 & -1 & -1\\
    & -1 & 0 & 1 & -1 & 0 & -1 & 1 & 1\\
    & -1 & 0 & 1 & -1 & 0 & -1 & 1 & 1
  }
  \displaybreak[0]\\
  & = 5\cdot
  \left(
  \begin{array}{c}
    1 \\ 0 \\ -1 \\ 1 \\ 0 \\ 1 \\ -1 \\ -1
  \end{array}
  \right)
  \cdot
  \left(
  \begin{array}{c}
    1 \\ 0 \\ -1 \\ 1 \\ 0 \\ 1 \\ -1 \\ -1
  \end{array}
  \right)^\top.
\end{align*}

\begin{align*}
  Q(T_5^7,\vec C_3^*) & =
  12\cdot
  \!\!\!\kbordermatrix{
    & R_4^{\vec C_3^*,3} & L_4^{\vec C_3^*} &
    R_4^{\vec C_3^*,2} & R_4^{\vec C_3^*,1} &
    R_4^{\vec C_3^*,23} & W_4^{\vec C_3^*} &
    R_4^{\vec C_3^*,12} & R_4^{\vec C_3^*,13}\\
    & 0 & 0 & 0 & 0 & 0 & 0 & 0 & 0\\
    & 0 & 1 & 0 & 0 & 0 & -1 & 0 & 0\\
    & 0 & 0 & 0 & 0 & 0 & 0 & 0 & 0\\
    & 0 & 0 & 0 & 0 & 0 & 0 & 0 & 0\\
    & 0 & 0 & 0 & 0 & 0 & 0 & 0 & 0\\
    & 0 & -1 & 0 & 0 & 0 & 1 & 0 & 0\\
    & 0 & 0 & 0 & 0 & 0 & 0 & 0 & 0\\
    & 0 & 0 & 0 & 0 & 0 & 0 & 0 & 0
  }
  \displaybreak[0]\\
  & = 12\cdot
  \left(
  \begin{array}{c}
    0 \\ 1 \\ 0 \\ 0 \\ 0 \\ -1 \\ 0 \\ 0
  \end{array}
  \right)
  \cdot
  \left(
  \begin{array}{c}
    0 \\ 1 \\ 0 \\ 0 \\ 0 \\ -1 \\ 0 \\ 0
  \end{array}
  \right)^\top.
\end{align*}

\begin{align*}
  Q(T_5^8,1) & =
  \kbordermatrix{
    & \Tr_3^{1,L} & \vec C_3^1 & \Tr_3^{1,M} & \Tr_3^{1,W}\\
    & \frac{2473}{6400} & -\frac{363}{1600} & -\frac{757}{1600} & \frac{2007}{6400}
    \\
    & -\frac{363}{1600} & \frac{1407}{6400} & \frac{1441}{6400} & -\frac{349}{1600}
    \\
    & -\frac{757}{1600} & \frac{1441}{6400} & \frac{4659}{6400} & -\frac{12}{25}
    \\
    & \frac{2007}{6400} & -\frac{349}{1600} & -\frac{12}{25} & \frac{2461}{6400}
  }.
\end{align*}

\begin{align*}
  Q(T_5^8,A) & =
  \frac{99}{3200}\cdot
  \!\!\!\kbordermatrix{
    & I^A & \vec C_3^A & \Tr_3^A & O^A\\
    & 1 & -1 & -1 & 1\\
    & -1 & 1 & 1 & -1\\
    & -1 & 1 & 1 & -1\\
    & 1 & -1 & -1 & 1
  }
  =
  \frac{99}{3200}\cdot
  \left(
  \begin{array}{c}
    1\\ -1\\ -1\\ 1\\
  \end{array}
  \right)
  \cdot
  \left(
  \begin{array}{c}
    1\\ -1\\ -1\\ 1\\
  \end{array}
  \right)^\top.
\end{align*}

\begin{align*}
  Q(T_5^8,\Tr_3^*) & =
  \kbordermatrix{
    & \Tr_4^{\Tr_3^*,3} & W_4^{\Tr_3^*} & 
    \Tr_4^{\Tr_3^*,2} & R_4^{\Tr_3^*,1} &
    L_4^{\Tr_3^*} & \Tr_4^{\Tr_3^*,1} &
    R_4^{\Tr_3^*,2} & \Tr_4^{\Tr_3^*,0}
    \\
    & \frac{31}{40} & \frac{43}{40} & -\frac{61}{200} & \frac{79}{100} 
    & -\frac{241}{200} & \frac{1}{25} & -\frac{4}{5} & -\frac{37}{100}
    \\
    & \frac{43}{40} & \frac{1511}{200} & -\frac{131}{200} & \frac{331}{100} 
    & -\frac{1253}{200} & \frac{3}{4} & -\frac{113}{25} & -\frac{5}{4}
    \\
    & -\frac{61}{200} & -\frac{131}{200} & \frac{18}{25} & -\frac{29}{50} 
    & \frac{3}{4} & -\frac{1}{4} & \frac{7}{25} & \frac{1}{25}
    \\
    & \frac{79}{100} & \frac{331}{100} & -\frac{29}{50} & \frac{187}{50} 
    & -\frac{113}{25} & \frac{7}{25} & -\frac{223}{100} & -\frac{79}{100}
    \\
    & -\frac{241}{200} & -\frac{1253}{200} & \frac{3}{4} & -\frac{113}{25} 
    & \frac{15}{2} & -\frac{67}{100} & \frac{331}{100} & \frac{11}{10}
    \\
    & \frac{1}{25} & \frac{3}{4} & -\frac{1}{4} & \frac{7}{25} 
    & -\frac{67}{100} & \frac{67}{100} & -\frac{27}{50} & -\frac{7}{25}
    \\
    & -\frac{4}{5} & -\frac{113}{25} & \frac{7}{25} & -\frac{223}{100} 
    & \frac{331}{100} & -\frac{27}{50} & \frac{187}{50} & \frac{19}{25}
    \\
    & -\frac{37}{100} & -\frac{5}{4} & \frac{1}{25} & -\frac{79}{100} 
    & \frac{11}{10} & -\frac{7}{25} & \frac{19}{25} & \frac{79}{100}
  }.
\end{align*}

\begin{align*}
  Q(T_5^8,\vec C_3^*) & =
  \kbordermatrix{
    & R_4^{\vec C_3^*,3} & L_4^{\vec C_3^*} &
    R_4^{\vec C_3^*,2} & R_4^{\vec C_3^*,1} &
    R_4^{\vec C_3^*,23} & W_4^{\vec C_3^*} &
    R_4^{\vec C_3^*,12} & R_4^{\vec C_3^*,13}
    \\
    & \frac{391}{100} & -\frac{319}{100} & \frac{13}{20} & \frac{13}{20} 
    & -\frac{123}{50} & \frac{87}{50} & \frac{9}{50} & -\frac{37}{25}
    \\
    & -\frac{319}{100} & \frac{367}{50} & -\frac{159}{50} & -\frac{159}{50} 
    & \frac{7}{4} & -\frac{76}{25} & \frac{7}{4} & \frac{7}{4}
    \\
    & \frac{13}{20} & -\frac{159}{50} & \frac{389}{100} & \frac{13}{20} 
    & -\frac{37}{25} & \frac{7}{4} & -\frac{123}{50} & \frac{9}{50}
    \\
    & \frac{13}{20} & -\frac{159}{50} & \frac{13}{20} & \frac{389}{100} 
    & \frac{9}{50} & \frac{7}{4} & -\frac{37}{25} & -\frac{123}{50}
    \\
    & -\frac{123}{50} & \frac{7}{4} & -\frac{37}{25} & \frac{9}{50} 
    & \frac{389}{100} & -\frac{159}{50} & \frac{13}{20} & \frac{13}{20}
    \\
    & \frac{87}{50} & -\frac{76}{25} & \frac{7}{4} & \frac{7}{4} 
    & -\frac{159}{50} & \frac{367}{50} & -\frac{159}{50} & -\frac{159}{50}
    \\
    & \frac{9}{50} & \frac{7}{4} & -\frac{123}{50} & -\frac{37}{25} 
    & \frac{13}{20} & -\frac{159}{50} & \frac{389}{100} & \frac{13}{20}
    \\
    & -\frac{37}{25} & \frac{7}{4} & \frac{9}{50} & -\frac{123}{50} 
    & \frac{13}{20} & -\frac{159}{50} & \frac{13}{20} & \frac{389}{100}
  }.
\end{align*}


\begin{align*}
  Q(T_5^9,1) & =
  \frac{1}{40}\cdot
  \!\!\!\kbordermatrix{
    & \Tr_3^{1,L} & \vec C_3^1 & \Tr_3^{1,M} & \Tr_3^{1,W}\\
    & 75 & -33 & -117 & 75\\
    & -33 & 33 & 33 & -33\\
    & -117 & 33 & 201 & -117\\
    & 75 & -33 & -117 & 75
  }.
\end{align*}

\begin{align*}
  Q(T_5^9,\vec C_3^*) & =
  \frac{1}{5}\cdot
  \!\!\!\kbordermatrix{
    & R_4^{\vec C_3^*,3} & L_4^{\vec C_3^*} &
    R_4^{\vec C_3^*,2} & R_4^{\vec C_3^*,1} &
    R_4^{\vec C_3^*,23} & W_4^{\vec C_3^*} &
    R_4^{\vec C_3^*,12} & R_4^{\vec C_3^*,13}
    \\
    & 36 & 0 & -18 & -18 & -18 & 0 & -18 & 36\\
    & 0 & 192 & 0 & 0 & 0 & -192 & 0 & 0\\
    & -18 & 0 & 36 & -18 & 36 & 0 & -18 & -18\\
    & -18 & 0 & -18 & 36 & -18 & 0 & 36 & -18\\
    & -18 & 0 & 36 & -18 & 36 & 0 & -18 & -18\\
    & 0 & -192 & 0 & 0 & 0 & 192 & 0 & 0\\
    & -18 & 0 & -18 & 36 & -18 & 0 & 36 & -18\\
    & 36 & 0 & -18 & -18 & -18 & 0 & -18 & 36
  }.
\end{align*}


\begin{align*}
  Q(T_5^{11},1) & =
  \frac{5}{16}\cdot
  \!\!\!\kbordermatrix{
    & \Tr_3^{1,L} & \vec C_3^1 & \Tr_3^{1,M} & \Tr_3^{1,W}\\
    & 1 & -1 & -1 & 1\\
    & -1 & 1 & 1 & -1\\
    & -1 & 1 & 1 & -1\\
    & 1 & -1 & -1 & 1
  }
  =
  \frac{5}{16}\cdot
  \left(
  \begin{array}{c}
    1\\ -1\\ -1\\ 1\\
  \end{array}
  \right)
  \cdot
  \left(
  \begin{array}{c}
    1\\ -1\\ -1\\ 1\\
  \end{array}
  \right)^\top.
\end{align*}

\begin{align*}
  Q(T_5^{11},\vec C_3^*) & =
  \frac{1}{5}\cdot
  \kbordermatrix{
    & R_4^{\vec C_3^*,3} & L_4^{\vec C_3^*} &
    R_4^{\vec C_3^*,2} & R_4^{\vec C_3^*,1} &
    R_4^{\vec C_3^*,23} & W_4^{\vec C_3^*} &
    R_4^{\vec C_3^*,12} & R_4^{\vec C_3^*,13}
    \\
    & 24 & 12 & 6 & 6 & -6 & -12 & -6 & -24\\
    & 12 & 25 & 12 & 12 & -12 & -25 & -12 & -12\\
    & 6 & 12 & 27 & 6 & -27 & -12 & -6 & -6\\
    & 6 & 12 & 6 & 24 & -6 & -12 & -24 & -6\\
    & -6 & -12 & -27 & -6 & 27 & 12 & 6 & 6\\
    & -12 & -25 & -12 & -12 & 12 & 25 & 12 & 12\\
    & -6 & -12 & -6 & -24 & 6 & 12 & 24 & 6\\
    & -24 & -12 & -6 & -6 & 6 & 12 & 6 & 24
  }.
\end{align*}


\begin{align*}
  Q(T_5^{12},1) & =
  \frac{1}{16}\cdot
  \!\!\!\kbordermatrix{
    & \Tr_3^{1,L} & \vec C_3^1 & \Tr_3^{1,M} & \Tr_3^{1,W}\\
    & 1 & -1 & -1 & 1\\
    & -1 & 1 & 1 & -1\\
    & -1 & 1 & 1 & -1\\
    & 1 & -1 & -1 & 1
  }
  =
  \frac{1}{16}\cdot
  \left(
  \begin{array}{c}
    1\\ -1\\ -1\\ 1\\
  \end{array}
  \right)
  \cdot
  \left(
  \begin{array}{c}
    1\\ -1\\ -1\\ 1\\
  \end{array}
  \right)^\top.
\end{align*}

\begin{align*}
  Q(T_5^{12},\Tr_3^*) & =
  \frac{1}{2}\cdot
  \!\!\!\kbordermatrix{
    & \Tr_4^{\Tr_3^*,3} & W_4^{\Tr_3^*} & 
    \Tr_4^{\Tr_3^*,2} & R_4^{\Tr_3^*,1} &
    L_4^{\Tr_3^*} & \Tr_4^{\Tr_3^*,1} &
    R_4^{\Tr_3^*,2} & \Tr_4^{\Tr_3^*,0}
    \\
    & 3 & 2 & -3 & 3 & -2 & 3 & -3 & -3\\
    & 2 & 3 & -2 & 2 & -3 & 2 & -2 & -2\\
    & -3 & -2 & 3 & -3 & 2 & -3 & 3 & 3\\
    & 3 & 2 & -3 & 3 & -2 & 3 & -3 & -3\\
    & -2 & -3 & 2 & -2 & 3 & -2 & 2 & 2\\
    & 3 & 2 & -3 & 3 & -2 & 3 & -3 & -3\\
    & -3 & -2 & 3 & -3 & 2 & -3 & 3 & 3\\
    & -3 & -2 & 3 & -3 & 2 & -3 & 3 & 3
  }.
\end{align*}

\begin{align*}
  Q(T_5^{12},\vec C_3^*) & =
  \frac{1}{2}\cdot
  \!\!\!\kbordermatrix{
    & R_4^{\vec C_3^*,3} & L_4^{\vec C_3^*} &
    R_4^{\vec C_3^*,2} & R_4^{\vec C_3^*,1} &
    R_4^{\vec C_3^*,23} & W_4^{\vec C_3^*} &
    R_4^{\vec C_3^*,12} & R_4^{\vec C_3^*,13}
    \\
    & 7 & 5 & 7 & 7 & -7 & -5 & -7 & -7\\
    & 5 & 5 & 5 & 5 & -5 & -3 & -5 & -5\\
    & 7 & 5 & 7 & 7 & -7 & -5 & -7 & -7\\
    & 7 & 5 & 7 & 7 & -7 & -5 & -7 & -7\\
    & -7 & -5 & -7 & -7 & 7 & 5 & 7 & 7\\
    & -5 & -3 & -5 & -5 & 5 & 5 & 5 & 5\\
    & -7 & -5 & -7 & -7 & 7 & 5 & 7 & 7\\
    & -7 & -5 & -7 & -7 & 7 & 5 & 7 & 7
  }.
\end{align*}


\subsection{Characteristic polynomials of matrices used}\label{appendix:polchar}

\begin{align*}
  P_{Q(\tsete,1)}(x) & =
  x^4 - \frac{35}{12} x^3.
  \\
  P_{Q(\tsete,\Tr_3^*)}(x) & =
  x^8 - 30 x^7.
  \\
  P_{Q(\tsete,\vec C_3^*)}(x) & =
  x^8 - 24 x^7.
  \\
  P_{Q(\toito,1)}(x) & =
  x^4
  - \frac{55}{32} x^3
  + \frac{3450823}{10240000} x^2
  - \frac{255999851}{16384000000} x.
  \\
  P_{Q(\toito,A)}(x) & =
  x^4
  - \frac{99}{800} x^3.
  \\
  P_{Q(\toito,\Tr_3^*)}(x) & =
  x^8
  - \frac{2549}{100} x^7
  + \frac{675593}{5000} x^6
  - \frac{149230249}{500000} x^5
  + \frac{133434036319}{400000000} x^4
  \\ & \qquad {}
  - \frac{1980952353887}{10000000000} x^3
  + \frac{11839377144943}{200000000000} x^2
  - \frac{346051162035699}{50000000000000} x.
  \\
  P_{Q(\toito,\vec C_3^*)}(x) & =
  x^8
  - \frac{951}{25} x^7
  + \frac{5084929}{10000} x^6
  - \frac{159696453}{50000} x^5
  + \frac{125755799203}{12500000} x^4
  \\ & \qquad {}
  - \frac{1934738582639}{125000000}x^3
  + \frac{700918768199117}{62500000000} x^2
  - \frac{300346502258201}{97656250000} x.
  \\
  P_{Q(\tnove,1)}(x) & =
  x^4 - \frac{48}{5} x^3 + \frac{693}{100} x^2.
  \\
  P_{Q(\tnove,\vec C_3^*)}(x) & =
  x^8 - 120 x^7 + \frac{94608}{25} x^6 - \frac{4478976}{125} x^5.
  \\
  P_{Q(\tonze,1)}(x) & =
  x^4 - \frac{5}{4} x^3.
  \\
  P_{Q(\tonze,\vec C_3^*)}(x) & =
  x^8
  - 40 x^7
  + \frac{12828}{25} x^6
  - \frac{327744}{125} x^5
  + \frac{565056}{125} x^4.
  \\
  P_{Q(\tdoze,1)}(x) & =
  x^4 - \frac{1}{4} x^3.
  \\
  P_{Q(\tdoze,\Tr_3^*)}(x) & =
  x^8 - 12 x^7 + 15 x^6.
  \\ 
  P_{Q(\tdoze,\vec C_3^*)}(x) & =
  x^8 - 26 x^7 + 34 x^6 - 9 x^5.
\end{align*}


\end{document}